\documentclass{article}
\usepackage{ems-journal}
\usepackage{comment}
\usepackage{enumitem}
\usepackage{tikz}
\usepackage{tikz-cd}
\usepackage{overpic}
\usepackage[noabbrev,capitalize,nameinlink]{cleveref}
\usepackage{wrapfig}

\newtheorem{theorem}{Theorem}[section]
\newtheorem{proposition}[theorem]{Proposition}
\newtheorem{lemma}[theorem]{Lemma}
\newtheorem{corollary}[theorem]{Corollary}

\newtheorem{question}[theorem]{Question}
\newtheorem{problem}[theorem]{Problem}

\theoremstyle{definition}
\newtheorem{definition}[theorem]{Definition}
\newtheorem{example}[theorem]{Example}

\theoremstyle{remark}
\newtheorem{remark}[theorem]{Remark}
\numberwithin{equation}{section}

\newcommand{\norm}[1]{\left\lVert#1\right\rVert}
\usepackage{mleftright}

\newcommand{\ip}[2]{\left\langle#1, #2\right\rangle}

\newcommand{\R}{\mathbb{R}}
\newcommand{\C}{\mathbb{C}}
\newcommand{\D}{\mathbb{D}}

\newcommand{\ve}{\varepsilon}

\definecolor{darkgreen}{RGB}{0,153,0}
\definecolor{darkred}{RGB}{204,0,0}
\definecolor{darkblue}{RGB}{0,51,204}
\definecolor{lightblue}{RGB}{0,94,255}
\definecolor{red}{RGB}{242,43,29}
\definecolor{gold}{RGB}{240,203,0}
\hypersetup{colorlinks,linkcolor={darkblue},citecolor={darkgreen},urlcolor={darkgreen}}

\usepackage{tikz-cd}
\usetikzlibrary{decorations.pathmorphing}

\numberwithin{equation}{section}

\begin{document}

\title{Regularly slice implies once-stably decomposably slice}

\emsauthor{1}{
	\givenname{Joseph}
	\surname{Breen}
	\mrid{4741449}
	\orcid{0009-0001-7794-7206}}{J.~Breen}
\Emsaffil{1}{
	\department{Department of Mathematics}
	\organisation{University of Iowa}
	\rorid{036jqmy94}
	\address{2 West Washington Street}
	\zip{52242}
	\city{Iowa City}
	\country{USA}
	\affemail{joseph-breen-1@uiowa.edu}}

\begin{abstract}
We investigate the relationship between regular and decomposable Lagrangian cobordisms in $4$-dimensional symplectizations. First, we show that regular sliceness implies once-stably decomposable sliceness, and offer a stabilization-free strategy. On the other hand, we show that satelliting preserves regularity of concordance, suggesting that regularity and decomposability are distinct in general. Among other results, we compare the symplectic and smooth slice-ribbon conjectures and construct decomposably slice knots that may not be strongly decomposably slice.
\end{abstract}

\maketitle\thispagestyle{empty}
\tableofcontents

\section{Introduction}

There is a known hierarchy of Lagrangian cobordisms between Legendrian links in the symplectization of a contact $3$-manifold:
\begin{center}
\begin{tikzcd}
  \{\text{decomposable}\} \arrow[r,Rightarrow] 
    & \{\text{regular}\} \arrow[r,Rightarrow] &  \{\text{exact}\}. 
\end{tikzcd}    
\end{center}
These will be reviewed briefly in \cref{subsec:context} and thoroughly in \cref{sec:background}. 

A wealth of literature on exact Lagrangian cobordisms now exists, with particular attention given to decomposability. Less has been said about regularity, and in general the reversibility of the above implications is not fully understood. The broad goal of this paper is to establish a variety of results on regularity and to investigate its relationship with decomposability. 

\begin{remark}
The symplectic adjective \textit{decomposable} is the analogue of the smooth adjective \textit{ribbon}. For more context, we contrast the symplectic and smooth slice-ribbon conjectures in \cref{sec:gradations}; readers preferential to smooth topology may wish to start there.    
\end{remark}

\subsection{Exact, regular, and decomposable cobordisms}\label{subsec:context}

Let $\Lambda_-, \Lambda_+\subset (Y, \ker \alpha)$ be Legendrian links in a contact $3$-manifold. An \textit{exact Lagrangian cobordism} from $\Lambda_-$ to $\Lambda_+$ in the symplectization $(\R_s \times Y, e^s\, \alpha)$ is a surface $L\subset \R \times Y$, cylindrical over $\Lambda_{\pm}$ near $\{s=\pm \infty\}$, for which $(e^s\, \alpha)\mid_L$ is exact. The importance of such cobordisms was first made clear at the turn of the millennium in the context of Eliashberg, Givental, and Hofer's symplectic field theory \cite{EGHSFT}; indeed, there are strong Legendrian invariants in both SFT \cite{ekholm2012exactcobordisms} and Floer theory \cite{golla2019functoriality,baldwin2022lagrangian} which are functorial with respect to exact Lagrangian cobordisms. Existence and classification up to Hamiltonian isotopy, particularly of \emph{fillings} (cobordisms with $\Lambda_- = \emptyset$) and \textit{concordances} (cobordisms diffeomorphic to a cylinder with $\Lambda_{\pm}\neq \emptyset$), is an active area of research \cite{chantraine2010concordance,ekholm2012exactcobordisms,cornwell2016concordance,pan2016fillings,casals2022infinitely}. Classification is difficult. At the time of writing, the standard unknot in $(\R^3, \xi_{\mathrm{st}})$ with maximal Thurston-Bennequin invariant is the only knot for which a complete classification of fillings is known, dating back to work of Eliashberg and Polterovich \cite{eliashberg1996unknotclassification}; however, there is a conjectural classification for a large class of positive braids \cite{casals2022planecurves}. See \cite{blackwell2021lagrangians} for an additional survey. 

The class of \textit{regular} Lagrangians was introduced by Eliashberg, Ganatra, and Lazarev \cite{eliashberg2018flexiblelagrangians}. In our context, these are (necessarily exact) Lagrangian cobordisms $L$ for which there exists a Weinstein structure homotopic to the symplectization with Liouville vector field everywhere tangent to $L$. Regularity is a natural requirement in Weinstein topology, where regular disks correspond to co-cores of critical Weinstein handles, and more generally, the complement of a regular cobordism is Weinstein. Regularity plays a fundamental role in recent advances in high-dimensional contact topology \cite{HH18,honda2019convex}; for instance, in practice, open book stabilizations are performed along regular Lagrangian disks in the page \cite{BHH23}. The existence of non-regular exact Lagrangian \textit{caps} (cobordisms with $\Lambda_+ = \emptyset$) has been known  in dimension $4$ since the work of Lin \cite{lin2016caps}, but for a long time it was unknown if every exact cobordism with $\Lambda_+ \neq \emptyset$ was necessarily regular \cite[Problem 2.5]{eliashberg2018flexiblelagrangians}. Very recently, Dimitroglou Rizell and Golovko \cite{rizell2024instabilitylegendrianknottednessnonregular} exhibited non-regular concordances from sufficiently stabilized knots, showing that this is not the case. However, the question remains open for fillings.

\begin{figure}[t]
	\begin{overpic}[scale=.33]{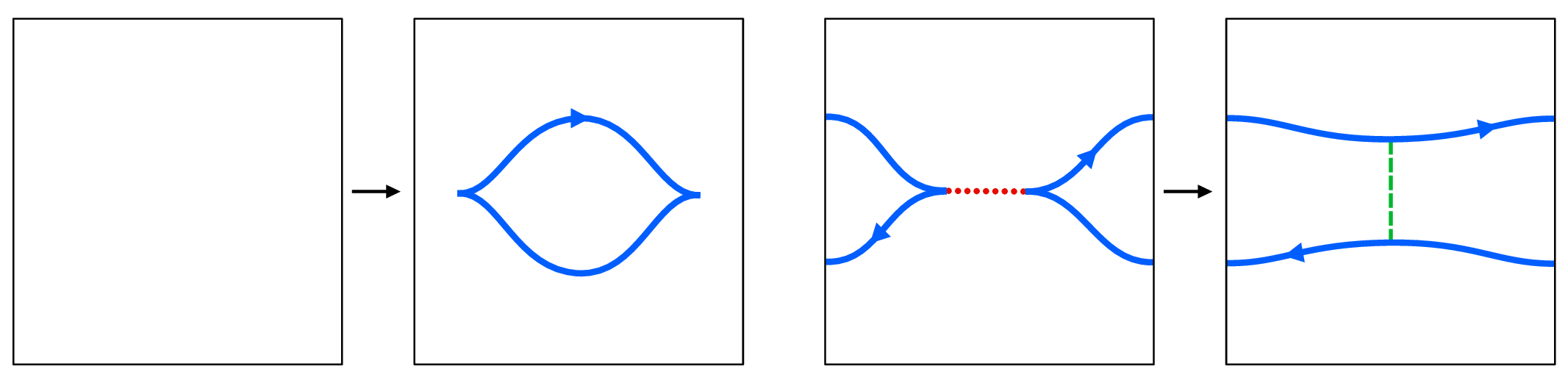}   
        
	\end{overpic}
	\caption{The decomposable moves, unknot birth and (ambient) Legendrian surgery.}
	\label{fig:pinchmove}
\end{figure}

A primary way to construct exact Lagrangian cobordisms is by concatenating elementary cobordisms which are traces of three moves in the front projection: Legendrian isotopy, isolated standard unknot birth, and Legendrian surgery (called a \textit{pinch move} in reverse); see \cref{fig:pinchmove}. The latter two induce Lagrangian $0$-handle and $1$-handle attachments. Such cobordisms are called \textit{decomposable} and were first considered by Chantraine \cite{chantraine2012noncollarable} and Ekholm, Honda, and Kálmán \cite{ekholm2012exactcobordisms}, while the surgery theory was developed more generally by Dimitroglou Rizell \cite{dimitroglourizell2016ambient}. Conway, Etnyre, and Tosun \cite{conway2017symplectic} extended decomposability to symplectizations of any contact $3$-manifold, and moreover showed (concurrently with Golla and Juhász \cite{golla2019functoriality}) that decomposable Lagrangian cobordisms in symplectizations are regular. Whether or not regularity implies decomposability is completely open.

The difference between regularity and decomposability is subtle. For instance, when $\Lambda_+ \neq 0$, standard Weinstein homotopy techniques (e.g. \cref{lemma:cot_bund_str}) imply that a regular cobordism is built out of Lagrangian handles of index of at most $1$ with respect to a certain Liouville vector field. However, decomposability requires such a handle decomposition induced specifically by the symplectization. This is analogous to the difference between \textit{ribbon} and \textit{handle-ribbon} in the smooth category; see \cref{sec:gradations} for additional discussion on the matter.    

Rather than asking whether a given cobordism is regular or decomposable, one can ask weaker (but still hard) questions about the induced relations on Legendrian knot types. Regarding the relationship between regularity and decomposability, the motivating open questions of this article are: 

\begin{question}\label{q:3}
If $\Lambda_-$ and $\Lambda_+$ are regularly cobordant, are they decomposably cobordant? 
\end{question}

\begin{question}\label{q:4}
In particular, is every regularly slice knot decomposably slice?
\end{question}

\noindent Our results serve as evidence for a negative answer to \cref{q:3} (see \cref{cor:neg_ans}), but for a positive answer to \cref{q:4} (see \cref{conj:slice} and its supporting evidence \cref{thm:main3}, \cref{remark:dec_fil_norm}, \cref{prop:reg_smth_ribbon}, and the discussion in \cref{sec:nec}). This potential disparity suggests that the question of regularity vs. decomposability is subtle and interesting. Moreover, the takeaway of \cref{sec:gradations} is that any subtlety is symplectic in nature, as opposed to being symptomatic of smooth $4$-manifold phenomena.

\subsection{Main results} We now specialize to the symplectization of $(S^3, \xi_{\mathrm{st}}).$ Our main results are \cref{thm:main3}, \cref{thm:cob_diag}, \cref{thm:main_sat}, \cref{thm:main_ruling}, and their associated corollaries. Discussion around the four theorems is organized into four corresponding subsections below.

\subsubsection{Regularly slice implies once-stably decomposably slice} 

First, we provide a stable affirmative answer to \cref{q:4} in a sense made precise by \cref{thm:main3}. Let $U$ denote the max-tb unknot, let $S_{\pm}(\Lambda)$ denote a single positive (resp. negative) stabilization of an oriented Legendrian knot $\Lambda$, and let $\Lambda_- \prec_{\ast} \Lambda_+$ where $\ast \in\{ \mathrm{reg}, \mathrm{dec}\}$ denote the existence of a regular (resp. decomposable) concordance from $\Lambda_-$ to $\Lambda_+$. 

\begin{theorem}\label{thm:main3}
Every regularly slice Legendrian knot is once-stably (strongly) decomposably slice. That is, if $U \prec_{\mathrm{reg}} \Lambda$, then $S_{\pm}(U) \prec_{\mathrm{dec}} S_{\pm}(\Lambda)$.
\end{theorem}

\noindent Briefly, a decomposable cobordism is \textit{strongly decomposable} if all Legendrian surgeries commute; for more information, see \cref{subsec:sympelctic_grads} and \cref{subsec:decvsstrong}.

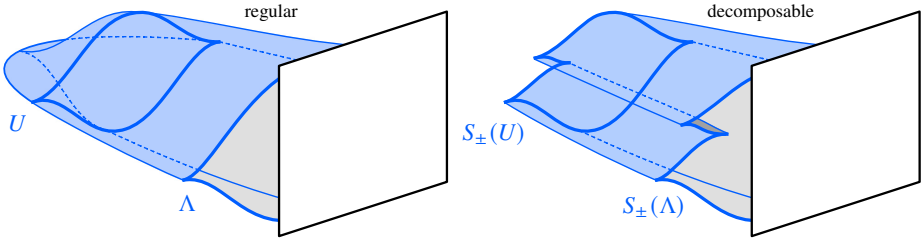
\begin{figure}[ht]
    \vskip-2cm
	\begin{overpic}[scale=.345]{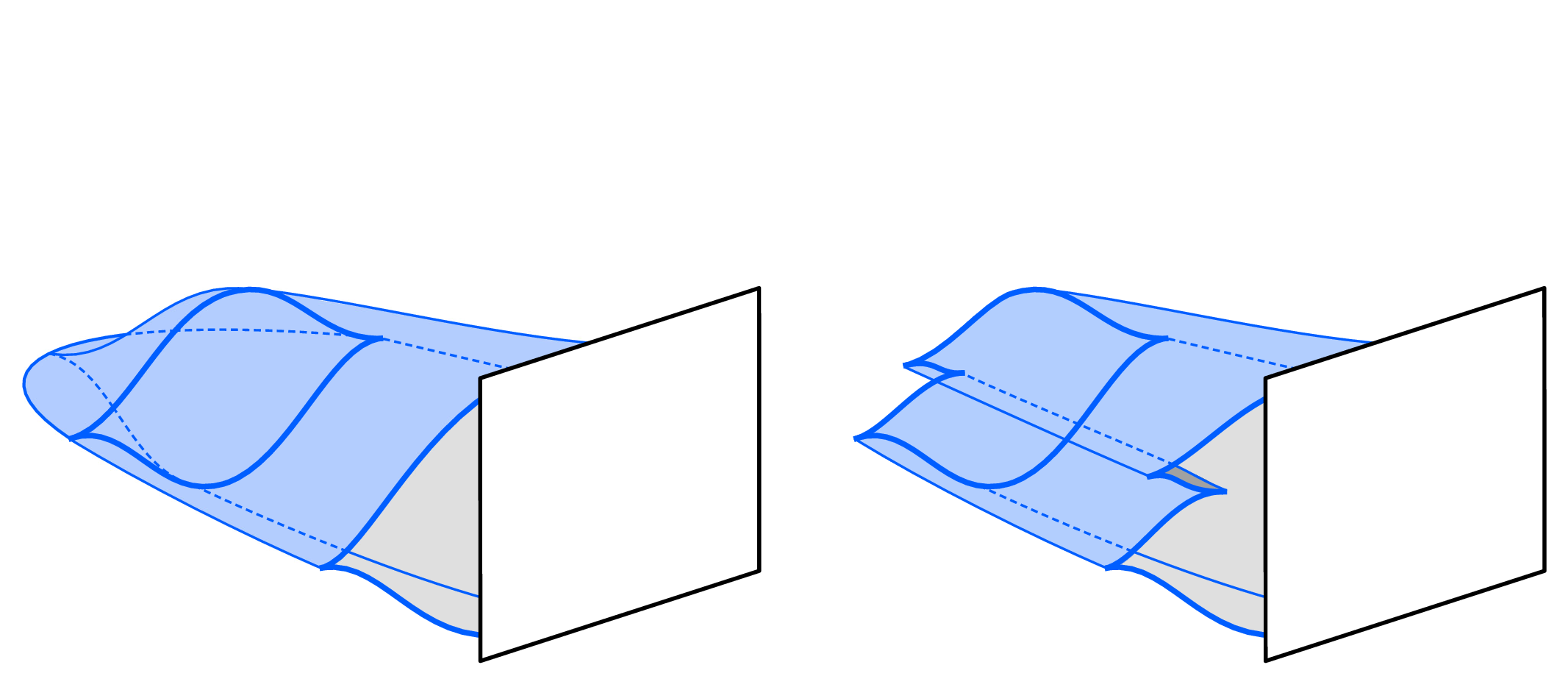}   
        \put(2,13){\small \textcolor{lightblue}{$U$}}
        \put(20,4.5){\small \textcolor{lightblue}{$\Lambda$}}
        \put(50,12){\small \textcolor{lightblue}{$S_{\pm}(U)$}}
        \put(67,4.25){\small \textcolor{lightblue}{$S_{\pm}(\Lambda)$}}
        \put(27,25){\scriptsize regular}
        \put(76,25){\scriptsize decomposable}
	\end{overpic}
	\caption{A schematic picture of \cref{thm:main3}. On the left is a regular slice disk, given as a concatenation of the standard filling of $U$ with a regular concordance. On the right is a decomposable concordance between the once-stabilized knots. The box is a placeholder for a Legendrian tangle.}
	\label{fig:main3}
\end{figure}

To contrast recent results, Dimitroglou Rizell and Golovko \cite{rizell2024instabilitylegendrianknottednessnonregular} produce non-regular concordances between sufficiently stabilized Legendrian knots, while Etnyre and Leverson \cite{etnyre2024lagrangian} approximate smooth ribbon cobordisms with decomposable cobordisms between sufficiently stabilized Legendrian links; see \cref{remark:ELremark} for additional discussion on the latter. Our result assumes regularity of concordance and produces a decomposable concordance after only one stabilization. It seems plausible that the single stabilization is unnecessary, though the proof used for \cref{thm:main3} does not go through. Nevertheless, we record the following two problems, discuss an alternate approach in \cref{sec:nec}, and provide additional context in \cref{subsec:sat} and \cref{subsec:rulings}. 

\begin{problem}\label{conj:slice}
Show that every regularly slice Legendrian knot is decomposably slice.    
\end{problem}

\begin{problem}\label{con:fill}
More generally, show that every regularly fillable Legendrian link is decomposably fillable.    
\end{problem}

\subsubsection{Diagrammatic presentation of regular cobordisms}\label{subsec:diag}

\cref{thm:main3} begins from the observation of Conway, Etnyre, and Tosun \cite[Theorem 1.10]{conway2017symplectic} that every regular Lagrangian disk filling admits the following presentation. First, let $U$ be a max-tb unknot with its standard Lagrangian disk filling. Then attach Weinstein $1$-handles and $2$-handles along attaching spheres that avoid $U$, so that the resulting Weinstein structure is homotopic to the symplectization, i.e., up to Legendrian Kirby calculus, the handles may be canceled. In the surgered contact manifold, $U$ is not necessarily the unknot; one identifies the knot $\Lambda_+$ being filled in the symplectization by sliding $U$ off of the $1$-handles to perform cancellations. See \cref{fig:m946ex} for an example. 

\begin{figure}[ht]
	\begin{overpic}[scale=.346]{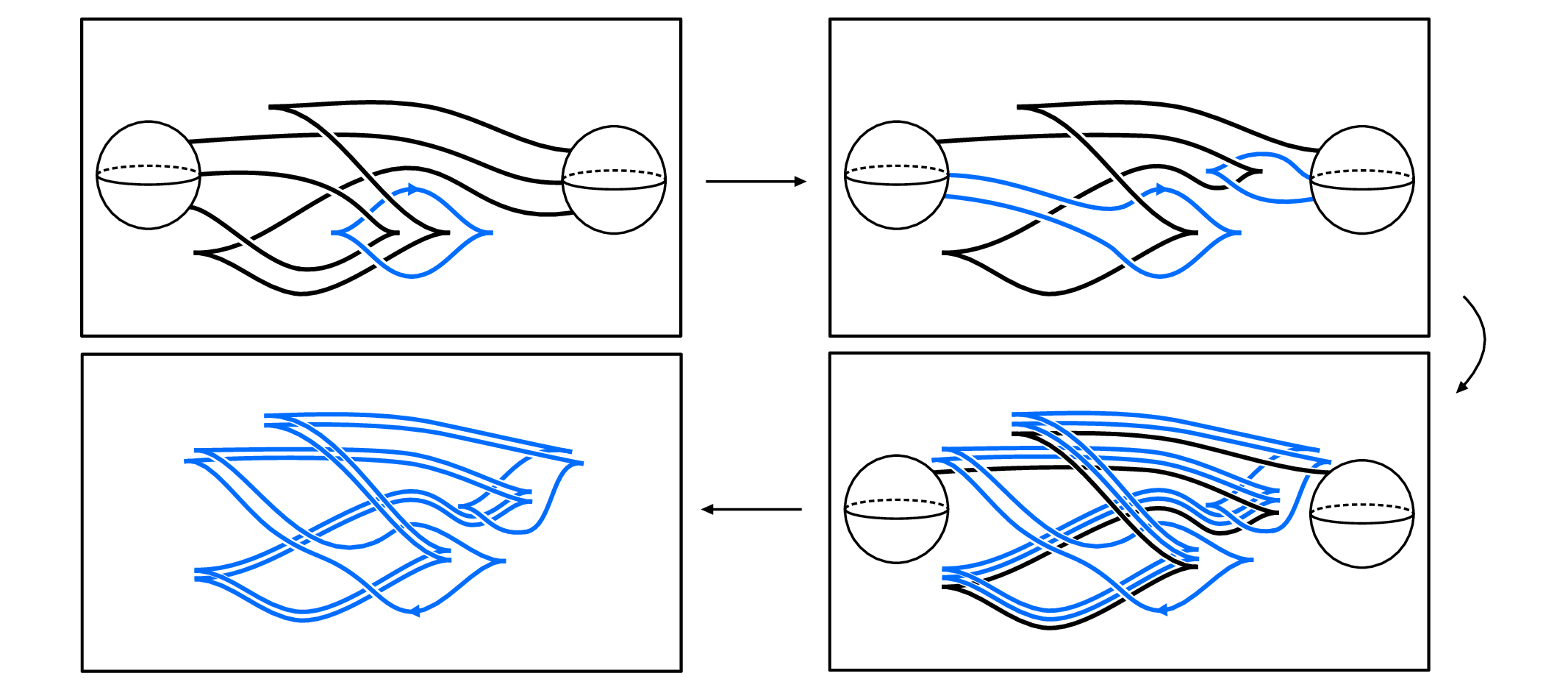}   
    \put(25,39){\scriptsize $(-1)$}
    \put(28,25){\scriptsize \textcolor{lightblue}{$U$}}
    \put(29,4){\scriptsize \textcolor{lightblue}{$\Lambda_+$}}
    \put(74,39){\scriptsize $(-1)$}
    \put(76,18){\scriptsize $(-1)$}
	\end{overpic}
	\caption{A regular disk filling of $\Lambda_+$, a Legendrian $\overline{9_{46}}$.}
	\label{fig:m946ex}
\end{figure}
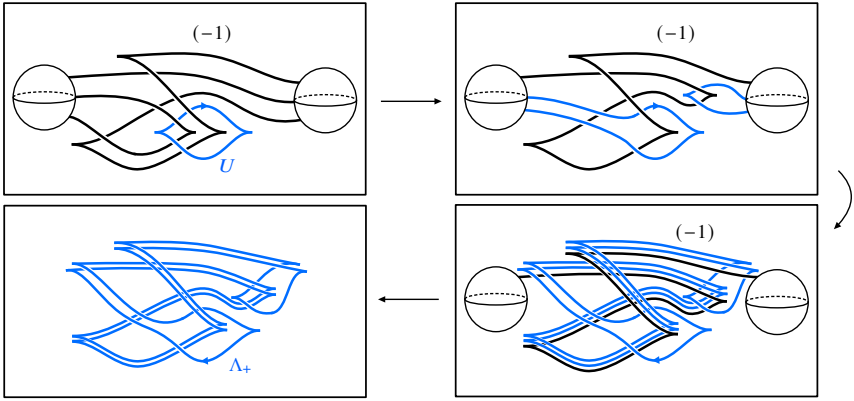

This presentation is a diagrammatic interpretation of \cite[Proposition 2.3]{eliashberg2018flexiblelagrangians}. Our next result generalizes this to oriented cobordisms in general. We provide additional specifications for fillings of knots and concordances. 

\begin{remark}
The recent work of Dimitroglou Rizell and Golovko uses the same diagrammatic interpretation of \cite{eliashberg2018flexiblelagrangians}.    
\end{remark}

\begin{theorem}[Regular cobordism diagrams]\label{thm:cob_diag}
Let $L \subset \R \times S^3$ be a regular Lagrangian cobordism in the symplectization of $(S^3, \xi_{\mathrm{st}})$ from $\Lambda_-\neq \emptyset$ to $\Lambda_+\neq \emptyset$. There is a Weinstein handlebody diagram with ($n + k$)-many $1$-handles and a link $\Lambda_0\in (\#^{n+k} S^1\times S^2, \xi_{\mathrm{st}})$ satisfying the following properties. 
\begin{enumerate}
    \item The Weinstein structure is homotopic to the symplectization. 
    \item The link $\Lambda_0$ has geometric intersection number $2$ and algebraic intersection number $0$ with $n$-many $1$-handle belt spheres, and does not cross the other $1$-handles. 
    \item Before attaching $2$-handles, performing pinch moves on $\Lambda_0$ through the $1$-handles produces $\Lambda_-$. All $2$-handle attaching spheres are disjoint from $\Lambda_0$, and canceling the Weinstein handles transforms $\Lambda_0$ into $\Lambda_+$.
\end{enumerate}
Conversely, any cobordism constructed in this way is regular.
\end{theorem}

\begin{figure}[ht]
	\begin{overpic}[scale=.3475]{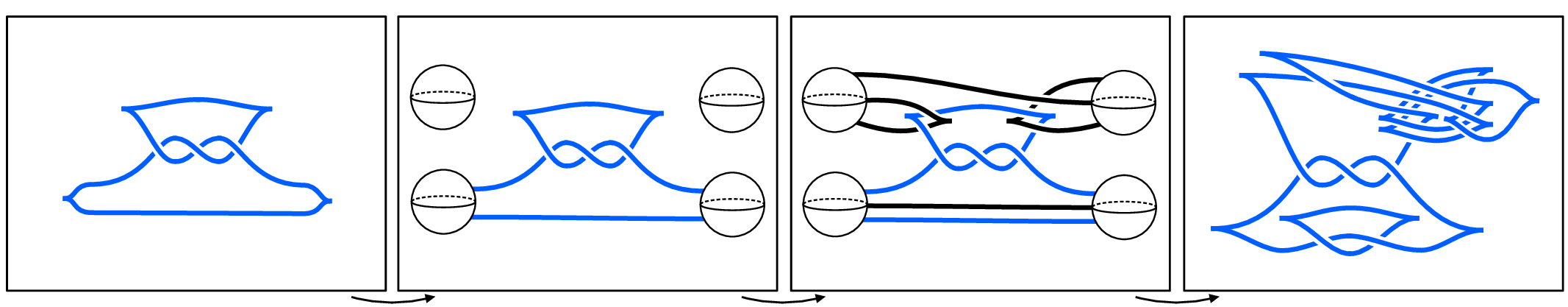}   
        \put(11.5,3){\scriptsize \textcolor{lightblue}{$\Lambda_-$}}
        \put(36.5,3){\scriptsize \textcolor{lightblue}{$\Lambda_0$}}
         \put(94.5,7){\scriptsize \textcolor{lightblue}{$\Lambda_+$}}
        \put(61,15){\tiny $(-1)$}
        \put(61,7.1){\tiny $(-1)$}
	\end{overpic}
	\caption{A regular (saddle) cobordism as described by \cref{thm:cob_diag} from the trefoil $\Lambda_-$ to some two-component Legendrian link $\Lambda_+$. Here $n=k=1$.}
	\label{fig:cobdiag}
\end{figure}
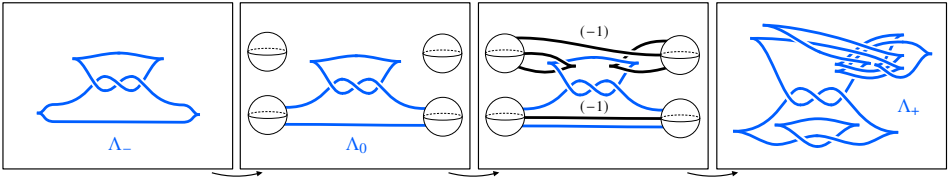

%See \cref{fig:cobdiag} for an example. 

\begin{corollary}[Regular filling diagrams]\label{thm:fill_diag}
Let $L \subset \R \times S^3$ be a regular genus $g\geq 1$ Lagrangian filling of a Legendrian knot $\Lambda$. There is a Weinstein handlebody diagram with ($2g+k$)-many $1$-handles, homotopic to the symplectization, such that: 
\begin{enumerate}
    \item After attaching $2g$-many $1$-handles, $\Lambda \in (\#^{2g} S^1 \times S^2, \xi_{\mathrm{st}})$ is the \emph{Gompf cotangent knot} \cite{gompf1998handlebody} in \cref{fig:filldiag} and $L$ is its standard filling.  
    \item The $k$ additional $1$-handles and all $2$-handles are disjoint from $\Lambda$. 
\end{enumerate}
Conversely, any filling constructed in this way is regular. 
\end{corollary}

\begin{corollary}[Regular concordance diagrams]\label{thm:conc_diag}
Let $L \subset \R \times S^3$ be a regular Lagrangian concordance from $\Lambda_-$ to $\Lambda_+$. There is a Weinstein handlebody diagram homotopic to the symplectization with attaching locus disjoint from $\Lambda_-$ such that $L$ is cylindrical over $\Lambda_-$. Conversely, any concordance constructed in this way is regular.
\end{corollary}

%The right side of \cref{fig:filldiag} gives an example concordance with $\Lambda_-$ the max-tb trefoil. 

\begin{figure}[ht]
	\begin{overpic}[scale=.3475]{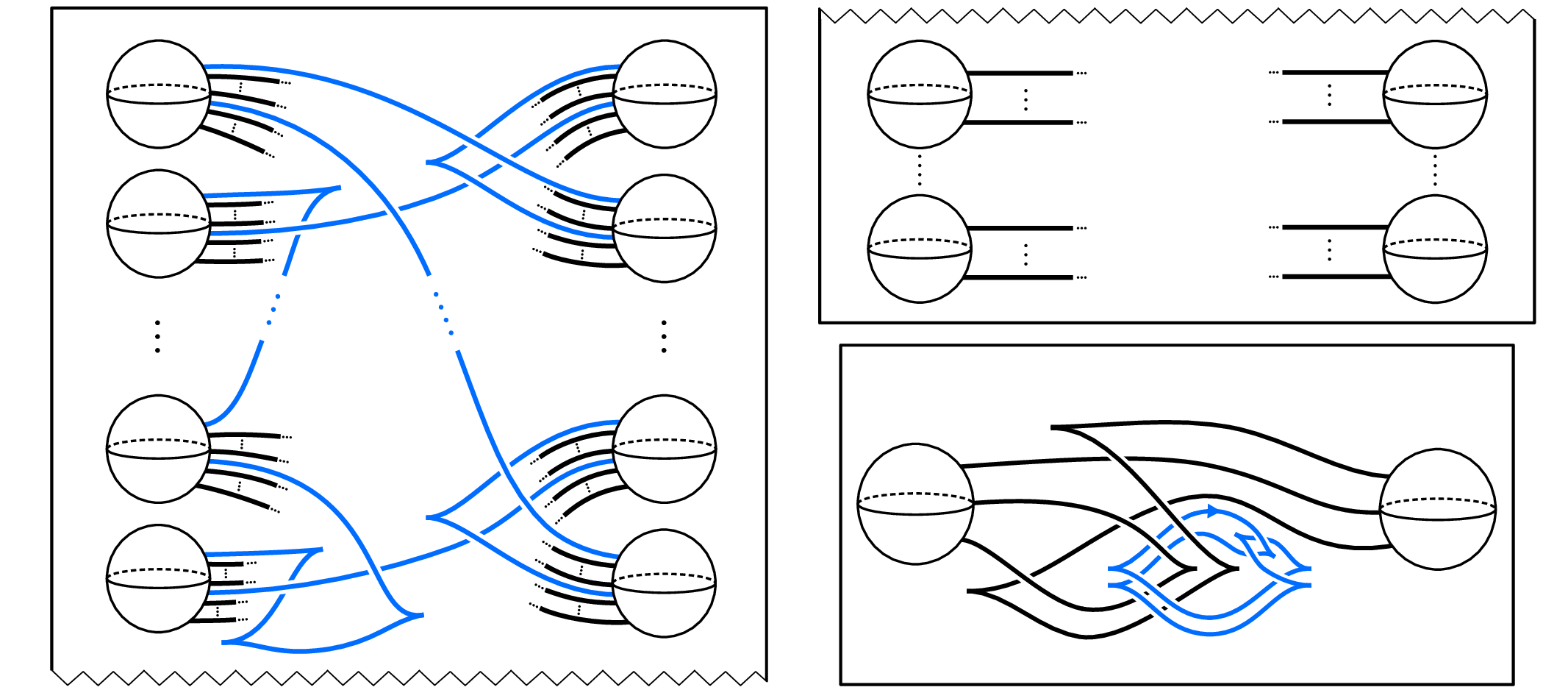}   
    \put(67.5,41){\tiny $(-1)$ on all black strands}
    \put(18,41.5){\tiny $(-1)$ on all black strands}
    \put(33,19){\small \textcolor{lightblue}{$\Lambda$}}
    \put(-0.5,37){\scriptsize $1$}
    \put(99,37){\scriptsize $2g+1$}
    \put(-0.5,29){\scriptsize $2$}
    \put(99,29){\scriptsize $2g+k$}
    \put(-3,14.5){\scriptsize $2g-1$}
    \put(-0.75,6.5){\scriptsize $2g$}
    
    \put(82,17.5){\scriptsize $(-1)$}
    
	\end{overpic}
	\caption{The diagrammatic presentation of a genus $g\geq 1$ filling on the left and continuing on the top right; an example of a regular concordance on the lower right. In fact, the concordance is from the trefoil to the Whitehead double of $\overline{9_{46}}$ as considered by \cite{cornwell2016concordance}; see the discussion in \cref{subsec:sat}.}
	\label{fig:filldiag}
\end{figure}
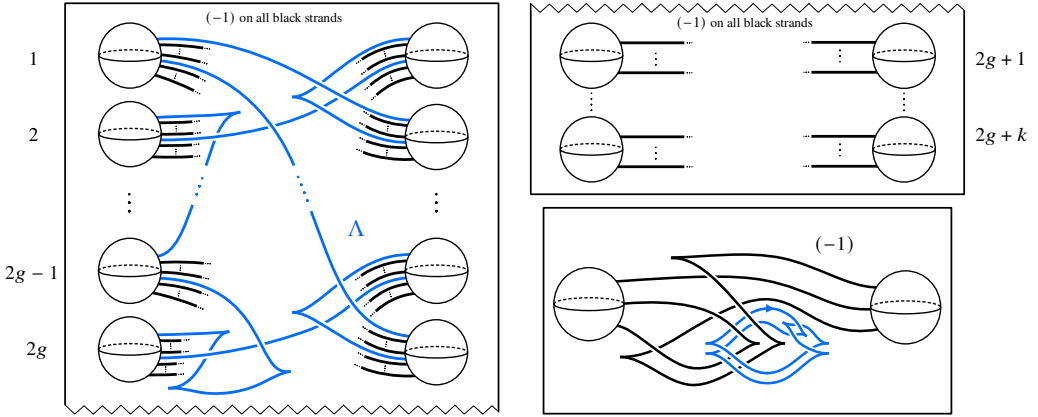

\subsubsection{Regularity and satelliting}\label{subsec:sat}

One disadvantage of the decomposable class is that it is unclear when it is preserved under geometric constructions such as satelliting. In fact, via a satellite operation, Cornwell, Ng, and Sivek constructed a concordance from the trefoil to the Whitehead double of $\overline{9_{46}}$ and conjectured that it is not decomposable \cite[Conjecture 3.4]{cornwell2016concordance}; see \cref{fig:satellite}. Our next result is a corollary of the above diagrammatic presentations and shows that regularity of concordance is preserved under satelliting.

\begin{figure}[ht]
	\begin{overpic}[scale=.3465]{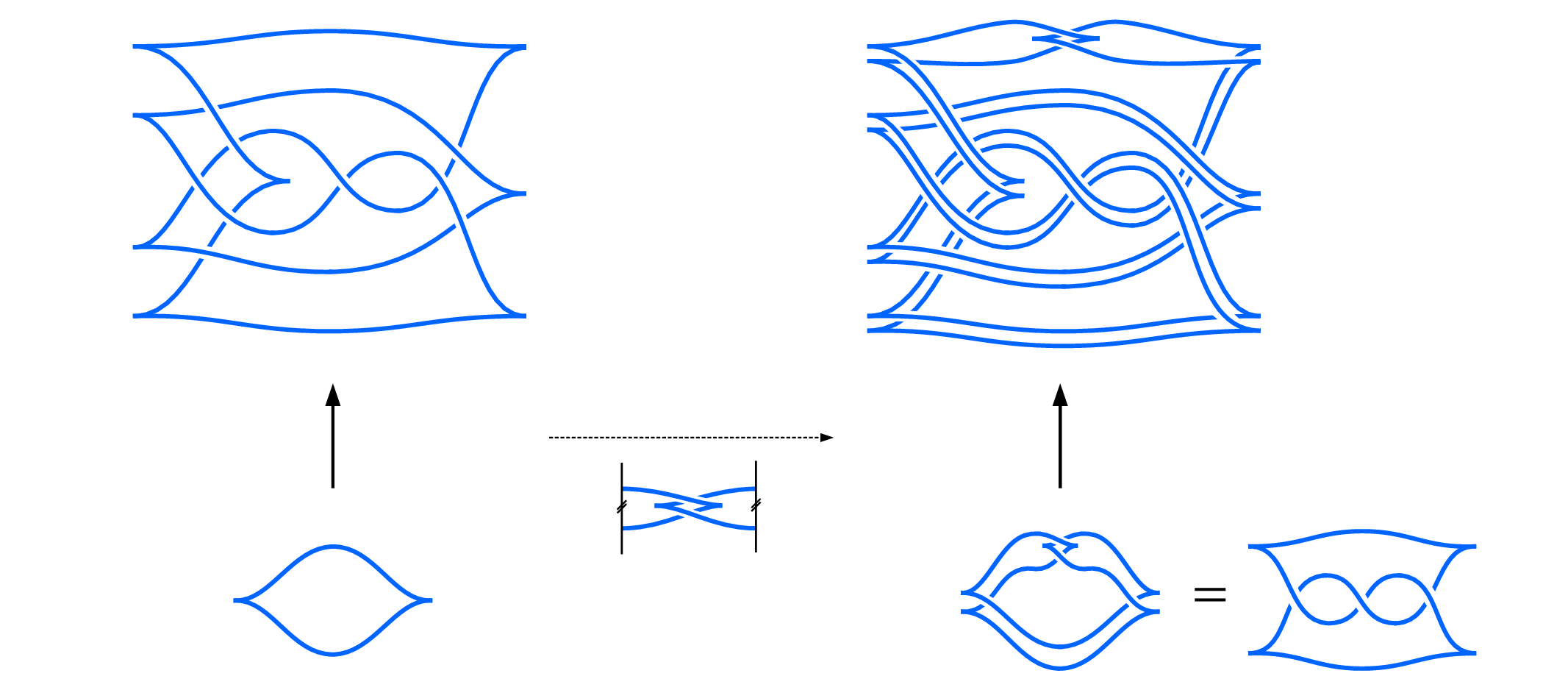}   
        \put(40.25,17){\footnotesize satellite}
        
	\end{overpic}
	\caption{Satelliting the concordance $U \prec_{\mathrm{reg}} \overline{9_{46}}$ on the left via the Whitehead pattern. By \cref{thm:main_sat}, the concordance on the right --- the subject of \cite[Conjecture 3.4]{cornwell2016concordance} --- is regular.}
	\label{fig:satellite}
\end{figure}

Before a precise statement, we briefly recall the Legendrian satellite operation. Let $\Lambda \subset (Y, \xi)$ be a Legendrian \textit{companion} link, and let $\Lambda' \subset J^1(S^1)$ be any Legendrian \textit{pattern} link in the standard solid torus. Let $\psi:J^1(S^1) \to N(\Lambda)$ be a contactomorphism onto a standard neighborhood of $\Lambda$, and let $S(\Lambda, \Lambda'):=\psi(\Lambda')$. Cornwell, Ng, and Sivek showed that if $\Lambda_- \prec \Lambda_+$, then $S(\Lambda_-, \Lambda')\prec S(\Lambda_+, \Lambda')$.

\begin{theorem}\label{thm:main_sat}
Suppose that $\Lambda_- \prec_{\mathrm{reg}} \Lambda_+$. Let $\Lambda'\subset J^1(S^1)$ be a Legendrian pattern link. Then $S(\Lambda_-, \Lambda') \prec_{\mathrm{reg}} S(\Lambda_+, \Lambda')$.     
\end{theorem}

\begin{corollary}\label{cor:neg_ans}
There is a regular concordance from the trefoil to the Whitehead double of $\overline{9_{46}}$. Consequently, the conjectural non-decomposably concordant relation of these knots \cite[Conjecture 3.4]{cornwell2016concordance} implies a negative answer to \cref{q:3}.    
\end{corollary}

In \cite[Theorem 1.4]{guadagni2022satellites}, Guadagni, Sabloff, and Yacavone generalized the operation of \cite{cornwell2016concordance} to satellites of Lagrangian cobordisms, with additional twists of the pattern link to account for nonzero genus. They provided a second and distinct construction that is decomposability-preserving, with the drawback that it applies only to a subclass of decomposable cobordisms. Following \cite{cornwell2016concordance}, it is unlikely that their first construction preserves decomposability in general. However, given \cref{thm:main_sat}, the following question is natural. 

\begin{question}
Does the cobordism satelliting operation of \cite[Theorem 1.4]{guadagni2022satellites} preserve regularity of cobordance?     
\end{question}

%\begin{remark}
We use the decomposability-preserving satellite operation of \cite[Theorem 1.6]{guadagni2022satellites} in \cref{sec:nec} to construct candidate decomposably slice knots that may not be strongly decomposably slice; see \cref{prop:slicesats} and \cref{q:satstrong}. 
%\end{remark}

Etnyre and Leverson \cite{etnyre2024lagrangian} recently defined a notion of \textit{stabilization} of Lagrangian cobordisms. As Legendrian stabilization is a special case of the satellite operation --- the pattern link being a stabilized $S^1$-core --- \cref{thm:cob_diag} allows us to prove: 

\begin{corollary}\label{cor:EL_stab}
The stabilization operation of Etnyre and Leverson \cite{etnyre2024lagrangian} preserves regularity of Lagrangian cobordism.    
\end{corollary}

\subsubsection{Regularity and normal rulings}\label{subsec:rulings}

A \textit{normal ruling} of a Legendrian front is a pairing of left and right cusps with colored paths that enclose, after resolving certain \textit{(normally) switched crossings}, standard max-tb planar disks; see \cref{fig:ruling1}. We defer a precise review to \cref{sec:background}.

First considered by Eliashberg \cite{eliashberg1987wave}, rulings of fronts and their combinatorial applications to Legendrian links were fleshed out by Fuchs \cite{fuchs2003rulings} and Chekanov and Pushkar \cite{chekanov2007pushkar}. Counts of normal rulings are a Legendrian isotopy invariant, and existence of a normal ruling is connected with existence of augmentations of the Chekanov-Eliashberg DGA \cite{fuchs2003rulings,fuchs2004invariants,sabloff2005augmentations,leverson2014augmentations}. 

\begin{figure}[ht]
	\begin{overpic}[scale=.25]{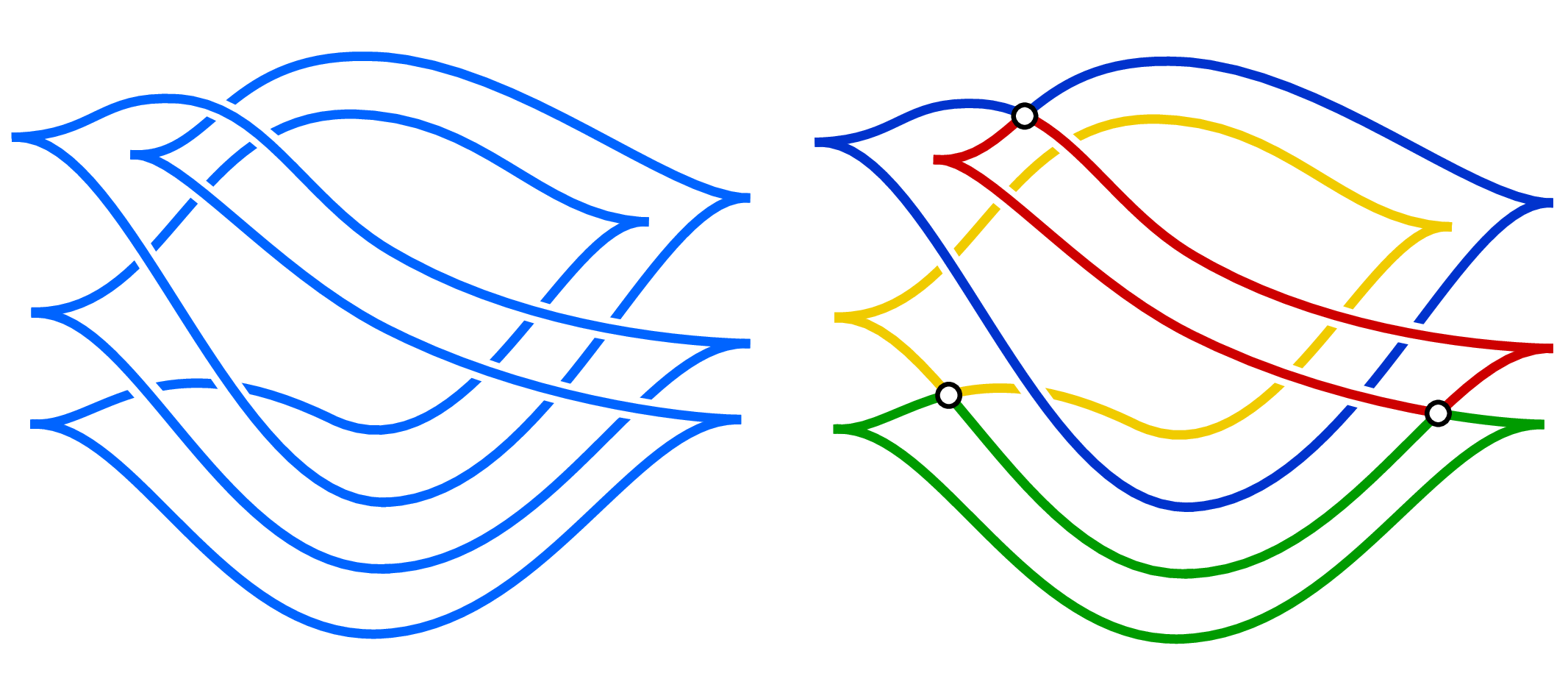}   
        
	\end{overpic}
	\caption{A front projection of the max-tb unknot (left) and its unique normal ruling (right).}
	\label{fig:ruling1}
\end{figure}

An immediate consequence of the local models for decomposable moves is the fact that if $\Lambda_-$ admits a normal ruling and $L$ is a decomposable cobordism from $\Lambda_-$ to $\Lambda_+$, then $\Lambda_+$ inherits a normal ruling from that of $\Lambda_-$ induced by the cobordism. This fact is used in \cite{cornwell2016concordance}, for example, to obstruct the existence of decomposable cobordisms. Our next result extends this to regular cobordisms.

\begin{theorem}\label{thm:main_ruling}
Let $L\subset \R \times S^3$ be a regular cobordism in the symplectization from $\Lambda_-$ to $\Lambda_+$. Suppose that $\Lambda_-$ admits a normal ruling. Then $L$ induces a normal ruling of $\Lambda_+$.     
\end{theorem}

\begin{corollary}\label{cor:normal_ruling_filling}
Counts of normal rulings are nondecreasing under the regular Lagrangian cobordism relation.    
\end{corollary}

\begin{remark}\label{remark:dec_fil_norm}
Call a normal ruling of $\Lambda$ \textit{decomposably fillable} if it is the canonical normal ruling of $\Lambda$ associated to a decomposable filling \cite{chen2024nonorientable}. From top to bottom, a normal ruling is decomposably fillable if paired colors can be successively pinched down to a max-tb unlink. There are some known necessary conditions for a normal ruling to be decomposably fillable \cite{chen2024nonorientable}, but sufficient conditions are unknown. \cref{conj:slice} and \cref{con:fill} can be rephrased by asking if the normal ruling furnished by applying \cref{thm:main_ruling} to a regular filling is decomposably fillable.   
\end{remark}

\subsection{Organization}

In \cref{sec:gradations} we discuss various gradations of the symplectic and smooth slice-ribbon conjecture. In \cref{sec:background} we provide the relevant background on exact, regular, and decomposable Lagrangian cobordisms, Weinstein topology, and normal rulings. In \cref{sec:slice} we prove \cref{thm:main3} regarding once-stable decomposable sliceness, in \cref{sec:diag} the diagrammatic results stated in \cref{subsec:diag}, and in \cref{sec:sat} the statements on satellites and normal rulings from \cref{subsec:sat} and \cref{subsec:rulings}. Finally, in \cref{sec:nec} we discuss a relevant example toward a solution to \cref{conj:slice} and construct a class of decomposably slice knots that may not be strongly decomposably slice.

\cref{sec:gradations} (on gradations of the slice-ribbon conjecture) and \cref{sec:nec} (on regular, decomposable, and strongly decomposable sliceness) may, for the most part, be read independently of the rest of the paper.

\begin{conventions}
Legendrians bounding cobordisms are drawn in blue, and never represent loci for surgeries or $2$-handle attaching spheres; the latter are typically black and always have explicit surgery coefficients. Smooth framing coefficients in Kirby diagrams are typeset without parentheses, while contact framing coefficients are enclosed in parentheses: $(\pm 1) = \mathrm{tb} \pm 1$. Legendrian surgery arcs are dotted and red.
\end{conventions}

\begin{ack}
We thank John Etnyre and Josh Sabloff for symplectic conversations, correspondence, and comments, and Alex Zupan for an insightful smooth discussion. Georgios Dimitroglou Rizell also provided helpful remarks and questions via email correspondence. 
\end{ack}

\begin{funding}
This work was partially supported by NSF Grant DMS-2038103 and an AMS-Simons Travel Grant.
\end{funding}

\section{Gradations of the slice-ribbon conjecture}\label{sec:gradations}

The goal of this section is to give additional context to our main results by discussing gradations of the slice-ribbon conjecture in the smooth and symplectic categories. Specifically, we explain the diagram in \cref{fig:SRgradations}. Such an explicit comparison appears lacking from the literature. 

\begin{figure}[ht]
    \centering
    \begin{tikzcd}
    |[yshift=-0.5em,overlay]|\text{\underline{Symplectic category}} & |[yshift=-0.5em,overlay]|\text{\underline{Smooth category}}\\
 \{\text{Lagrangian slice}\}\arrow[dddddr,darkred,bend right = 0,Rightarrow,in=220,out=50]\arrow[r,leftrightsquigarrow,crossing over]& \{\text{slice}\}\\
\{\text{Lagrangian homotopy-ribbon}\} \arrow[r,leftrightsquigarrow,crossing over]\arrow[u,darkgreen,Leftrightarrow] & \{\text{homotopy-ribbon}\}\arrow[u,Rightarrow]\\
 \{\text{weakly regular}\} \arrow[r,leftrightsquigarrow,crossing over]\arrow[u,Rightarrow] & \{\text{handle-ribbon}\}\arrow[u,Rightarrow]\\
 \{\text{regular}\} \arrow[r,leftrightsquigarrow,crossing over]\arrow[u,darkgreen,Leftrightarrow] & \{\text{AC-trivial handle-ribbon}\}\arrow[u,Rightarrow]\\
\{\text{decomposable}\} \arrow[r,leftrightsquigarrow,crossing over]\arrow[u,Rightarrow] & \{\text{ribbon}\}\arrow[u,darkgreen,Leftrightarrow] \\
\{\text{strongly decomposable}\} \arrow[r,leftrightsquigarrow]\arrow[u,Rightarrow] & \{\text{strongly ribbon}\}\arrow[u,darkgreen,Leftrightarrow]
\end{tikzcd}    
    \caption{Classes of Legendrian knots on the left, and their smooth analogues on the right. Reversibility of any of the black unidirectional implications in the diagram is unknown.}
    \label{fig:SRgradations}
\end{figure}
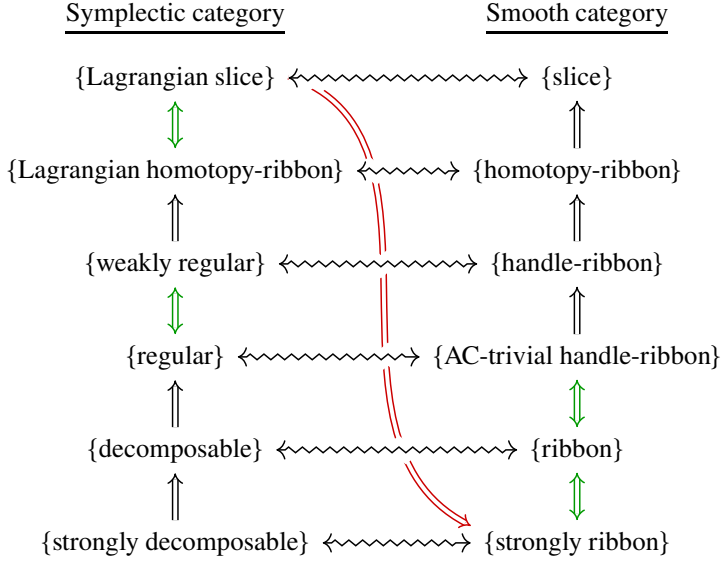

In the diagram are a number of novel terms: \textit{Lagrangian homotopy-ribbon}, 
\textit{weakly regular}, \textit{AC-trivial handle-ribbon}, \textit{strongly decomposable}, and \textit{strongly ribbon}. These are defined below only to put the symplectic and smooth gradations in one-to-one correspondence, and to clarify that the lower rungs of smooth slice-ribbon and symplectic slice-ribbon are difficult for different reasons. 

\begin{remark}
In the smooth category we consider knots bounding disks in $B^4$. However, by removing a small interior $4$-ball, it is equivalent to work within $[0,1]\times S^3$. Likewise, it is equivalent to consider Lagrangian fillings in the symplectization $\R \times S^3$ or in the completion of the standard Weinstein structure on $B^4$.     
\end{remark}

\subsection{Smooth gradations}\label{subsec:grad_smooth}
A smooth knot $K\subset S^3$ is \textit{slice} if there is a smoothly and properly embedded disk $D\subset B^4$, called a \textit{slice disk}, such that $K = \partial D \subset \partial B^4$. If the radial function $B^4 \to [0,1]$ restricts to a Morse function on $D$ without critical points of index $2$, we call $K$ \textit{ribbon} and $D$ a \textit{ribbon disk}. Fox's famously open slice-ribbon conjecture \cite{fox1962knot} posits that every slice knot is ribbon. 

Multiple intermediary classes have been studied in the literature \cite{larson2015ribbondisks,hom2021ribbon,miller2021homotopyribbon,miller2023handleribbon}. Gordon \cite{gordon1981ribbon} showed that every ribbon knot $K$ is \textit{handle-ribbon},\footnote{The language \textit{strongly homotopy-ribbon} is also used in the literature. Here we follow the conventions of Miller and Zupan \cite[Remark 2.2]{miller2023handleribbon} and refer to their paper for more discussion on nomenclature.} which means that $K$ bounds a slice disk $D$ whose complement $B^4 - N(D)$ admits a handle decomposition without $3$-handles; such a handle decomposition is a \textit{ribbon handle decomposition} and $D$ is a \textit{handle-ribbon disk}. Handle-ribbonness in turn implies that $K$ is \textit{homotopy-ribbon}, which means that the inclusion of the knot complement into the disk complement induces a surjection
$\pi_1(S^3 - K)  \,\twoheadrightarrow\, \pi_1(B^4 - D)$. Whether or not slice implies homotopy-ribbon, homotopy-ribbon implies handle-ribbon, or handle-ribbon implies ribbon is unknown. 

We name two additional classes, known to be equivalent to ribbonness, for the purpose of clarifying the analogy with the symplectic category. Starting from the strongest condition, we say that a slice knot is \textit{strongly ribbon} if it admits a disk-band surgery presentation in which all band surgeries commute; i.e., if the order of index $1$ critical points of the Morse function on $D$ can be freely shuffled. That ribbonness implies strong ribbonness is straightforward (see e.g. \cite[Lemma 3.2]{etnyre2024lagrangian}).

Next, we say that a knot $K$ is \textit{AC-trivial handle-ribbon} if it admits a handle-ribbon disk $D$ whose induced ribbon handle decomposition of $B^4$ is \textit{Andrews-Curtis trivial}. Here, the induced handle decomposition of $B^4$ is obtained by viewing $D$ as the co-core of a $2$-handle attached to the ribbon handle decomposition of its complement. To say that this handle decomposition is Andrews-Curtis (AC) trivial means that it can be transformed into the single $0$-handle decomposition by births, deaths, and slides that never introduce $3$-handles. The ability to do so is not necessarily guaranteed, and the widely-believed falsity of the Andrews-Curtis conjecture \cite{andrews1965curtis} on balanced presentations of the trivial group is an obstruction. For example, the handlebodies in \cref{fig:gompf} are all diffeomorphic to $B^4$, but they are not known to be AC-trivial for $n\geq 3$. 

\begin{figure}[ht]
	\begin{overpic}[scale=.3]{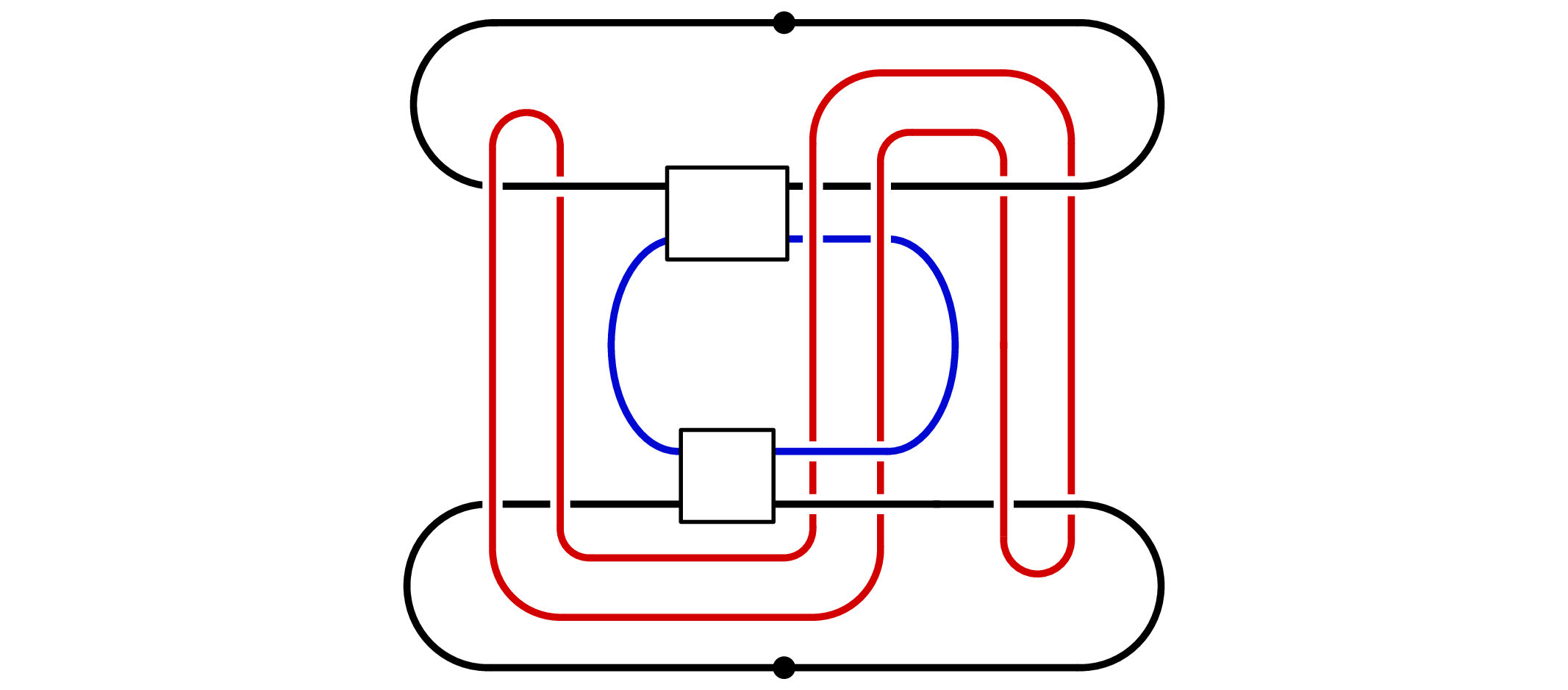}   
        \put(28,21){\small \textcolor{darkred}{$0$}}
        \put(40.75,21){\small \textcolor{darkblue}{$-1$}}
        \put(45.75,12.95){\scriptsize $n$}
        \put(43.1,29.9){\scriptsize $-\mkern-1mu 1\mkern-2mu-\mkern-2mu n$}
	\end{overpic}
	\caption{Handlebody presentations of $B^4$ that are not known to be AC-trivial for $n\geq 3$ \cite{gompf1991killingtheakbulut}. A $k$-box indicates $k$ full right-handed twists.}
	\label{fig:gompf}
\end{figure}
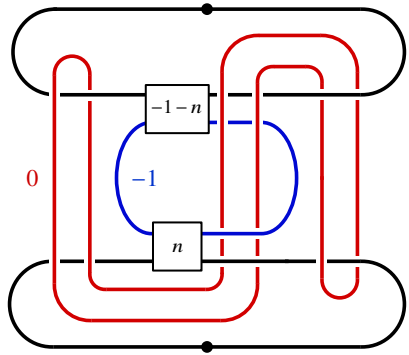

These handlebodies and their generalizations have a long history, dating back to Cappell and Shaneson \cite{cappell1976newfourmanifolds} and Akbulut and Kirby \cite{akbulut1985potential}, who studied their closures as potential counterexamples to the smooth $4$-dimensional Poincaré conjecture. Gompf \cite{gompf1991killingtheakbulut} eventually showed that the handlebodies in \cref{fig:gompf} are in fact standard, but needed $3$-handles to do so; much later, Akbulut \cite{akbulut2009cappell} and subsequently Gompf \cite{gompf2010morecappell} standardized even more Cappell-Shaneson families. Although they have failed to break Poincaré, $\pi_1$-presentations of Cappell-Shaneson homotopy balls remain well-known potential counterexamples to the Andrews-Curtis conjecture. 

In general, the slice-ribbon, Poincaré, Andrews-Curtis, and generalized Property R conjectures \cite{gompf2010property2R} form a tightly woven tapestry of open problems in low-dimensional topology that are all relevant to our discussion; see \cite{freedman2010manandmachine} for a survey. For instance, if $K$ is handle-ribbon, it is a component of an \textit{$R$-link}, i.e., an $n$-component link in $S^3$ on which $0$-surgery produces $\#^n S^1 \times S^2$. Indeed, the belt spheres of all $2$-handles of a ribbon handle decomposition of $B^4$ form an $R$-link, as $0$-surgery is witnessed by removing each co-core disk, leaving $\natural^n S^1\times B^3$ with boundary $\#^n S^1 \times S^2$. The generalized Property R conjecture (GPRC), a multi-component generalization of Gabai's famous Property R theorem for knots \cite{gabai1987propertyR}, asserts that every $R$-link is handle slide equivalent to a $0$-framed unlink. As $0$-framed handle slides preserve ribbonness, any component of a link satisfying GPRC is ribbon. Consequently, GPRC would imply equivalence of handle-ribbonness and ribbonness. 

By flipping a ribbon handle decomposition of $B^4$ upside down, we obtain a $2$-/$3$-/$4$-handle decomposition of a standard $B^4$-cap cobordism of $S^3$ with $2$-handles attached along a $0$-framed $R$-link. If the original ribbon handle decomposition is AC-trivial, then the $R$-link at the base of the upside-down cobordism is handle slide equivalent to a $0$-framed unlink, as the $2$-handles are handle slide-cancelable by the $3$-handles. Consequently, the above discussion implies:

\begin{proposition}\label{prop:smooth}
If a knot is AC-trivial handle-ribbon, then it is ribbon.     
\end{proposition}

Though \cref{prop:smooth} follows from the above discussion, itself a reformulation of the discussions in \cite{gompf2010property2R,freedman2010manandmachine}, the author is unaware of an explicit statement in the literature. For more context, we provide an additional right-side-up proof, analogous to our symplectic proof of \cref{thm:main3} but somewhat more informal, in deference to the more careful treatment in \cref{sec:slice}.

\begin{remark}
    The upward implication from ribbon to AC-trivial handle-ribbon follows from the fact that, by definition, inclusion of a ribbon disk back into its complement yields the radial function on $B^4$.
\end{remark}

\begin{proof}[Proof of \cref{prop:smooth}.]
Let $K\subset S^3$ be AC-trivial handle-ribbon. Then $K$ is the belt sphere of a $2$-handle of an AC-trivial ribbon handle decomposition of $B^4$. Consequently, $K$ is isotopic to a small meridian curve of a $2$-handle attaching sphere in an AC-trivial Kirby diagram for $B^4$ involving only $1$-handles and $2$-handles. 

Ignoring $K$, by AC-triviality, we may perform isotopies, $1$-/$2$-births, and handle slides to transform the Kirby diagram into one in geometrically canceling position, intentionally delaying all cancellations. Some observations about this process:
\begin{enumerate}
    \item Viewing the $1$-handles as dotted circles, $K$ is unlinked from the initial $1$-handle configuration. handle slides of $1$-handles preserve this property, as do $1$-/$2$-handle-births.   
    \item handle slides of $2$-handles are performed along bands and can generically be done in the complement of $K$ (of course, possibly linking with $K$).
\end{enumerate}
By these observations and isotopy extension, we produce a Kirby diagram for $B^4$ consisting of a collection of geometrically canceling $1$-/$2$-handle pairs, in which $K$ is drawn as a knot which, upon erasing all $2$-handles, is the unknot (with a possibly complicated diagram) in $\#^n S^1\times S^2$.

We wish to perform the $1$-/$2$-handle cancellations, but $K$ may intersect the belt spheres of the $1$-handles. Since $K$ is an unknot in $\#^n S^1\times S^2$, there is an ``illegal isotopy'' of $K$ to a small unknot, disjoint from the belt spheres, involving passage through the fixed $2$-handle attaching spheres; see the left two panels of \cref{fig:bandint} for the local intersection model.

\begin{figure}[ht]
	\begin{overpic}[scale=.346]{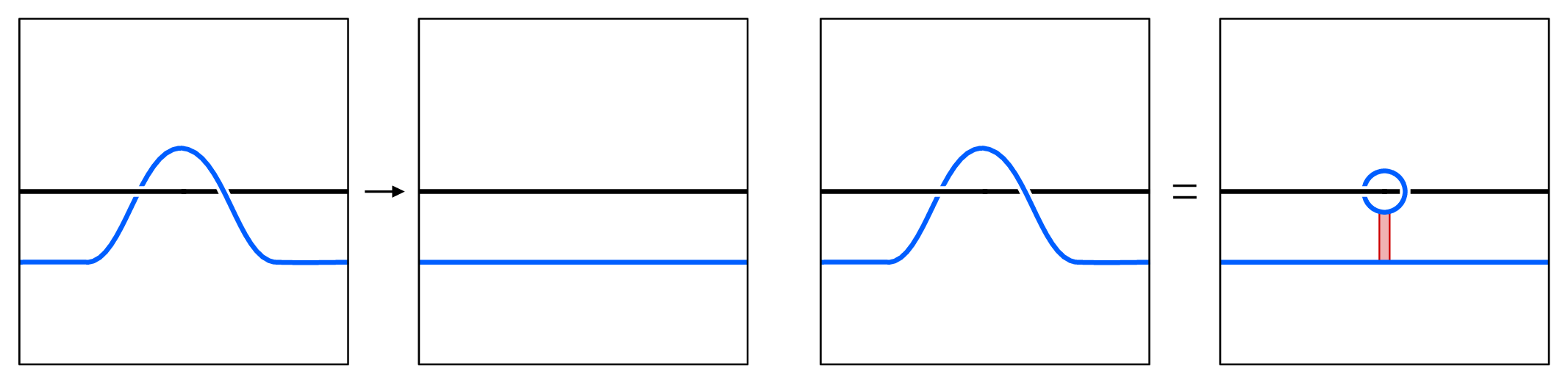}   
        \put(3,13){\footnotesize $n$}
        \put(28.5,13){\footnotesize $n$}
        \put(54,13){\footnotesize $n$}
        \put(80,13){\footnotesize $n$}

        \put(2.5,4.5){\footnotesize \textcolor{lightblue}{$K$}}
        \put(53.5,4.5){\footnotesize \textcolor{lightblue}{$K$}}
	\end{overpic}
	\caption{The left two panels give a local model for $2$-handle intersections in the illegal isotopy of $K$. The right two panels give a band-surgery presentation of $K$ with a meridian and a component that has ``passed through'' the intersection.}
	\label{fig:bandint}
\end{figure}
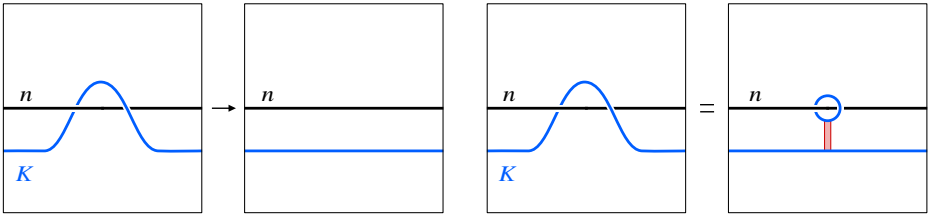

We account for the intersections with $2$-handle attaching spheres by recording the ``trace'' of the illegal isotopy through the intersection with a meridian curve and band surgery; see the right side of \cref{fig:bandint}. By applying this local model to each $2$-sphere intersection, and tracing the isotopy with the bands, we may present $K$ via band surgeries in the geometrically canceling diagram on a link which does not cross the $1$-handles. In the concrete example of \cref{fig:bandslide}, this is the status of the third panel. 

\begin{figure}[hbt]
	\begin{overpic}[scale=.345]{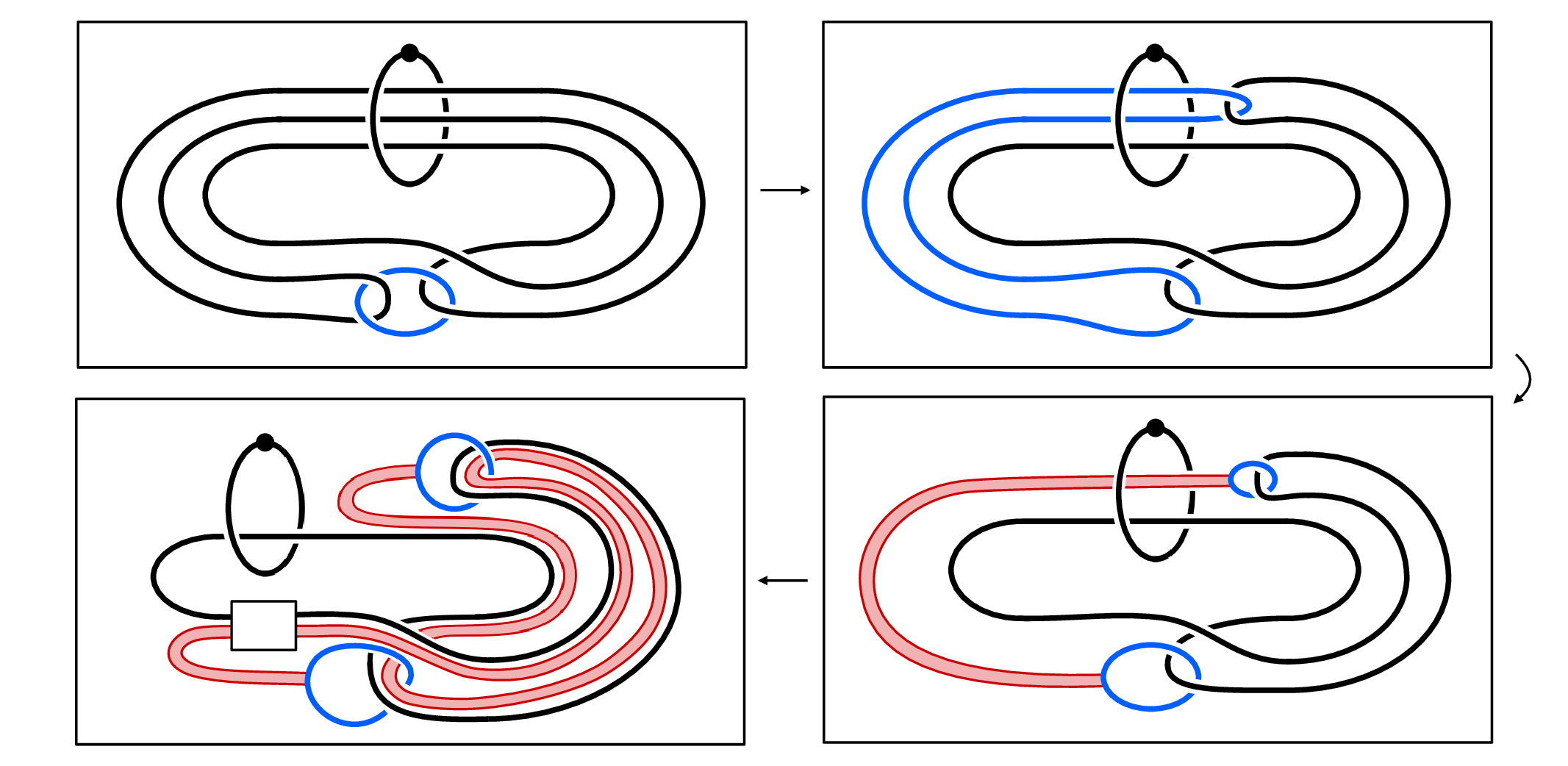}   
        \put(43,29){\footnotesize $n$}
        \put(42,4.5){\footnotesize $n$}
        \put(16.25,8.35){\footnotesize $n$}
        \put(91,29){\footnotesize $n$}
         \put(91,5){\footnotesize $n$}
       
        \put(28,26){\scriptsize \textcolor{lightblue}{$K$}}
        \put(76,26){\scriptsize \textcolor{lightblue}{$K$}}
     
	\end{overpic}
	\caption{In the first panel, the $1$-/$2$-handle pair is in algebraically canceling position. After an isotopy that drags $K$ through the $1$-handle, the handles are in geometrically canceling position. Tracing the illegal isotopy of $K$ back to a small unknot with a band surgery yields the third panel. In the fourth panel, we slide the band across the $2$-handle. The handles may subsequently be canceled, leaving a disk-band presentation of $K$.}
	\label{fig:bandslide}
\end{figure}
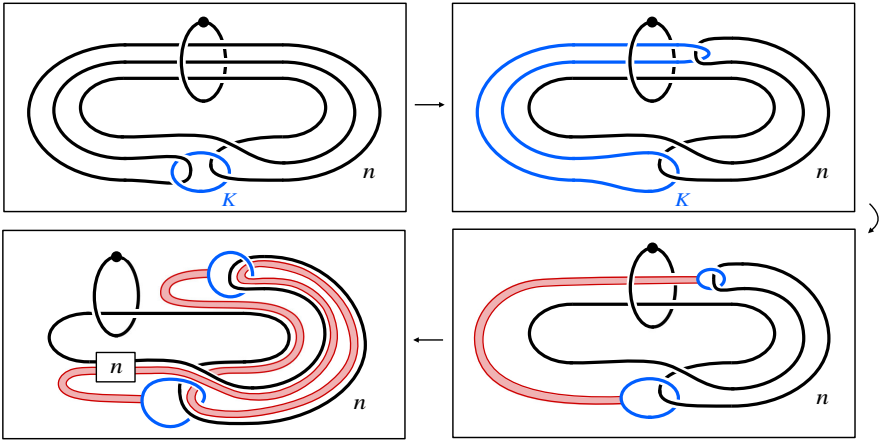

To perform the handle cancellations, it suffices to handle slide any bands crossing $1$-handles across the corresponding canceling $2$-handles. After a sufficient number of such band-slides, the handles in the diagram may be erased. What remains is a disk-band presentation of $K$, implying that $K$ is ribbon. 
\end{proof}

\begin{remark}
In the above proof, AC-triviality prevents $2$-/$3$-handle births. The problem this poses in terms of concluding ribbonness is that the birthed $0$-framed $2$-handle attaching sphere could link with $K$. It is then unclear how to move $K$ out of the way in order to ultimately cancel the handles as is done with the $1$-/$2$-pairs. In particular, canceling such a $2$-/$3$-pair induces a maximum in the slice disk with respect to the radial function. 
\end{remark}

\subsection{Symplectic gradations}\label{subsec:sympelctic_grads}

In this subsection only, we use slice-like adjectives to refer to Legendrian knot types, analogous to the practice in smooth topology. Aside from \textit{Lagrangian slice}, the author is unaware of such usage in the symplectic literature. Afterwards, we will revert to saying e.g. \textit{regularly slice} instead of simply \textit{regular}. 

The following definition collects all of the symplectic terms in \cref{fig:SRgradations}.

\begin{definition}
A Legendrian knot type $\Lambda \subset (S^3, \xi_{\mathrm{st}})$ is:
\begin{enumerate}
    \item \textbf{Lagrangian slice} if it admits a Lagrangian disk filling $D\subset \R \times S^3$ in the symplectization.
    \item \textbf{Lagrangian homotopy-ribbon} if it is Lagrangian slice and its underlying smooth knot type is homotopy-ribbon. 
    \item \textbf{Weakly regular} if it bounds the co-core disk of a $2$-handle in a finite-type Weinstein structure on $\R\times S^3$ symplectomorphic to the symplectization. 
    \item \textbf{Regular} if it bounds the co-core disk of a $2$-handle in a finite-type Weinstein structure on $\R\times S^3$ Weinstein homotopic to the symplectization. 
    \item \textbf{Decomposable} if it admits a decomposable disk filling $D\subset \R \times S^3$ in the symplectization. 
    \item \textbf{Strongly decomposable} if it admits a decomposable disk filling $D\subset \R \times S^3$ in the symplectization in which all Legendrian surgeries commute, i.e., if there is a max-tb unlink $\tilde{U}$ and a set of embedded surgery arcs $G$ such that $\Lambda = \mathrm{Surg}(\tilde{U}, G)$; see the discussion after \cref{def:dec_moves}.
\end{enumerate}
\end{definition}

Strongly decomposable cobordisms are those given by the \textit{Legendrian handle graph} framework of \cite{sabloff2021upper}. We do not know if decomposable implies strongly decomposable; see \cite[Remark 2.7]{sabloff2021upper}. Locally, the reason is given by \cref{fig:strongdec}, which is \cite[Figure 7]{sabloff2021upper}. In \cref{subsec:decvsstrong} we provide a recipe (see \cref{prop:slicesats}) for producing candidate decomposably slice knots that may potentially not be strongly decomposably slice; the Legendrian $\Delta_2^{3}$-satellite of $\overline{9_{46}}$ is an example. 

\begin{figure}[ht]
	\begin{overpic}[scale=.28]{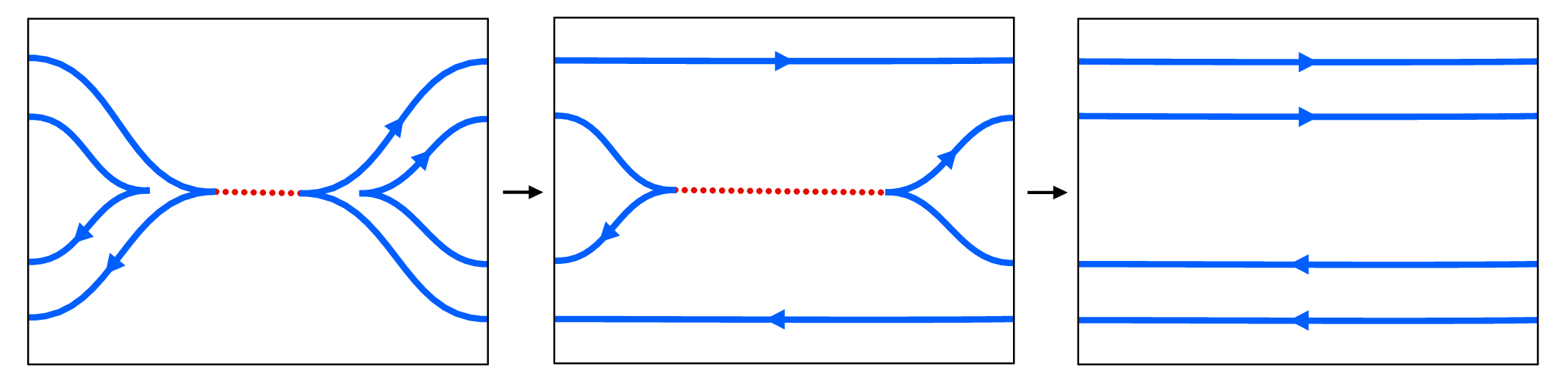}   
        
	\end{overpic}
	\caption{A sequence of two Legendrian surgeries where it is unclear if the second can be performed before the first. (Locally, the order \textit{cannot} be reversed; if commutation is possible, it is for a semi-local or global reason.)}
	\label{fig:strongdec}
\end{figure}

Observe the difference between weak regularity and regularity. Weinstein homotopy implies (exact) symplectomorphism, but we do not know in general if symplectomorphic Weinstein structures on a fixed manifold are Weinstein homotopic \cite[Problem 1.4]{eliashberg2018weinsteinrevisited}. However, there is a unique Weinstein filling of $(S^3, \xi_{\mathrm{st}})$ up to \textit{deformation equivalence}, i.e., Weinstein homotopy after a self-diffeomorphism of $B^4$ \cite[Theorem 16.5]{cieliebak2012stein}. Here the problem is that we do not know if $\mathrm{Diff}_+(B^4, \partial B^4)$ is connected \cite{hatcher2012diffeogroups}, so the Weinstein structure furnished by weak regularity may be Weinstein homotopic to the symplectization only after a diffeomorphism not isotopic to the identity; see the discussion on \cite[p. 311]{cieliebak2012stein}. Nevertheless, any such diffeomorphism is relative to the boundary and does not affect the knot type $\Lambda$. We have thus proven: 

\begin{proposition}
If a Legendrian knot is is weakly regular, then it is regular.     
\end{proposition}

Next, note that regularity of a Legendrian knot implies AC-trivial handle-ribbonness of the smooth knot, as (parametric) Weinstein structures do not have handles of index $3$. Consequently, \cref{prop:smooth} gives: 

\begin{proposition}\label{prop:reg_smth_ribbon}
If a Legendrian knot is (weakly) regular, then it is smoothly ribbon.     
\end{proposition}

In fact, G. Dimitroglou Rizell pointed out to the author that every Lagrangian-fillable Legendrian is smoothly ribbon fillable, a stronger version of \cref{prop:reg_smth_ribbon} which in particular gives the diagonal (non-reversible) red implication, and consequently the top-left equivalence, in \cref{fig:SRgradations}. Indeed, as explained in \cite{cao2014topologicallydistinct}, work of Eliashberg \cite{eliashberg1995pushoff} implies that any Lagrangian filling of a Legendrian may be perturbed to a symplectic filling of the transverse pushoff of the Legendrian. By Boileau and Orevkov \cite{boileau2001quasipositive}, this transverse pushoff is quasipositive, and is therefore the cross section of a holomorphic curve in $\C^2$ by Rudolph \cite{rudolph1983seifertribbons}, hence smoothly ribbon; see e.g. \cite{hayden2015positivefillable,hayden2019crosssections,mark2024fillable} for additional relevant discussion. 

%Alternatively, the same conclusion follows from \cref{thm:main3}, which is a once-stable Legendrian version of \cref{prop:smooth}. Indeed, decomposable cobordisms are smoothly ribbon and Legendrian stabilization does not change the smooth knot type. In any case, this provides additional evidence for \cref{conj:slice}, which asserts equivalence of regularity and decomposability in \cref{fig:SRgradations}.

\begin{remark}\label{remark:ELremark}
Recent work of Etnyre and Leverson \cite{etnyre2024lagrangian} shows that smooth ribbon cobordisms may be approximated by (strongly) decomposable Lagrangian cobordisms whose Legendrian ends are sufficiently stabilized. Their result then gives a weaker version of \cref{thm:main3} where the single stabilization is replaced by ``sufficiently many stabilizations.'' In fact, by generalizing the above discussion on AC-trivial handle-ribbonness to handle-ribbon cobordisms rather than just slice disks, it is likely that Etnyre and Leverson's result gives a ``sufficiently stabilized'' affirmative answer to \cref{q:3}. We do not consider the details here.
\end{remark}

Though it is not directly relevant, within a survey of slice-ribbon and its role in symplectic and contact topology we would be remiss not to mention work of Baker \cite{baker2016concordancefibered}. Among other results, he showed that to disprove slice-ribbon, it suffices to find two open book decompositions of $(S^3, \xi_{\mathrm{st}})$ whose connected bindings are smoothly concordant, but not isotopic.

\section{Background}\label{sec:background}

Here we collect the background material necessary for the rest of the article. This includes exactness and decomposability in \cref{subsec:classes}, Weinstein topology and Legendrian Kirby calculus in \cref{subsec:weinstein_review}, regularity in \cref{subsec:reg_review}, normal rulings in \cref{subsec:NR_review}, and a bit of convex surface theory in \cref{subsec:CHT_review}.

\subsection{Exact and decomposable cobordisms}\label{subsec:classes}

\begin{definition}
Let $\Lambda_-, \Lambda_+ \subset (Y^3, \xi=\ker\alpha)$ be oriented Legendrian links in a co-oriented contact manifold. A \textbf{Lagrangian cobordism from $\Lambda_-$ to $\Lambda_+$} is an embedded oriented Lagrangian surface $L\subset (\R_s \times Y, d\lambda_{\mathrm{st}})$ where $\lambda_{\mathrm{st}}:=e^s\, \alpha$ such that 
\begin{enumerate}
    \item $L \cap ([-s_0, s_0]\times Y)$ is compact for any $s_0\in \R_{\geq 0}$, and 
    \item there exists a $C>0$ such that 
    \begin{align*}
        L \cap ([C, \infty) \times Y) &= [C, \infty) \times \Lambda_+, \\
        L \cap ((-\infty, -C] \times Y) &= (-\infty, -C] \times \Lambda_-.
    \end{align*}
\end{enumerate}
If $\lambda_{\mathrm{st}}\mid_L = df$ for some function $f:L \to \R$ which is constant on $[C, \infty) \times \Lambda_+$ and $(-\infty, -C] \times \Lambda_-$, then $L$ is \textbf{exact}.
\end{definition}

\begin{figure}[ht]
	\begin{overpic}[scale=.3475]{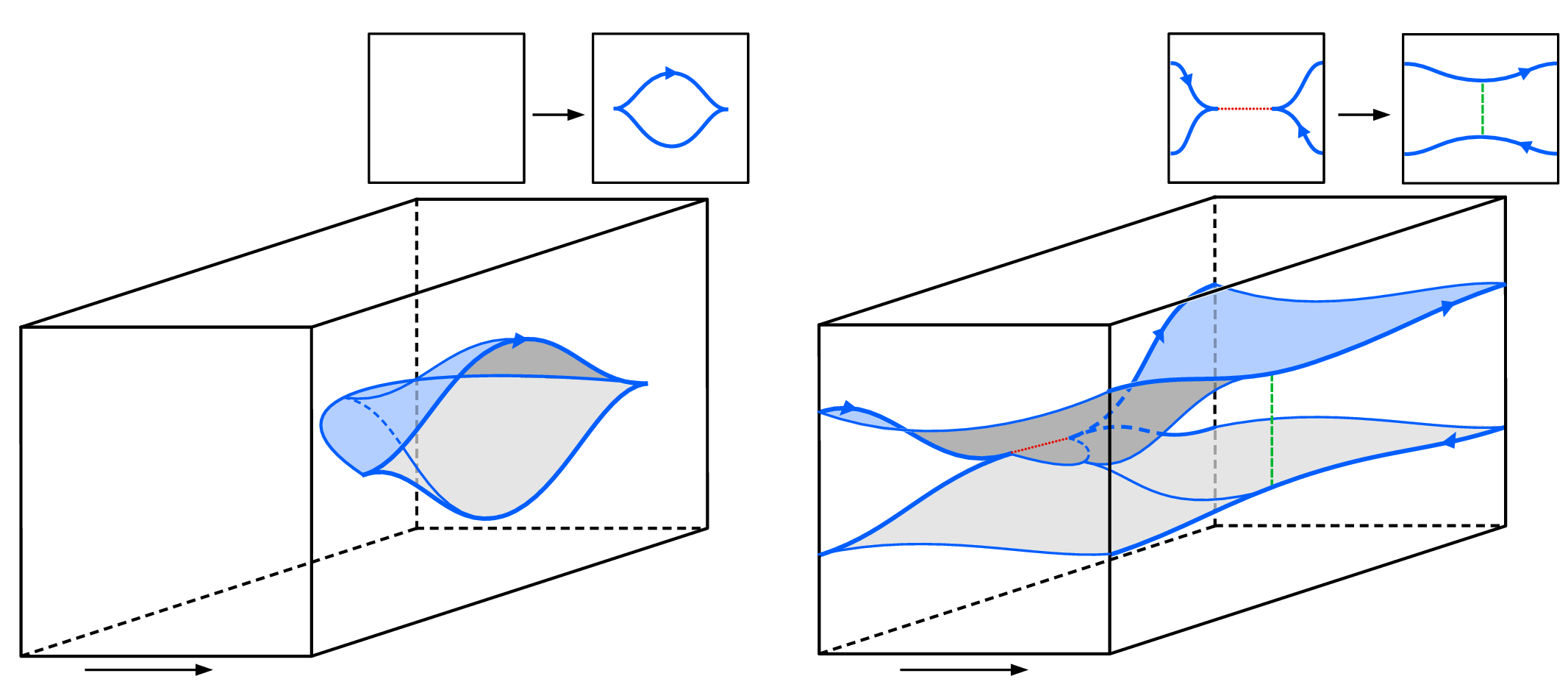}   
      \put(14,1){\tiny $s$}
      \put(66,1){\tiny $s$}
	\end{overpic}
	\caption{Exact cobordisms (lifted to fronts of Legendrian surfaces in $\R^5$) associated to the decomposable moves unknot birth (left) and Legendrian surgery (right).}
	\label{fig:decomposable}
\end{figure}

\begin{definition}[Decomposable moves]\label{def:dec_moves}\textbf{Unknot death} is the removal of an isolated standard\footnote{Here and throughout, \textit{standard} refers to the $2$-cusp front projection of the max-tb unknot.} max-tb unknot component; \textbf{unknot birth} is the reverse procedure. A \textbf{pinch move} is performed along a contractible Reeb chord (c.f. \cref{def:contr}) and results in a Legendrian whose front is given by modifying the upper far right front in \cref{fig:decomposable} to produce the upper third front. When $\Lambda$ is oriented, a pinch move is \textbf{orientable} if the initial strands are oppositely oriented in the front. The reverse procedure is \textbf{(ambient) Legendrian surgery}.
\end{definition}

Ambient Legendrian surgery on a link $\Lambda$ may be performed more generally along an embedded Legendrian arc $\gamma$ with boundary on $\Lambda$ such that $\gamma \cap \Lambda$ is $\xi$-transverse. The result is band surgery with core $\gamma$ and framing that makes half of a left-handed turn with respect to the contact framing; consequently, the surgery is \textit{orientable} only if this left-handed half twist respects the orientation of $\Lambda$. We denote the resulting link $\mathrm{Surg}(\Lambda, \gamma)$, or $\mathrm{Surg}(\Lambda, G)$ when $G=\{\gamma_1, \dots, \gamma_k\}$ is a set of multiple surgery arcs; see \cref{fig:LHG}.

\begin{figure}[ht]
	\begin{overpic}[scale=.345]{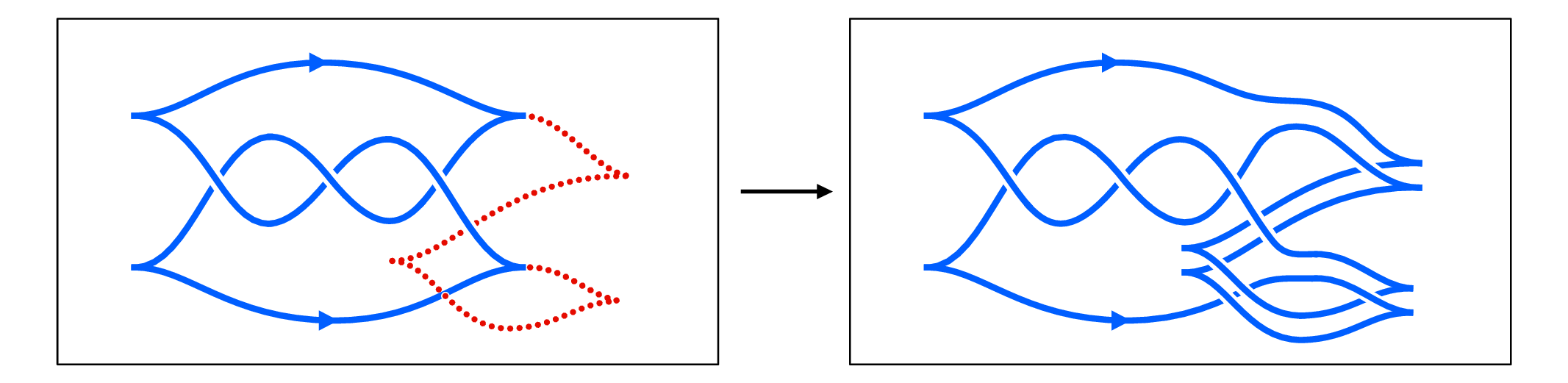}   
    \put(78,20){\scriptsize \textcolor{lightblue}{$\mathrm{Surg}(\Lambda, \gamma)$}}
    \put(28,20){\scriptsize \textcolor{lightblue}{$\Lambda$}}
    \put(38,16){\scriptsize \textcolor{darkred}{$\gamma$}}
	\end{overpic}
	\caption{A general Legendrian surgery.}
	\label{fig:LHG}
\end{figure}

\begin{theorem}\cite{ekholm2012exactcobordisms}
Suppose that $\Lambda_+\subset (Y, \xi)$ is obtained from $\Lambda_-$ by a sequence of Legendrian Reidemeister moves, unknot births, and (orientable) Legendrian surgeries. Then there is an (orientable) exact Lagrangian cobordism in $\R \times Y$ from $\Lambda_-$ to $\Lambda_+$. By definition, we call a cobordism arising in this way \textbf{decomposable}.  
\end{theorem}

One reason why decomposability is subtle is the requirement that pinch moves be performed along \textit{contractible} Reeb chords, and contractibility is somewhat nebulous.\footnote{The discussion of \textit{strong decomposability} in \cref{sec:gradations} is relevant here.} Strictly speaking, the definition (first given in \cite[Definition 6.13]{ekholm2012exactcobordisms}) refers explicitly to the Lagrangian projection.

\begin{definition}[Contractible Reeb chord]\label{def:contr}
Let $\Lambda$ be a Legendrian whose Lagrangian projection $\Pi_{xy}(\Lambda)$ contains only transverse double points. A Reeb chord $\zeta$ corresponding to such a double point is \textbf{contractible} if there is a Legendrian isotopy of $\Lambda$ inducing a planar isotopy of $\Pi_{xy}(\Lambda)$ --- that is, with no Lagrangian Reidemeister moves --- that makes the length of $\zeta$ arbitrarily small.     
\end{definition}

See \cref{ex:contr} below to clarify this subtlety. It is instead often convenient to appeal to a more generous definition of contractibility that, up to a reasonable isotopy, implies contractibility as in \cref{def:contr}.

\begin{definition}[Front contractible Reeb chord]
Let $\Lambda$ be a Legendrian link and $\zeta$ a Reeb chord with endpoints on $\Lambda$. We say that $\zeta$ is \textbf{front contractible} if $\Pi_{xz}(\gamma) \cap \Pi_{xz}(\Lambda) = \partial \Pi_{xz}(\zeta)$, where $\Pi_{xz}$ is the front projection. 
\end{definition}

\begin{lemma}\label{lemma:front_contr}
Let $\Lambda$ be a Legendrian link and $\zeta$ a front contractible Reeb chord. Let $\mathcal{U}$ be an arbitrarily small neighborhood of $\Pi_{xz}(\gamma)$ in the front. Then there is an isotopy of $\Lambda$ supported in $\Pi_{xz}^{-1}(\mathcal{U})$ such that the induced Reeb chord is contractible. 
\end{lemma}

\begin{proof}
By front contractibility, we may assume that $\Pi_{xz}(\Lambda)\vert_\mathcal{U}$ consists only of the two strands connected by $\Pi_{xz}(\zeta)$. We may then perform a planar isotopy of the front $\Pi_{xz}(\Lambda)\vert_\mathcal{U}$ to the top right diagram of \cref{fig:decomposable}. Then, converting to the Lagrangian projection --- not using Ng's resolution, but rather the honest Lagrangian projection --- we see that $\zeta$ is contractible in the sense of \cref{def:contr}.     
\end{proof}

\begin{figure}[ht]
	\begin{overpic}[scale=.34]{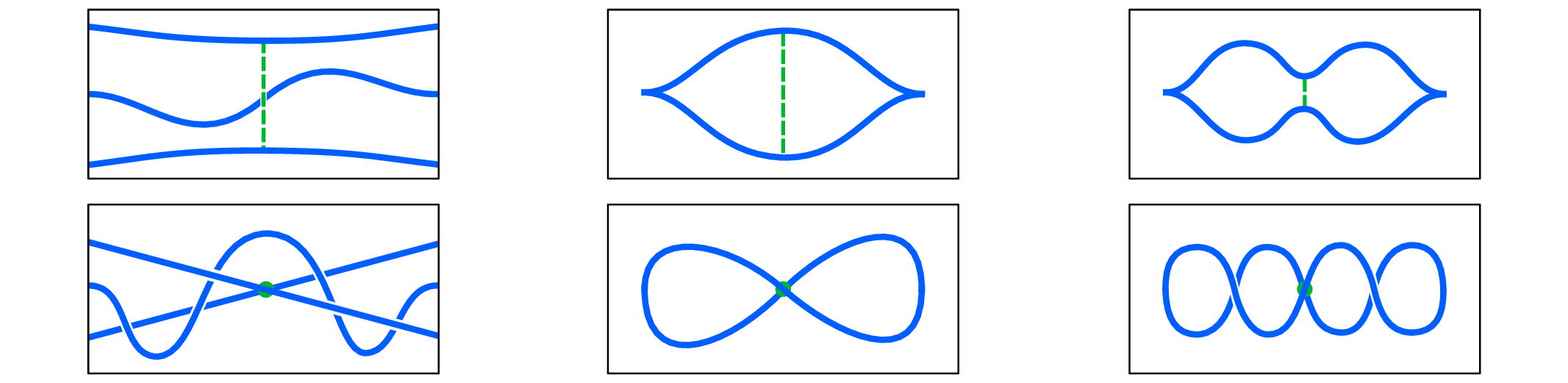}   

	\end{overpic}
	\caption{A Reeb chord that is not front contractible, front contractible but not contractible, and contractible, respectively.}
	\label{fig:contr}
\end{figure}

\begin{example}[Contractibility and front contractibility]\label{ex:contr}
\cref{fig:contr} gives three examples of Reeb chords. The top half of the figure consists of front projections, while the bottom half contains the corresponding honest Lagrangian projections. On the far left is a Reeb chord which is not front contractible. In the middle is the standard front for the max-tb unknot with its single Reeb chord. This Reeb chord is front contractible, but not contractible. Finally, the far right depicts the application of \cref{lemma:front_contr} to the standard unknot front, producing a genuinely contractible Reeb chord after an isotopy. 
\end{example}

\subsection{Weinstein topology}\label{subsec:weinstein_review}

Here we review Weinstein structures and discuss in detail the standard structure on the cotangent bundle. We also discuss the role of the contact topology of a regular level set, in particular the handlebody diagrams of Gompf \cite{gompf1998handlebody}. The standard reference is \cite{cieliebak2012stein}. 

\begin{definition}
Let $(W, \partial_- W, \partial_+W)$ be a compact cobordism of dimension $2n$. A \textbf{Weinstein structure} on the cobordism is a tuple $(\lambda, \phi)$ where
\begin{enumerate}
    \item $(W, \partial_- W, \partial_+ W, \lambda)$ is a \textbf{Liouville cobordism}, i.e., $d\lambda$ is symplectic and the \textbf{Liouville vector field} $X_{\lambda}$ defined by $\iota_{X_{\lambda}}d\lambda = \lambda$ is outwardly (resp. inwardly) transverse to $\partial_+ W$ (resp. $\partial_- W$), 
    \item $\phi: W \to \R$ is a Morse function such that $\partial_{\pm} W = \phi^{-1}(c_{\pm})$ is a regular level set, and 
    \item for some Riemannian metric, the Liouville vector field $X_{\lambda}$ is gradient-like for $\phi$. 
\end{enumerate}
If $\partial_- W = \emptyset$, we write $(W, \lambda, \phi)$ and call the cobordism a \textbf{Weinstein domain}. A \textbf{Weinstein homotopy} is a $1$-parameter family $(W, \partial_- W, \partial_+ W, \lambda_t, \phi_t)$ of Weinstein structures on a fixed cobordism $W$, allowing for embryonic critical points. Two Weinstein cobordisms are \textbf{deformation equivalent} if they are Weinstein homotopic under the pullback by a diffeomorphism.  
\end{definition}

As a consequence of the Liouville condition, critical points of a Weinstein structure of dimension $2n$ have index at most $n$. Consequently, the $4$-dimensional Weinstein cobordisms that we consider in this paper comprise only of $0$-, $1$-, and $2$-handles.

Every Weinstein cobordism $(W, \partial_- W, \partial_+ W, \lambda, \phi)$ can be canonically \textit{completed} by attaching cylindrical ends
\[
 ((-\infty,0]_s \times \partial_- W, \, e^s\, \lambda\mid_{\partial_- W}) \, \cup \, (W, \partial_- W, \partial_+ W, \lambda) \, \cup \, ([0,\infty)_s \times \partial_+ W, \, e^s\, \lambda\mid_{\partial_+ W})
\]
and extending $\phi$ in the obvious way. Weinstein homotopic cobordisms have symplectomorphic completions \cite[Corollary 11.21]{cieliebak2012stein}, and we will usually not distinguish between a Weinstein cobordism and its completion. Sometimes a Weinstein cobordism will be written as $(W, \lambda, \phi)$, or even just $W$. Finally, we also consider cobordisms with corners, where we allow $\partial_{\pm} W$ to have a common boundary $\partial^2 W$. This occurs, for instance, when discussing Weinstein handles associated to $\phi$ and $X_{\lambda}$.

\subsubsection{Cotangent bundles}
Let $L$ be a compact $n$-dimensional cobordism. The cotangent bundle $T^*L$ admits a canonical Liouville form $\lambda_{\mathrm{st}} = \vec{p}\, d\vec{q}$, where $\vec{q}$ is any local coordinate system on $L$ and $\vec{p}$ is the corresponding dual coordinate system. The Liouville vector field $\vec{p}\, \partial_{\vec{p}}$ vanishes along the $0$-section. Via a choice of Riemannian metric on $L$ we may induce a metric on $T^*L$ and consider the unit disk cotangent bundle $\D^*L = T^*L \cap \{\norm{\vec{p}} \leq 1\}$. Strictly speaking, to endow $\D^*L$ with a (genuinely Morse) Weinstein structure --- either of a Weinstein domain after rounding corners, or possibly a Weinstein cobordism with corners if $\partial L \neq \emptyset$; see \cref{fig:corners} --- we need to perform a deformation. To that end, we have the following lemma, which is essentially \cite[Lemma 12.8]{cieliebak2012stein}. We include a proof for completeness, as the statement is crucial for \cref{thm:cob_diag}.

\begin{lemma}[The standard Weinstein structure on the cotangent bundle]\label{lemma:cot_bund_str}
Let $L$ be a compact $n$-dimensional cobordism equipped with a Riemannian metric. Consider $\D^*L$ with the standard Liouville form $\lambda_{\mathrm{st}} = \vec{p}\, d\vec{q}$. Let $X\in \mathfrak{X}(L)$ be any vector field. There is a Liouville homotopy to a Liouville form $\lambda$, compactly supported outside of a neighborhood of the horizontal boundary $\partial_{\mathrm{hor}} \D^*L = \mathbb{S}^*L$, with the following properties. 
\begin{enumerate}
    \item The Liouville vector field $X_{\lambda}$ coincides with $X$ along the $0$-section.\label{cot_lvf_1} 
    
    \item If $X$ is gradient-like for a function $f:L \to \R$, and if the eigenvalues of $X$ at zeroes have real part $< 1$, then $X_{\lambda}$ is gradient-like for a function $\hat{f}:\D^*L \to \R$ that agrees with $\ve f(\vec{q}) + \norm{\vec{p}}^2$ near $\{\vec{p}=0\}$ for some sufficiently small $\ve > 0$, has no critical points away from $\{\vec{p}=0\}$, and agrees with $\norm{\vec{p}}^2$ near $\partial_{\mathrm{hor}} \D^*L$.\label{cot_lvf_2}
\end{enumerate} 
In particular, if $L$ is a cobordism  and $f:L \to \R$ is Morse, then there is a gradient-like vector field $X$ for $f$ such that $(\D^*L, \lambda, \hat{f})$ is a Weinstein domain if $\partial_- L = \emptyset$, or a Weinstein cobordism with corner $\mathbb{S}^*L\mid_{\partial_- L}$ if $\partial_- L \neq \emptyset$. Moreover, any two such choices of Weinstein structure are Weinstein homotopic, and if the two structures agree near $\partial_{\pm} L$ then the homotopy may be taken to be constant near $\partial_{\pm} L$.   
\end{lemma}

\begin{figure}[ht]
	\begin{overpic}[scale=.34]{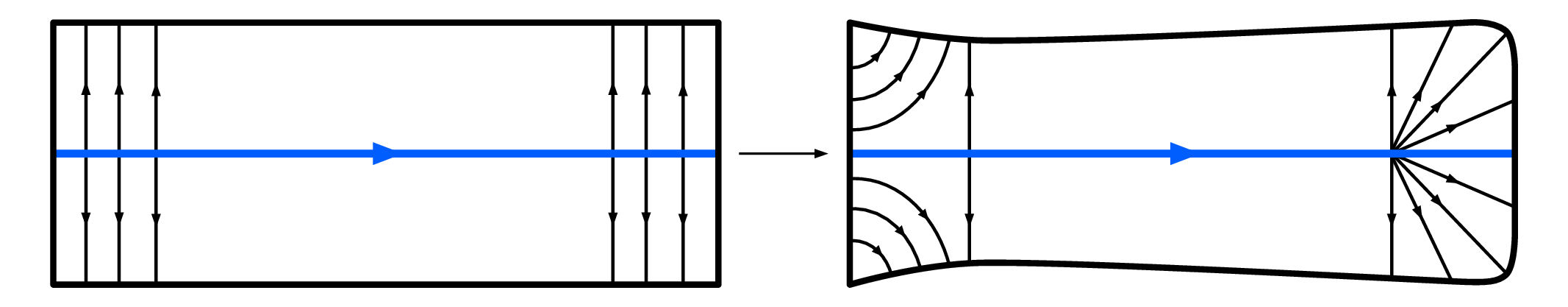}   
       
	\end{overpic}
	\caption{Perturbing the Liouville structure on $\D^*L$ when $L$ is a cobordism with $\partial_- L \neq \emptyset$. The result is a Weinstein cobordism with corners.}
	\label{fig:corners}
\end{figure}

\begin{proof}
Let $\eta:\D^n_{\vec{p}}: \to [0,1]$ be a smooth bump function, identically $1$ near the origin and identically $0$ in a neighborhood of $\partial \D^n$. Let $H: \D^*L \to \R$ be defined by $H(\vec{q}, \vec{p}) := \ip{\eta(\vec{p})\vec{p}}{X(\vec{q})}$, where $\ip{\cdot}{\cdot}$ is the pairing between $T^*L$ and $TL$; that is,
\[
H(\vec{q},\vec{p}) = \eta(\vec{p}) \, \lambda_{\mathrm{st}, (\vec{q},\vec{p})}(X(\vec{q})).
\]
Near $\vec{p} = 0$ we have $H(\vec{q}, \vec{p}) = \ip{\vec{p}}{X(\vec{q})}$, and by construction $H$ is compactly supported away from $\partial_{\mathrm{hor}} \D^*L$. The Liouville homotopy in question is $\lambda_t := \lambda_{\mathrm{st}} - t\, dH$ for $0\leq t \leq 1$. We let $\lambda = \lambda_1$ and observe that $X_{\lambda} = \vec{p}\, \partial_{\vec{p}} - X_H$, where $X_H$ is the Hamiltonian vector field of $H$, i.e., the vector field satisfying $\iota_{X_H}d\lambda_{\mathrm{st}} = dH$. 

We prove \eqref{cot_lvf_1}. Writing $X = \sum_{i=1}^n X_i\, \partial_{q_i}$, we have $H(\vec{q}, \vec{p}) = \eta(\vec{p})\sum_{i=1}^n p_iX_i(\vec{q})$. Then
\[
dH = \eta\sum_{i,j=1}^n \delta_{ij}X_i\, dp_j + \sum_{i,j=1}^n\frac{\partial \eta}{\partial p_j}p_iX_i \, dp_j + \eta\sum_{i,j=1}^n p_i\frac{\partial X^i}{\partial q_j}\, dq_j
\]
and
\[
X_H = -\eta X - \sum_{i,j=1}^n \frac{\partial \eta}{\partial p_j}p_iX_i \, \partial_{q_j} + \eta\sum_{i,j=1}^n  p_i\frac{\partial X^i}{\partial q_j}\,\partial_{p_j} .
\]
Along the $0$-section $\{\vec{p} = 0\}$, we see that $X_{\lambda} = -X_H = X$, which gives the first part of \eqref{cot_lvf_1}. For the second part of \eqref{cot_lvf_1}, note that for any $\vec{q}_0$ such that $X(\vec{q}_0) = 0$, the $\partial_{\vec{q}}$-component of $X_H$ vanishes, hence $X_{\lambda}$ is tangent to the cotangent fiber. 

Now consider \eqref{cot_lvf_2}. Given $\eta$ as above, fix $\ve > 0$ and define $\hat{f}(\vec{q}, \vec{p}) := \ve\eta(\vec{p})f(\vec{q}) + \norm{\vec{p}}^2$. By construction, $\hat{f}$ agrees with $\ve f(\vec{q}) + \norm{\vec{p}}^2$ near $\{\vec{p}=0\}$ and with $\norm{\vec{p}}^2$ near $\partial_{\mathrm{hor}} \D^*L$. Note that
\begin{equation}\label{eq:cotangent_bundle}
d\hat{f} = 2\vec{p}\, d\vec{p} + \ve (f\, d\eta + \eta\, df) = \left(2\vec{p} + \ve f\frac{\partial \eta}{\partial \vec{p}}\right)\, d\vec{p} + \ve \eta \, df.
\end{equation}
In order for $d\hat{f}$ to vanish, we need $df$ and $2\vec{p} + \ve f\frac{\partial \eta}{\partial \vec{p}}$ to vanish. Since $\eta \equiv 1$ in a neighborhood of $\vec{p} = 0$, for $\ve$ sufficiently small, the latter occurs only when $\vec{p} = 0$; this shows that the critical points of $\hat{f}$ coincide with those of $f$ along the $0$-section. To show that $X_{\lambda}$ is gradient-like for $\hat{f}$ and to complete the proof of \eqref{cot_lvf_2} it then suffices to show that $d\hat{f}(X_{\lambda}) > 0$ for $\vec{p}\neq 0$. 

Near $\{\vec{p}=0\}$, where $\eta \equiv 1$, we have $d\hat{f} = 2\vec{p}\, d\vec{p} + \ve\, df$ and 
\begin{align*}
    d\hat{f}(X_{\lambda}) &= 2 \vec{p}\, d\vec{p}(X_{\lambda}) + \ve \, df(X_{\lambda}) \\
    &= 2\norm{\vec{p}}^2 + 2\vec{p}\, d\vec{p}(-X_H) + \ve \, df(-X_H).
\end{align*}
In this region, $-X_H = X - \sum_{i,j=1}^n  p_i\frac{\partial X^i}{\partial q_j}\,\partial_{p_j}$. Thus, 
\begin{align*}
    d\hat{f}(X_{\lambda}) &= 2\sum_{i,j=1}^n \delta_{ij} p_ip_j - 2\sum_{i,j=1}^n p_ip_j\frac{\partial X^i}{\partial q_j} + \ve\, df(X) \\
    %&= \ve\, df(X) + 2\sum_{i,j=1}^n p_ip_j \left(\delta_{ij} - \frac{\partial X^i}{\partial q_j}\right) \\
    &= \ve\, df(X) + 2\ip{\vec{p}}{\left(I - \frac{\partial X}{\partial \vec{q}}\right)\vec{p}}.
\end{align*}
Away from zeroes of $X$, we may choose a coordinate system $\vec{q}$ on $L$ such that $\frac{\partial X}{\partial \vec{q}} = 0$, hence $d\hat{f}(X_{\lambda}) = \ve\, df(X) + \norm{\vec{p}}^2 > 0$ since $df(X) \geq 0$ by assumption. At zeroes of $X$, the assumption that the eigenvalues of $X$ have real part $< 1$ implies that the quadratic form  $\ip{\vec{p}}{\left(I - \frac{\partial X}{\partial \vec{q}}\right)\vec{p}}$ is positive definite, hence for $\vec{p}\neq 0$ we have $d\hat{f}(X_{\lambda}) > 0$. 

It remains to show $d\hat{f}(X_{\lambda}) > 0$ in the region where $d\eta \neq 0$. By assumption there is some $\delta > 0$ such that the support of $d\eta$ is contained in $\{\norm{\vec{p}} > \delta\}$. By the first equality in \eqref{eq:cotangent_bundle}, we have 
\[
d\hat{f}(X_{\lambda}) = 2\vec{p}\, d\vec{p}(X_{\lambda}) + \ve(f\, d\eta + \eta\, df)(X_{\lambda}) > 2\delta^2 + \ve(f\, d\eta + \eta\, df)(X_{\lambda}).
\]
We thus ensure $d\hat{f}(X_{\lambda}) > 0$ by taking $\ve > 0$ sufficiently small. 

Finally, for the ``in particular'' statement, note that if $f:L \to \R$ is Morse, then $\hat{f}(\vec{q}, \vec{p}) := \ve\eta(\vec{p})f(\vec{q}) + \norm{\vec{p}}^2$ as defined above is Morse. Indeed, the critical points coincide with those of $f$ along the $0$-section, where $\hat{f}$ agrees with $\ve f(\vec{q}) + \norm{\vec{p}}^2$, hence each critical point is nondegenerate and in fact has the same index as a critical point of $f$. By \eqref{cot_lvf_2}, it suffices to exhibit a gradient-like vector field $X$ on $L$ for $f$ whose eigenvalues at zeroes have real part $<1$. We can take $X = \nabla f$ and perform a local modification near zeroes if necessary. 
\end{proof}

Any Weinstein structure arising via \cref{lemma:cot_bund_str} is known as the \textit{standard} Weinstein structure on the cotangent bundle; there are other exotic Weinstein structures on cotangent bundles that are not even symplectomorphic to the standard one \cite[Theorem 4.7]{eliashberg2018flexiblelagrangians}. 

\subsubsection{Handlebody diagrams and Legendrian Kirby calculus} A regular level set $Y=\phi^{-1}(c)$ of a Weinstein cobordism naturally inherits a contact structure, and in fact, a contact form $\alpha = \lambda\mid_Y$. The following elementary fact is used implicitly throughout the article.

\begin{lemma}
Let $(Y, \ker\alpha)$ be a contact manifold and $(\R\times Y, \lambda=e^s\, \alpha)$ its symplectization. Let $\alpha'$ be any other contact form for $\ker \alpha$. Then there is a function $f:Y \to \R$ such that $\lambda\mid_{Y'} = \alpha'$, where $Y'=\{s=f\}$ is graphical in the symplectization over $Y$.      
\end{lemma}

As we work with completed Weinstein cobordisms, this lemma (with a corresponding adjustment of the Weinstein Morse function $\phi$ to make $Y'$ a regular level set) allows us to do contact topology in level sets at the level of structure, rather than form. This leads to contact topological diagrammatic presentations of Weinstein cobordisms; see \cite{gompf1998handlebody}, and \cite{rizell2024instabilitylegendrianknottednessnonregular} for a recent treatment. 

Kirby diagrams for Weinstein cobordisms involve only $1$-handles and $2$-handles. Weinstein $1$-handles may either be drawn in \textit{Gompf standard form} by their double-ball attaching region,\footnote{Here there is more subtlety than in the smooth case; we defer to \cite{gompf1998handlebody} for a thorough discussion.} or, by \cite{Ding2009HandleMI}, as a max-tb unknot with contact $(+1)$-framed (i.e. smooth $0$-framed) surgery. For consistency with smooth Kirby calculus, in the latter formulation we dot the $1$-handle circles, though this does not appear to be a common convention in the symplectic literature. See \cref{fig:1handle}. Weinstein $2$-handles are attached along Legendrians and have contact framing $(-1)$. 

\begin{figure}[ht]
	\begin{overpic}[scale=.33]{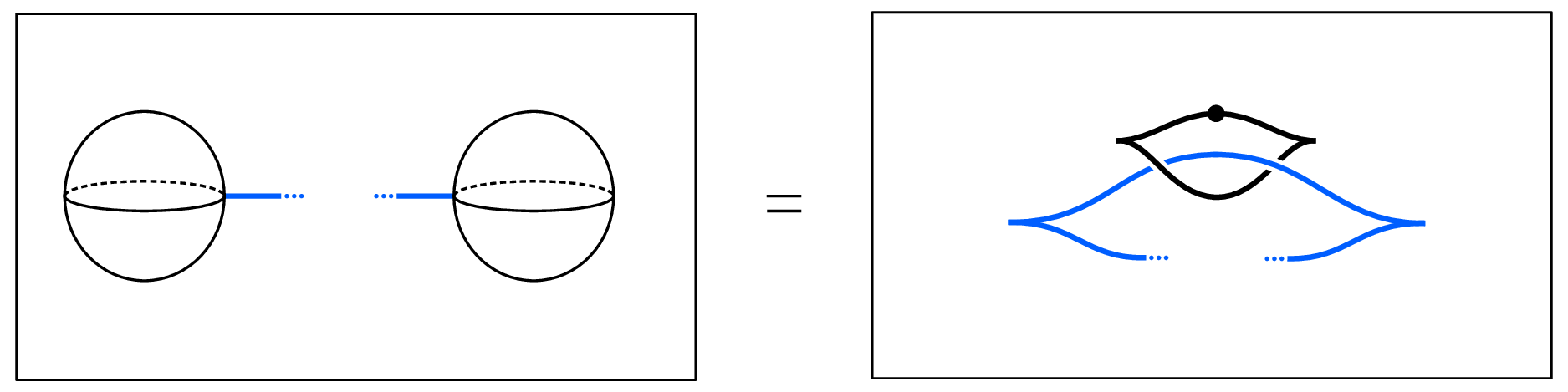}   
    \put(85,15.25){\scriptsize $(+1)$}
	\end{overpic}
	\caption{Weinstein $1$-handles.}
	\label{fig:1handle}
\end{figure}

Local models for handle slides over $(\pm 1)$-framed surgeries are given in \cref{fig:slidemodel}, and such handle slides preserve the contact-$(\pm 1)$ framing discussed above. Transforming a diagram by Legendrian isotopies, handle slides, and births/deaths of $1$-/$2$-handle pairs amounts to homotoping the underlying Weinstein structure. 

\begin{figure}[ht]
	\begin{overpic}[scale=.33]{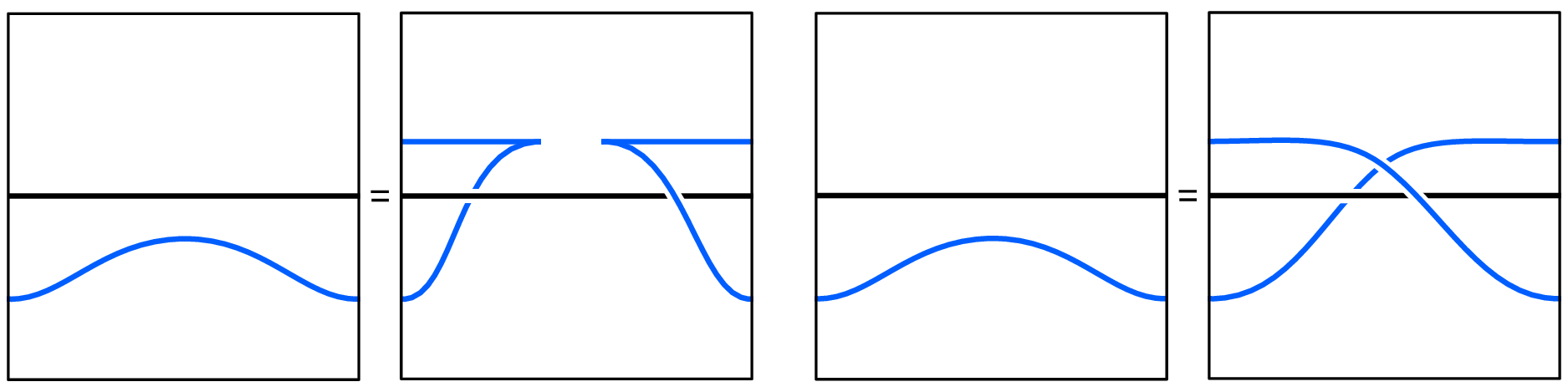}   
    \put(95.25,13.75){\tiny $(+1)$}
    \put(69.85,13.75){\tiny $(+1)$}

    \put(43.75,13.75){\tiny $(-1)$}
    \put(18.25,13.75){\tiny $(-1)$}
	\end{overpic}
	\caption{Legendrian handle slides.}
	\label{fig:slidemodel}
\end{figure}
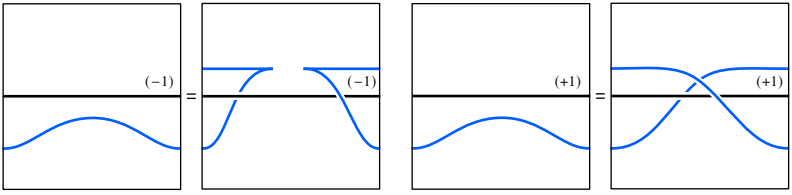

\subsection{Regular Lagrangians}\label{subsec:reg_review}

\begin{definition}\cite{eliashberg2018flexiblelagrangians}
A properly embedded Lagrangian cobordism $L\subset (W, \lambda, \phi)$ in a Weinstein cobordism is \textbf{regular} if there is a Weinstein homotopy $(\lambda, \phi) \rightsquigarrow (\lambda',\phi')$ such that $L$ is Lagrangian throughout the homotopy, and such that the Liouville vector field $X_{\lambda'}$ is everywhere tangent to $L$. 
\end{definition}

There are a number of useful properties and characterizations of regularity, but the most important for us is the fact that, informally, $L$ is regular if and only if the Weinstein cobordism admits a presentation as $T^*L \cup W_L$, where $W_L$ is a \textit{complementary} Weinstein cobordism. The following lemma gives a precise formulation. It is \cite[Lemma 2.2]{eliashberg2018flexiblelagrangians}, but rephrased to clarify the Weinstein structure in the case of nonempty negative ends.

\begin{lemma}\label{lemma:regular_cot_decomp}
Let $L\subset (W, \lambda, \phi)$ be a regular Lagrangian cobordism in a Weinstein cobordism. Assume $\partial_- W \neq \emptyset$ and let $\alpha = \lambda\mid_{\partial_- W}$. Then $(W, \lambda, \phi)$ is deformation equivalent to
\[
([-\ve, 0] \times \partial_- W) \cup \D^*L \cup W_L
\]
where $([-\ve, 0] \times \partial_- W, e^s\, \alpha, \phi)$ is a small symplectization, $(\D^*L, \lambda_{\mathrm{st}}, \phi_{\mathrm{st}})$ is the standard Weinstein structure (where $\D^*L$ is a Weinstein domain if $\partial_- L = \emptyset$, and $\D^*L$ is a Weinstein cobordism with corners if $\partial_- L \neq \emptyset$), and $(W_L, \lambda_L, \phi_L)$ is a Weinstein cobordism whose attaching locus is disjoint from $\partial_+ L$. Moreover, we identify $L$ with the $0$-section of $\D^*L$ when $\partial_- L = \emptyset$, and with the $0$-section of $\D^*L$ together with a collar $[-\ve, 0]\times \partial_- L \subset [-\ve, 0] \times \partial_- W$ when $\partial_- L \neq \emptyset$.
\end{lemma}

\begin{proof}
We refer to \cite{eliashberg2018flexiblelagrangians} for details but sketch the idea. By regularity, critical points of $\phi'\mid_L$ are in fact critical points of $\phi'$. Moreover, by birthing additional critical points, one can arrange for all critical points of $\phi'\mid_L$ to have the same index as critical points of $\phi'$. The values of the critical points along $L$ can further be adjusted so that all of the handles corresponding to $L$ are attached before any additional Weinstein handles. 

If $\partial_- L = \emptyset$, then the Weinstein handles corresponding to $L$ present $\D^*L$ as a Weinstein domain with $L$ as the $0$-section. We thus begin building the cobordism with a disjoint union of $\D^*L$ and a collar neighborhood of $\partial_- W$ with no critical values. If $\partial_- L \neq \emptyset$, then the symplectization $[-\ve, 0]\times \partial_- W$ contains a collar neighborhood $[-\ve, 0]\times \partial_- L$, and the Weinstein handles corresponding to $L$ attach $\D^*L$ as a cobordism with corners. In either case, the additional Weinstein handles give the cobordism $W_L$ and one can arrange by a general position argument for the attaching locus to be disjoint from $\partial_+ L$. 
\end{proof}

\subsection{Normal rulings}\label{subsec:NR_review}

Let $\Lambda\subset (\R^3, \ker (dz-y\, dx))$ be a Legendrian link. Assume that all cusps and crossings of the front projection $\Pi_{xz}(\Lambda)$ have distinct $x$-coordinates, which can be achieved by a generic perturbation. A \textit{normal ruling} of the front $\Pi_{xz}(\Lambda)$ is a subset of the crossings, called \textit{switched crossings}, or just \textit{switches}, satisfying the following properties. 
\begin{enumerate}
    \item Performing $0$-resolutions of all switched crossings yields a link projection with each component a max-tb unknot in its standard $2$-cusp projection. 

    \item Each switch is \textit{incident} to exactly two link components in the $0$-resolution; that is, each strand of the $0$-resolution near a switch belongs to a different component of the link.  

    \item If $\Lambda_1$ and $\Lambda_2$ are two such components incident to a switched crossing with $x$-coordinate $x_0$, then for $(x_0-\ve, x_0+\ve)$ the pre-resolution configuration of the paths in the front associated to $\Lambda_1$ and $\Lambda_2$ satisfies the \textit{normality} condition given by one of the top two left panels (and their horizontal reflections) of \cref{fig:rulingdef}. That is, near the switched crossing, the \textit{ruling disks} enclosed by $\Pi_{xz}(\Lambda_1)$ and $\Pi_{xz}(\Lambda_2)$ are either disjoint or nested.
\end{enumerate}

\begin{figure}[ht]
	\begin{overpic}[scale=.33]{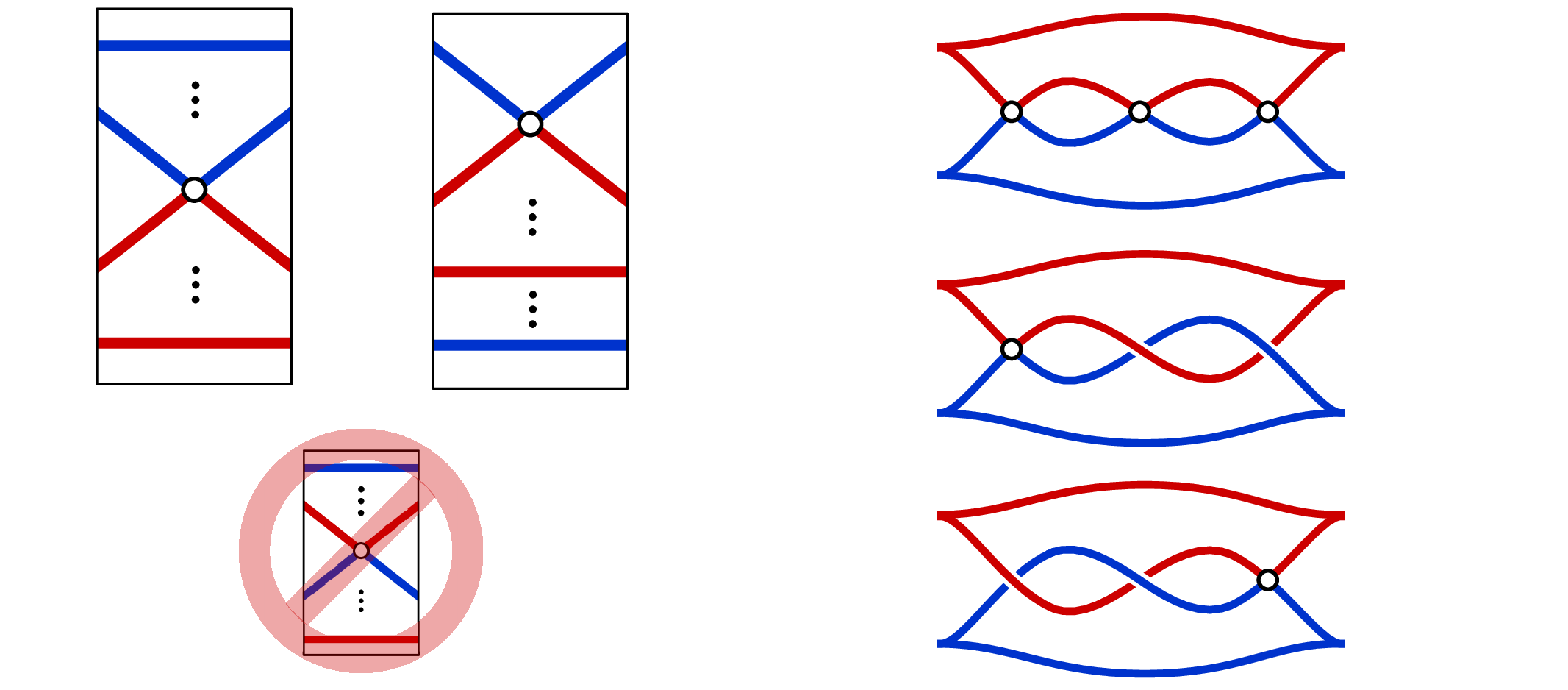}   
    
	\end{overpic}
	\caption{The normality condition near switched crossings on the left, and the three normal rulings of the standard front projection of the trefoil on the right.}
	\label{fig:rulingdef}
\end{figure}

We call the paths in the front that enclose a ruling disk \textit{companion paths} and say that they are \textit{paired} by the ruling, and, although it is unnecessary, we color companion paths when visualizing a ruling. We will abuse language and refer to a normal ruling of $\Lambda$, rather than of its front projection. 

An alternative way to describe a normal ruling is as a fixed-point free involution on sections of a front projection. We only adopt this perspective in \cref{sec:sat} when we consider pattern links of $2k$-many parallel strands; in this case, a normal ruling is simply a pairing of the strands via the involution. For details on this perspective we refer to \cite{chekanov2007pushkar}. 

Not every Legendrian admits a normal ruling. By \cite{fuchs2003rulings,fuchs2004invariants,sabloff2005augmentations}, existence of a normal ruling of a knot is equivalent to existence of an augmentation of the Chekanov-Eliashberg DGA; in particular, every fillable Legendrian admits a normal ruling. Counts of rulings are a Legendrian isotopy invariant, and as the $2$-cusp projection of the max-tb unknot has a unique normal ruling, the same can be said of any of its front projections. Finally, if $\Lambda_-$ admits a normal ruling and $\Lambda_+$ is obtained by one of the decomposable moves (unknot birth or Legendrian surgery), the ruling of $\Lambda_-$ extends to a ruling of $\Lambda_+$ in a unique way.

\subsection{Convex surface theory}\label{subsec:CHT_review}

Convex surface theory in contact topology was first introduced by Giroux \cite{giroux1991convexite}.

\begin{definition}
A surface $\Sigma\subset (Y, \xi)$, either closed or with Legendrian boundary, is \textbf{convex} if there is a contact vector field $X$, i.e. a vector field satisfying $\mathcal{L}_X\xi = \xi$, everywhere transverse to $\Sigma$. The \textbf{dividing set} is the embedded $1$-manifold in $\Sigma$ given by $\Gamma:=\{X \in \xi\}$.    
\end{definition}

Although the dividing set depends on the choice of contact vector field $X$, its isotopy class in $\Sigma$ does not. In fact, $\Gamma\subset (Y, \xi)$ is transverse to the contact structure, hence transverse to the characteristic foliation of $\Sigma$, and its transverse isotopy class in $\Sigma$ is well-defined. If $\alpha$ is any contact form for $\xi$, then $\Gamma = \{\alpha(X) = 0\}$ and $\Gamma$ splits $\Sigma$ into two regions $R_{\pm} = \{\pm \alpha(X) > 0\}$. 

\begin{theorem}\cite{giroux1991convexite,honda2000classification}
Let $\Sigma\subset (S^3, \xi)$ be an orientable surface with Legendrian boundary satisfying $\mathrm{tb}(\Lambda) \leq 0$. There is a perturbation of $\Sigma$ relative to $\partial \Sigma$, $C^0$-small near $\partial \Sigma$ and $C^{\infty}$-small in the interior, such that the resulting surface is convex. Moreover, $\mathrm{tb}(\Lambda) = -\frac{1}{2}|\Gamma \cap \Lambda|$. 
\end{theorem}

\begin{example}
Let $U\subset (Y, \xi)$ be a max-tb unknot in a tight contact manifold. Then any convex Seifert disk bounding $U$ has one dividing curve. Existence of closed embedded dividing circles is obstructed by Giroux's criterion for tightness \cite[Theorem 3.5]{honda2000classification}.   
\end{example}

Recall that there is a support-respecting correspondence between contact vector fields and smooth contact Hamiltonian functions \cite{geiges2008introduction}. Fixing a contact form $\alpha$ and given a contact vector field $X$, the corresponding smooth function is $H := \alpha(X)$. Conversely, given $H:Y \to \R$, the corresponding contact vector field is $X = H\, R_{\alpha} + Z$ where $Z\in \xi$ is the unique vector field such that $\iota_Zd\alpha\vert_{\xi} = -dH\vert_{\xi}$. This gives a significant amount of control over neighborhoods of convex surfaces. In particular, the following lemma says that every convex surface has an infinite $\R$-invariant neighborhood $\Sigma \times \R$ inside any arbitrarily small neighborhood in $Y$. 

\begin{lemma}\label{lemma:nbd_size}
Let $\Sigma \subset (Y, \xi)$ be a convex surface, possibly with boundary. Let $\mathcal{U} \subset Y$ be any open neighborhood of $\Sigma$. There is a contact embedding 
\[
(\Sigma \times \R_t, \ker \alpha_{\mathrm{inv}})\hookrightarrow (\mathcal{U}, \xi)
\]
mapping $\Sigma\times \{0\}$ to $\Sigma \subset \mathcal{U}$, such that 
\begin{enumerate}
    \item $\partial_t$ is a strict contact vector field, so that $\Sigma\times \{0\} \subset \Sigma \times \R$ is convex and the form $\alpha_{\mathrm{inv}}$ is $\R$-invariant, and 
    \item the dividing set of $\Sigma \times \{0\}$ is mapped to the dividing set of $\Sigma\subset \mathcal{U}$. 
\end{enumerate}
\end{lemma}

\begin{proof}[Proof idea.]
Since $\Sigma$ is convex, there is a contact vector field transverse to $\Sigma$. Let $H:Y \to \R$ be the corresponding contact Hamiltonian function. Multiply $H$ by a bump function identically $1$ near $\Sigma$ and compactly supported in $\mathcal{U}$. The flow of the resulting contact vector field then gives the desired embedding.      
\end{proof}

We can further localize the contact form near the dividing set.

\begin{lemma}\cite{giroux1991convexite}\label{lemma:dividingsetnormal}
Let $\Sigma \subset (Y, \xi)$ be a convex hypersurface with dividing set $\Gamma$. Let $\iota:\Gamma \to Y$ be the inclusion. After a contact isotopy rel $\Sigma$, there are coordinates on a neighborhood $N(\Gamma) \cong (-\ve, \ve)_{\tau} \times \Gamma$ and a contact form $\alpha$ on $N(\Gamma) \times \R_t$ such that 
\begin{equation}
    \alpha = -\tau\, dt + \alpha_{\Gamma}
\end{equation}
where $\alpha_{\Gamma}:= \iota^*\alpha$ is the induced contact form on $(\Gamma, \xi_{\Gamma})$. 
\end{lemma}

\begin{remark}
    Another consequence of the contact Hamiltonian vector field correspondence is the following observation. If $\Sigma$ is convex and $X$ is a contact vector field witnessing its convexity associated to a contact Hamiltonian function $H$, then the dividing set is precisely the $0$-locus of $H\vert_{\Sigma}$. 
\end{remark}

The \textit{Legendrian realization principle}, first discovered by Kanda \cite{kanda1998legendrian} and proved in full generality by Honda \cite{honda2000classification}, allows one to witness certain curves and arcs embedded in convex surfaces as Legendrians. We state the version we need below, which is not fully general but sufficient for our purposes.

\begin{theorem}[Legendrian realization principle]\label{thm:LeRP}
Let $\Sigma \subset (Y, \xi)$ be a convex surface (closed, or with Legendrian boundary) with dividing set $\Gamma$. Let $C\subset \Sigma$ be an embedded closed curve such that $C\pitchfork \Gamma$ and $C\cap \Gamma \neq \emptyset$. There is an isotopy from the identity $\phi_t:\Sigma \to \mathcal{U}\subset Y$, $t\in [0,1]$, supported in an arbitrarily small invariant neighborhood of $\Sigma$ and constant outside of an arbitrarily small neighborhood of $C\subset \Sigma$, such that
\begin{enumerate}
    \item $\phi_t(\Sigma)$ is graphical (in particular convex) in the invariant neighborhood, and 
    \item $\phi_1(C)$ is Legendrian.
\end{enumerate}
\end{theorem}

\section{Regularly slice implies once-stably decomposably slice}\label{sec:slice}

Here we prove \cref{thm:main3}. The goal is to upgrade the smooth proof of \cref{prop:reg_smth_ribbon} given in \cref{sec:gradations} to the symplectic setting. Throughout we let $(Y, \xi) = (\#^n S^1\times S^2, \xi_{\mathrm{st}})$.

\subsection{Tracing isotopies through intersections with surgeries}

The key mechanism in the proof of \cref{prop:reg_smth_ribbon} was the ability to effectively isotope $K$ past obstructing strands associated to $2$-handle attaching spheres via meridian curves and band surgeries. By a suitable choice of coordinates and illegal isotopy, we may localize all the obstructing intersections to lie sequentially on a single path. The necessary composite maneuver is then given by \cref{fig:stabsequencesmooth}.

\begin{figure}[ht]
	\begin{overpic}[scale=.3]{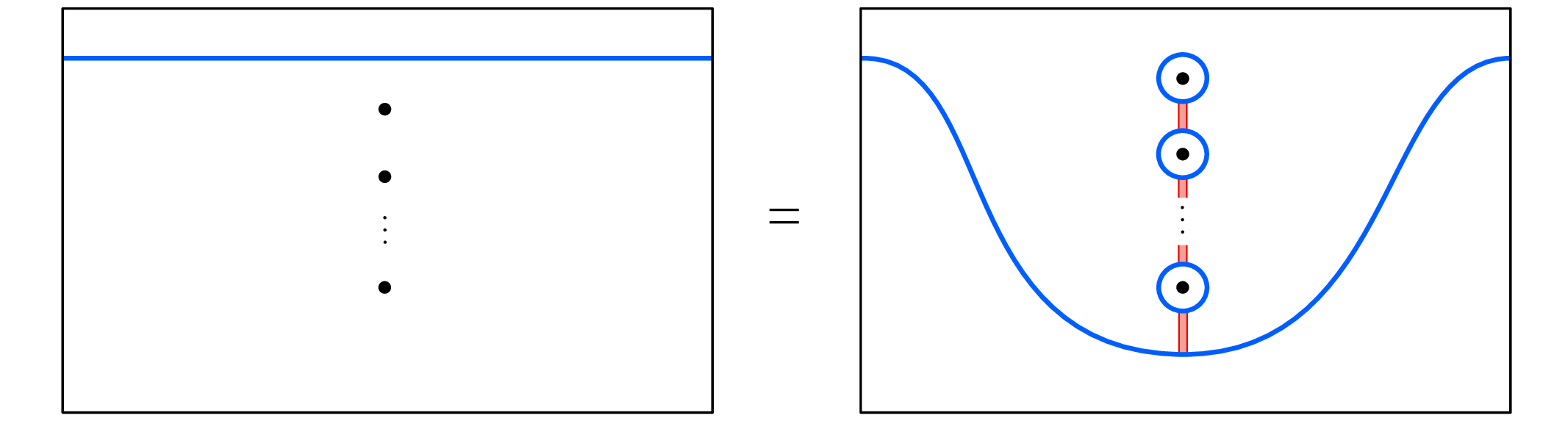}

	\end{overpic}
	\caption{The band surgeries on the right produce a strand isotopic to the strand on the left in the complement of the preimage of the marked points under the diagram projection.}
	\label{fig:stabsequencesmooth}
\end{figure}

The following lemma is the analogue in the Legendrian setting, where the intersections with obstructing $2$-handle strands are normalized to lie along a single Reeb chord. Here is precisely where we make use of the single stabilization. 

\begin{figure}[ht]
	\begin{overpic}[scale=.345]{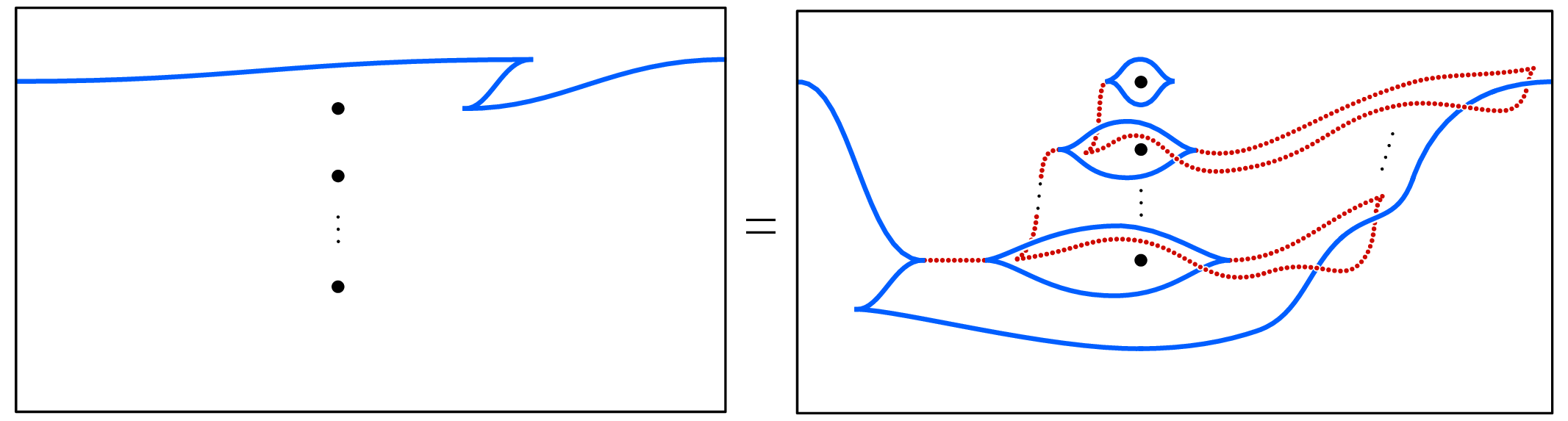}

    \put(7,23.25){\footnotesize \textcolor{lightblue}{$U^{\mathrm{Z}}$}}
    \put(18,19.75){\footnotesize $p_1$}
    \put(18,15.6){\footnotesize $p_2$}
    \put(18,8.45){\footnotesize $p_k$}

     \put(57,3.5){\footnotesize \textcolor{lightblue}{$U_0^{\mathrm{Z}}$}}
     \put(76,12){\tiny \textcolor{lightblue}{$U_k$}}
     \put(75,18.75){\tiny \textcolor{lightblue}{$U_2$}}
     \put(74.4,23){\tiny \textcolor{lightblue}{$U_1$}}

     \put(60,11.5){\tiny \textcolor{darkred}{$\gamma_k$}}
     \put(87,8.5){\tiny \textcolor{darkred}{$\gamma_{k-1}$}}
     \put(64.25,16.95){\tiny \textcolor{darkred}{$\gamma_2$}}
     \put(68,21.25){\tiny \textcolor{darkred}{$\gamma_1$}}
        
	\end{overpic}
	\caption{The statement of \cref{lemma:stabsequence}.}
	\label{fig:stabsequence}
\end{figure}

\begin{lemma}\label{lemma:stabsequence}
Let $U^{\mathrm{Z}}$ be the stabilized Legendrian strand on the left side of \cref{fig:stabsequence}, and $p_1,\dots, p_k$ the marked points. Let $\tilde{U} = U_0^{\mathrm{Z}} \cup U_1 \cup \cdots \cup U_k$ be the link on the right side of the figure and $G = \{\gamma_1, \dots, \gamma_k\}$ the set of surgery arcs in the same panel. Then $\mathrm{Surg}(\tilde{U}, G)$ is Legendrian isotopic to $U^{\mathrm{Z}}$ in the complement of $\Pi_{xz}^{-1}\{p_1, \dots, p_k\}$, where $\Pi_{xz}$ denotes the front projection.  
\end{lemma}

\begin{proof}
The proof is contained in \cref{fig:stabsequenceproof}. The left panel gives $\mathrm{Surg}(\tilde{U}, G)$, and the top right panel performs the first sequence of Reidemeister moves to isotope the strand up past $p_k$. The passage from the second to the third panel is the repetition of these moves through each point until $p_1$. Finally, inspection of this third panel reveals that the resulting strand is $U^{\mathrm{Z}}$.     
\end{proof}

\begin{figure}[ht]
	\begin{overpic}[scale=.345]{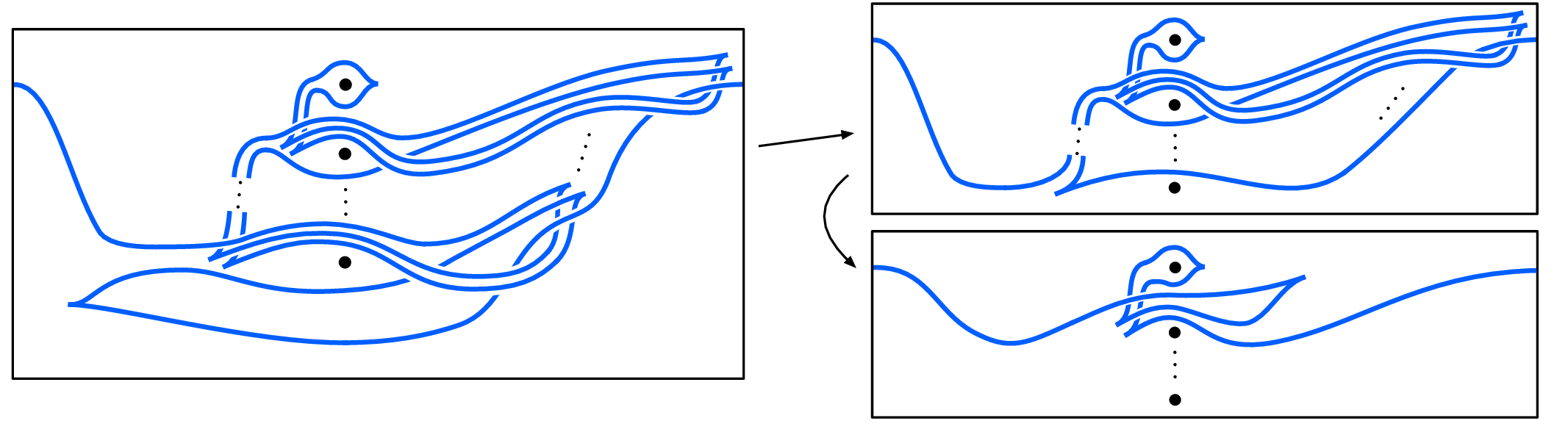}

    \put(5,20){\footnotesize \textcolor{lightblue}{$\mathrm{Surg}(\tilde{U}, G)$}}
        
	\end{overpic}
	\caption{The proof of \cref{lemma:stabsequence}.}
	\label{fig:stabsequenceproof}
\end{figure}

\subsection{Normalizing a convex Seifert disk}

In order to use \cref{lemma:stabsequence}, we need to normalize the Legendrian isotopy of an unknot in a Weinstein diagram so that all obstructing $2$-handle attaching spheres appear in a model as in \cref{fig:stabsequence}. We do so using convex surface theory applied to a Seifert disk. 

\begin{lemma}\label{lemma:seifert}
Let $\Lambda \subset (Y, \xi)$ be a Legendrian link. Let $U\subset Y$ be a max-tb unknot which is disjoint, but not necessarily unlinked, from $\Lambda$. There is a convex Seifert disk $\Sigma$ for $U$ and a choice of transverse contact vector field $X\in \mathfrak{X}(Y)$ witnessing its convexity which is everywhere tangent to $\Lambda$.
\end{lemma}

\begin{proof}
Let $\Sigma$ be a Seifert disk for $U$. By general position we may assume that $\Lambda$ intersects $\Sigma$ transversally in its interior. Since $\mathrm{tb}(U) = -1 \leq 0$, we perform a $C^0$-small perturbation of $\Sigma$ near and relative to $U$, followed by a $C^{\infty}$-small perturbation of the interior, so that $\Sigma$ is convex. We may moreover assume that this perturbation preserves transversality with $\Lambda$. 

Convexity of $\Sigma$ means that there exists some contact vector field $X_0$ everywhere transverse to $\Sigma$. However, to prove the lemma we need to construct such a vector field which is everywhere tangent to $\Lambda$. We begin by constructing a partially defined contact vector field in a neighborhood of $\Lambda$.  

Using the flow of $X_0$, we first identify a small vertically invariant one-sided collar neighborhood $N(\Sigma) \cong [-\ve, 0]_t \times \Sigma$, where $\Sigma = \{0\} \times \Sigma$ and $X_0 = \partial_t$, such that each component of $\Lambda \cap N(\Sigma)$ is a Legendrian strand with one endpoint on $\{-\ve\}\times \Sigma$ and the other on $\{0\}\times \Sigma$. Let $\Lambda_j$ be a component of $\Lambda$. View each arc of $\Lambda_j$ lying outside $N(\Sigma)$ as the Legendrian core of a contact $1$-handle. Likewise, view a $J^1$-neighborhood of each arc of $\Lambda_j$ lying inside $N(\Sigma)$ as a contact $0$-handle for which $\Lambda_j$ is tangent to the ascending manifold. With contact $0$-handles and $1$-handles defined accordingly, we get a contact vector field $X$ defined on $J^1(\Lambda)$ which is positively transverse to $\Sigma$; see the red vector field in \cref{fig:seifertdisk}. 

\begin{figure}[ht]
	\begin{overpic}[scale=.347]{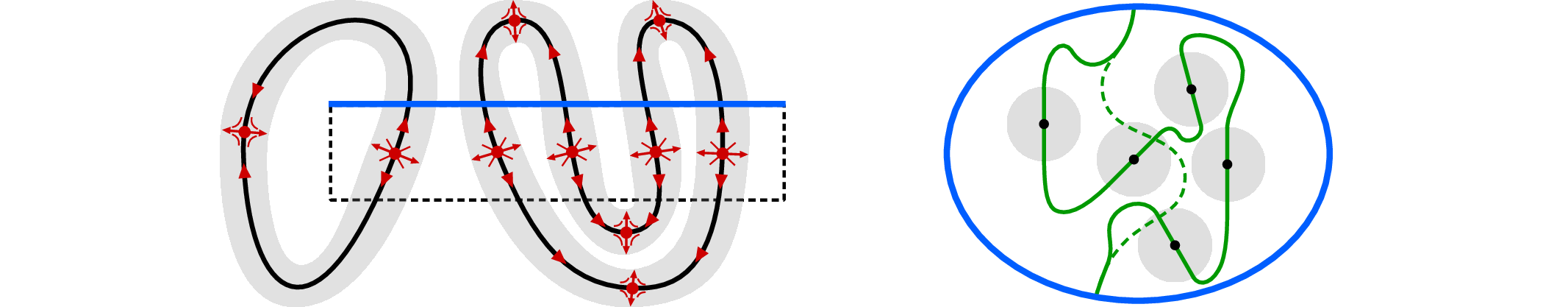}
        \put(65,4.5){\footnotesize \textcolor{darkgreen}{$\Gamma$}}
        \put(71,13){\tiny \textcolor{darkgreen}{$\Gamma_0$}}

        \put(50.5,12.5){\tiny \textcolor{lightblue}{$\{0\}\times \Sigma$}}
        \put(48.5,16.5){\footnotesize \textcolor{darkred}{$X$}}
        \put(83.5,16){\footnotesize \textcolor{lightblue}{$U$}}
	\end{overpic}
	\caption{The shaded region on the left is $J^1(\Lambda)$; on the right, the shaded disks are $D_{\Lambda}$.}
	\label{fig:seifertdisk}
\end{figure}
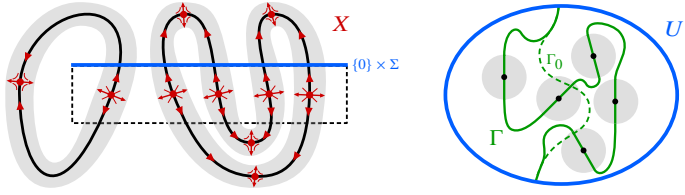

Next we must extend $X$ to be defined near the rest of $\Sigma$. We take $X = X_0$ near $U$, needing to interpolate appropriately across the remaining portion of $\Sigma$. Let $D_{\Lambda} = J^1(\Lambda) \cap \Sigma$, and let $\dot{\Sigma} = \Sigma - (D_{\Lambda} \cup \mathrm{Op}(U))$, where $\mathrm{Op}(U)$ denotes the neighborhood of $U$ on which we take $X = X_0$. 

Because $D_{\Lambda}$ is convex with respect to $X$ and $X$ is tangent to $\Lambda$ (which is Legendrian, hence tangent to $\xi$), each disk component of $D_{\Lambda}$ admits dividing curves.\footnote{Note that the boundary of $D_{\Lambda}$ is not necessarily Legendrian. This is not a problem, but such a surface may a priori have no dividing set.} By shrinking $D_{\Lambda}$ if necessary we may assume the dividing curves are transverse to $\partial D_{\Lambda}$. Let $\Gamma\subset \Sigma$ denote a smoothly embedded arc that
\begin{enumerate}
    \item agrees with the dividing curves of $D_{\Lambda}$, 
    \item agrees with the dividing curve $\Gamma_0$ of $\Sigma$ with respect to $X_0$ in $\mathrm{Op}(U)$, and 
    \item is smoothly isotopic in $\Sigma$ to $\Gamma_0$. 
\end{enumerate}
Fix a contact form $\alpha$. Choose a smooth function $H: \mathrm{Op}(\Sigma) \cup J^1(\Lambda) \to \R$ that agrees with $\alpha(X)$ on $J^1(\Lambda)$ and $\mathrm{Op}(U)$, such that $H\vert_{\Sigma}$ vanishes precisely along $\Gamma$. By the contact Hamiltonian vector field correspondence, the associated contact vector field (which we continue to call $X$) defined on $\mathrm{Op}(\Sigma) \cup J^1(\Lambda)$ is everywhere tangent to $\Lambda$ and transverse to $\Sigma$. We then extend $H$ to all of $Y$ in an arbitrary way (for instance, we can smoothly decay it to $0$ outside of $\mathrm{Op}(\Sigma) \cup J^1(\Lambda)$) to obtain the desired contact vector field $X$. 
\end{proof}

\subsection{Regularly-slice implies once-stably decomposably slice}

We begin with a lemma that summarizes the necessary technical work.

\begin{lemma}\label{lemma:mainlemma}
Let $\Lambda \subset (Y, \xi)$ be a Legendrian link. Let $U\subset Y$ be a max-tb unknot which is disjoint, but not necessarily unlinked, from $\Lambda$. There is a link $\tilde{U} = U_0^{\mathrm{Z}} \cup U_1, \dots, U_k \subset Y - \Lambda$ and a set of embedded surgery arcs $G = \{\gamma_1, \dots, \gamma_k\}\subset Y - \Lambda$ for $\tilde{U}$ with the following properties.
\begin{enumerate}
    \item Each component $U_1, \dots, U_k$ is a max-tb unknot in $Y$.
    \item The component $U_0^{\mathrm{Z}}$ is a $\mathrm{tb}=-2$ unknot in $Y-\Lambda$. 
    \item The knot $\mathrm{Surg}(\tilde{U}, G)$ is Legendrian isotopic to $S_{\pm}(U)$ in $Y - \Lambda$.
\end{enumerate}
Moreover, if $S_1, \dots, S_n$ is a choice of belt spheres for $Y$, we may assume that $\tilde{U} \cap S_j = \emptyset$ for each $j=1, \dots, n$. 
\end{lemma}

\begin{proof}
By \cref{lemma:seifert}, we may choose a convex Seifert disk for $U$ in $Y$ together with a convexity witnessing contact vector field $X$ everywhere tangent to $\Lambda$. By integrating the contact vector field we obtain a $t$-invariant neighborhood $(\R_t \times D^2, \alpha)$ of the Seifert disk $\{t=0\}$ such that the dividing set $\Gamma$ is a Reeb chord of $\alpha$, and 
\[
\Lambda \cap (\R\times D^2) = \R \times \{p_1, \dots, p_k\}
\]
where $p_1, \dots, p_k\in \Gamma$ is a set of points lying sequentially in their labeled order along the dividing set of the Seifert disk.

Let $q_{k+1}$ be an arbitrary point on $\Gamma$ past $p_k$ in the sequential ordering of $p_1, \dots, p_k$. Let $U_0$ be a small circle embedded in the Seifert disk enclosing $q_{k+1}$ and none of the other labeled points; see \cref{fig:diskmodel2}. By the Legendrian realization principle, after a graphical isotopy of $\Sigma$ supported near $U_0$, we may assume that $U_0$ is Legendrian. In particular, $U_0$ is a max-tb unknot, unlinked from $\Lambda$.

\begin{figure}[ht]
	\begin{overpic}[scale=.347]{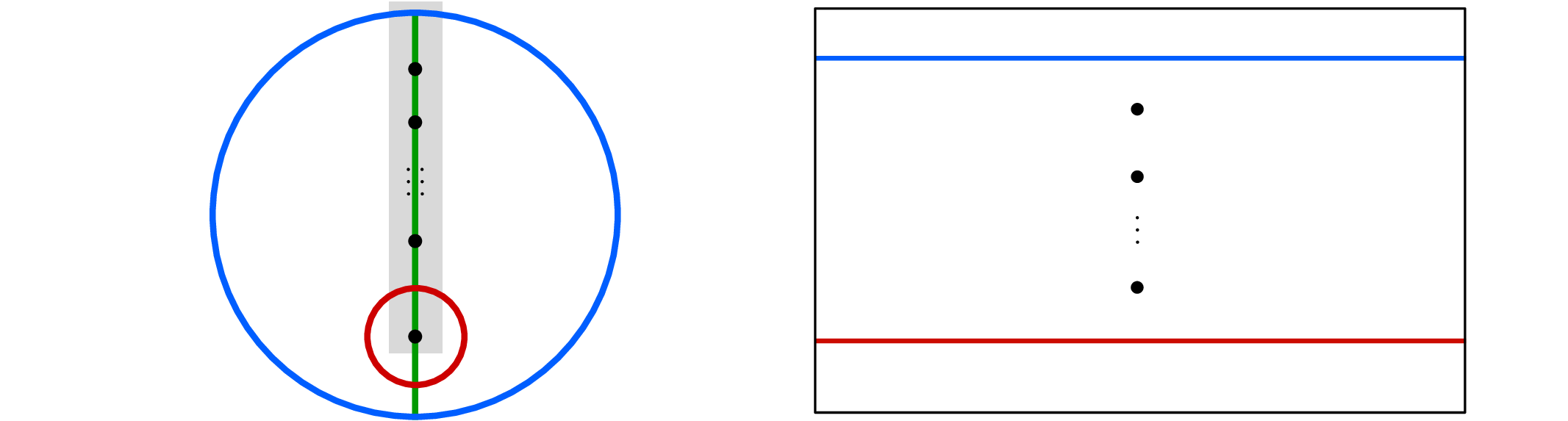}

    \put(15,23.25){\footnotesize \textcolor{lightblue}{$U$}}
    \put(23.5,22.5){\tiny $p_1$}
    \put(23.5,18.8){\tiny $p_2$}
    \put(23.5,11.45){\tiny $p_k$}
    \put(19.5,4.95){\tiny $q_{k+1}$}
    \put(30.75,4.75){\footnotesize \textcolor{darkred}{$U_0$}}

     \put(57,2.5){\footnotesize \textcolor{darkred}{$U_0$}}
     \put(57,20.5){\footnotesize \textcolor{lightblue}{$U$}}
     \put(69.5,19.75){\tiny $p_1$}
    \put(69.5,15.6){\tiny $p_2$}
    \put(69.5,8.45){\tiny $p_k$}

	\end{overpic}
	\caption{The model neighborhood in the proof of \cref{lemma:mainlemma}. The right side is the front projection (after an additional contactomorphism) of the shaded region in the convex model on the left.}
	\label{fig:diskmodel2}
\end{figure}
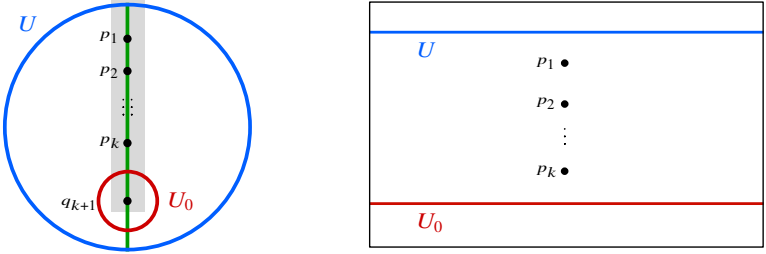

Next, let $\Gamma'\subset \Gamma$ be the subarc of $\Gamma$ connecting $q_{k+1}$ and $U$ that passes through the points $p_1, \dots, p_k$. Let $N(\Gamma')\subset D^2$ be a small tubular neighborhood of $\Gamma'$ in the Seifert disk, and finally let $\mathcal{U}:= \R\times N(\Gamma')$. By \cref{lemma:dividingsetnormal}, after shrinking $N(\Gamma)$ if necessary we may assume that $\R\times N(\Gamma')$ is contactomorphic to $(\R_t \times (-\ve, \ve)_{\tau} \times \Gamma_z, dz - \tau\, dt)$. After pulling back this model with the strict contactomorphism
\begin{align*}
(\R^3_{x,y,z}, dz - y\, dx) \, &\to \, (\R^3_{t,\tau,z}, dz - \tau\, dt) \\    
(x,y,z)\, &\mapsto \, (-y,x, z-xy)
\end{align*}
the neighborhood admits an $(x,z)$-front projection (with infinite thickness in the $y$-fiber direction) as depicted on the right side of \cref{fig:diskmodel2}. 

Now we are prepared to apply \cref{lemma:stabsequence} and complete the proof. Stabilizing $U$ once to $S_{\pm}(U)$ and applying the lemma produces a $\mathrm{tb}=-2$ unknot $U_0^{\mathrm{Z}}$, a set of max-tb unknots $U_1, \dots, U_k\subset Y$, and a set of surgery arcs $G = \{\gamma_1, \dots, \gamma_k\}$ such that $\mathrm{Surg}(\tilde{U},G)$ is Legendrian isotopic to $S_{\pm}(U)$ in $Y- \Lambda$. By construction of the model, $U_0^{\mathrm{Z}}$ is smoothly isotopic to $U_0$ in $Y-\Lambda$, hence $U_0^{\mathrm{Z}}$ is unknotted in $Y-\Lambda$. Finally, by choosing $U_0, U_1, \dots, U_k$ sufficiently small and performing a Legendrian isotopy of $U_1, \dots, U_k$ along $\Lambda$ if necessary, we may ensure that the link $\tilde{U}$ is disjoint from the prescribed belt spheres $S_1, \dots, S_n$. 
\end{proof}

\begin{proof}[Proof of \cref{thm:main3}.]
By \cite[Theorem 1.10]{conway2017symplectic}, we may present any regularly slice knot as the max-tb unknot $U$ in the boundary of a Weinstein handlebody diagram which is homotopic to the symplectization. This latter assumption provides a sequence of Legendrian isotopies, handle slides of attaching spheres, and $1$-/$2$-handle births to a Weinstein handlebody diagram in geometrically canceling position; we intentionally delay all handle cancellations.  

By contact isotopy extension, any Legendrian isotopies of attaching spheres induce a Legendrian isotopy of $U$. Moreover, any requisite handle births and handle slides may generically be performed in the complement of $U$; indeed, a Legendrian handle slide is supported in a tubular neighborhood of the attaching spheres in question and a short Reeb chord between them. Therefore, we may assume that we have a Weinstein handlebody diagram in geometry canceling position together with a complicated diagram of $U$ which is disjoint from all attaching spheres. Moreover, $U$ is a max-tb unknot in the copy of $(Y, \xi) = (\#^n S^1\times S^2, \xi_{\mathrm{st}})$ obtained after erasing all $2$-handle attaching spheres. 

Now we apply \cref{lemma:mainlemma} with $\Lambda$ the attaching locus of all Weinstein $2$-handles. This gives us a link $\tilde{U} = U_0^{\mathrm{Z}}, U_1, \dots, U_k \subset Y-\Lambda$, disjoint from the belt spheres of the $1$-handles, where $U_0^{\mathrm{Z}}$ is a $\mathrm{tb}=-2$ unknot in $Y-\Lambda$, $U_1, \dots, U_k$ are max-tb unknots in $Y$, and $G=\{G_1, \dots, G_k\}$ is a set of surgery arcs so that $\mathrm{Surg}(\tilde{U}, G)$ is Legendrian isotopic to $S_{\pm}(U)$ in $Y-\Lambda$. In particular, by filling each $U_1, \dots, U_k$ with their standard disk fillings, $\mathrm{Surg}(\tilde{U}, G)$ gives a (strong) decomposable Lagrangian concordance in the symplectization of $Y$ from $U_0^{\mathrm{Z}}$ to $S_{\pm}(U)$.

Although the base link $\tilde{U}$ does not intersect the belt spheres of the $1$-handles, the surgery arcs in $G$ may. For each such intersection, slide the surgery arc over the canceling $2$-handle attaching sphere. Once the $1$-handles are clear of intersections with surgery arcs, the canceling pairs can be erased. Left in the wake is a Legendrian surgery presentation of a (strong) decomposable concordance from $U_0^{\mathrm{Z}}$, a $\mathrm{tb}=-2$ unknot, to the once-stabilized original regularly slice knot.  
\end{proof}

\section{Diagrammatic presentation of regular cobordisms}\label{sec:diag}

Here we prove the extensions of \cite[Theorem 1.10]{conway2017symplectic} to cobordisms, knot fillings, and concordances.

\begin{proof}[Proof of \cref{thm:cob_diag}.]
By \cref{lemma:regular_cot_decomp}, there is a Weinstein deformation equivalence from the symplectization to 
\[
((-\infty, 0] \times S^3) \cup \D^*L \cup W_L
\]
where $((-\infty, 0] \times S^3, e^s\, \alpha, \phi)$ is a half-symplectization, $(\D^*L, \lambda_{\mathrm{st}}, \phi_{\mathrm{st}})$ is the Weinstein cobordism with corners furnished by \cref{lemma:cot_bund_str} attached to the half-symplectization along $\{0\}\times \Lambda_-$, and $(W_L, \lambda_L, \phi_L)$ is a Weinstein cobordism whose attaching locus is disjoint from the positive boundary of the $0$-section in $\D^*L$. Moreover, we identify $L$ with the $0$-section of $\D^*L$ glued onto the collar $(-\infty, 0]\times \Lambda_- \subset (-\infty, 0] \times S^3$.    

Since $L$ is a cobordism with nonempty positive and negative boundary, \cref{lemma:cot_bund_str} implies that we may homotope the Weinstein structure $(\D^*L, \lambda_{\mathrm{st}}, \phi_{\mathrm{st}})$ to have only critical points of index $1$, and additionally that the (un)stable manifolds of the critical points are contained in $L$. Let 
\[
(W_0, \lambda_0, \phi_0) := ((-\infty, 0] \times S^3) \cup \D^*L \cup (\textrm{$1$-handles from } W_L).
\]
By general position we may assume that the $1$-handles from $W_L$ are attached away from $\partial_+ L$. Let $L_0 \subset W_0$ denote the corresponding Lagrangian cobordism and $\Lambda_0 := \partial_+ L_0 \subset \partial W_0 = (\#^{n+k} S^1\times S^2, \xi_{\mathrm{st}})$ its positive boundary. 

The diagrammatic presentation of $W_0$ and $L_0$ constructed so far is as follows. First, the underlying Weinstein handlebody diagram presenting $W_0$ consists of ($n+k$)-many $1$-handles, $n$ of which are the handles from $\D^*L$ and $k$ of which are the $1$-handles from $W_L$. The former are coupled with Lagrangian $1$-handles attached to the cylinder $(-\infty, 0]\times \Lambda_-$, while the latter are attached away from $\Lambda_0$. Consequently, in the diagram, $\Lambda_0$ intersects each $1$-handle associated to $\D^*L$ twice geometrically and with algebraic intersection number $0$, avoiding the other $k$-many $1$-handles. It remains to attach all of the $2$-handles from $W_L$, which by general position is done in the complement of $\Lambda_0$. This gives the desired diagrammatic presentation as stated in the theorem.

Finally, the converse statement is clear, as any such cobordism is constructed by coupled Weinstein-Lagrangian handles, hence is regular by \cite{eliashberg2018flexiblelagrangians}.
\end{proof}

\begin{proof}[Proof of \cref{thm:fill_diag}.]
If $L$ is a regular genus $g\geq 1$ filling of a knot, the above proof provides a Weinstein deformation equivalence instead to $\D^*L \cup W_L$, where $\D^*L$ has a single critical point of index $0$ and $2g$ critical points of index $1$, and $W_L$ is a Weinstein cobordism whose attaching locus is disjoint from $\partial_+L$. Observe that attaching a $2$-handle to $\D^*L$ along $\partial_+ L$ produces $\D^*\Sigma_g$, where $\Sigma_g$ is the closed surface of genus $g$. By \cite{gompf1998handlebody}, $\D^*\Sigma_g$ admits a handlebody diagram with $2g$-many $1$-handles and a $2$-handle attached along the knot $\Lambda$ in \cref{fig:filldiag}; thus, we obtain a diagram for $\D^*L$ with $\Lambda = \partial_+ L$ drawn by simply erasing the $(-1)$ surgery coefficient. Subsequently attaching $W_L$ amounts to attaching some additional $1$-handles and $2$-handles, all disjoint from $\Lambda$. The converse statement is again clear, as Gompf's standard handlebody diagram presents $\D^*L$. 
\end{proof}

\begin{proof}[Proof of \cref{thm:conc_diag}.]
If $L$ is a regular concordance from $\Lambda_-$ to $\Lambda_+$, then \cref{lemma:cot_bund_str} allows us to Weinstein homotope away all critical points on $L$. The result is a deformation equivalence with 
\[
((-\infty, 0] \times S^3) \cup W_L
\]
where $((-\infty, 0] \times S^3, e^s\, \alpha, \phi)$ is a half-symplectization, $L$ is the cylinder $(-0,\infty]\times \Lambda_-$, and $W_L$ is a Weinstein cobordism, itself deformation equivalent to the symplectization, whose attaching locus is disjoint from $\partial_+ L$. The diagrammatic consequence is as stated in the corollary: $L$ is the trivial cylinder over $\Lambda_-\in S^3$, $W_L$ consists of Weinstein handles attached around $\Lambda_-$, and canceling these handles transforms $\Lambda_-$ into $\Lambda_+$. The converse statement is once again clear, as any such concordance is tangent to the symplectization Liouville vector field. 
\end{proof}
\section{Satellites and normal rulings}\label{sec:sat}

Here we prove the results stated in \cref{subsec:sat} and \cref{subsec:rulings} on satellites and normal rulings. 

\subsection{Satellites}

\begin{proof}[Proof of \cref{thm:main_sat}.]
By \cref{thm:conc_diag}, up to Weinstein deformation equivalence we may assume that the regular concordance $L$ is the trivial cylinder over $\Lambda_-$, where $\Lambda_-$ is a Legendrian knot in a Kirby diagram disjoint from all $1$-handles and $2$-handles. Satelliting the cobordism by the pattern knot $\Lambda'\subset J^1(S^1)$ is then witnessed by satelliting $\Lambda_-$ by $\Lambda'$ in the diagram and considering the resulting trivial cylinder. By the converse statement of \cref{thm:conc_diag}, the concordance is regular.     
\end{proof}

\begin{proof}[Proof of \cref{cor:neg_ans}.]
Let $\Lambda$ denote the standard Legendrian $\overline{9_{46}}$. As $\Lambda$ is decomposably slice, it is regularly slice by \cite{conway2017symplectic}. Therefore, there is a regular concordance $U \prec_{\mathrm{reg}} \Lambda$. The concordance of Cornwell, Ng, and Sivek arises by satelliting this concordance by the Whitehead pattern; see the right side of \cref{fig:filldiag}. By \cref{thm:main_sat}, the satellites are regularly concordant. 
\end{proof}

\begin{proof}[Proof of \cref{cor:EL_stab}.]
Let $L\subset \R_s \times S^3$ be a regular cobordism with nonempty positive and negative end. The Etnyre-Leverson stabilization operation first requires a choice of a smoothly embedded curve $c\subset L$ from $\partial_- L$ to $\partial_+ L$ which is cylindrical near $\{s = \pm \infty\}$. By regularity of $L$ and \cref{lemma:cot_bund_str}, we may assume that we have presented the cobordism in a regularizing Weinstein structure whose Liouville vector field is tangent to and nonvanishing along $c$. By \cref{thm:cob_diag}, the curve $c$ will correspond to a marked point on the link $\Lambda_0$ in the corresponding diagrammatic presentation. We stabilize the cobordism and confirm regularity of the result by simply stabilizing $\Lambda_0$ near the marked point and applying \cref{thm:cob_diag}.   
\end{proof}

\subsection{Normal rulings}

Next we prove \cref{thm:main_ruling} on normal rulings induced by regular cobordisms; the observation \cref{cor:normal_ruling_filling} follows immediately.

\begin{proof}[Proof of \cref{thm:main_ruling}.]
Let $L\subset \R\times S^3$ be a regular cobordism from $\Lambda_-$ to $\Lambda_+$. Assume that $\Lambda_-$ has a normal ruling. By \cref{thm:cob_diag}, regularity of $L$ allows us to construct $\Lambda_+$ from $\Lambda_-$ in two stages. First, we attach a number of Weinstein $1$-handles and perform Legendrian surgeries on $\Lambda_-$ along arcs passing through some subset of these handles to produce a link $\Lambda_0 \subset \#^n S^1\times S^2$. Second, we attach $2$-handles away from $\Lambda_0$, and perform the Kirby calculus necessary to trivialize the Weinstein structure, yielding $\Lambda_+$. 

Consider the first stage. Viewing $1$-handles as dotted contact-$(+1)$ surgeries on unknots, there is a map from front projections of links in $\#^n S^1\times S^2$ to front projections of links in $S^3$ obtained by erasing the dotted surgeries. Under this map, $\Lambda_0$ is obtained from $\Lambda_-$ by Legendrian surgeries, and as decomposable cobordisms induce canonical normal rulings it follows that $\Lambda_0$ inherits a ruling from $\Lambda_-$. Moreover, adding the dotted $1$-handles back into the picture, this normal ruling has the property that if a path crosses a $1$-handle, then its companion path does as well; see \cref{fig:NR1handle}. (In particular, $\Lambda_0$ has a normal ruling in $\#^n S^1\times S^2$ in the sense of \cite{leverson2017augmentationsconnectsum}. As we start and finish in $S^3$, we can work around this concept and can stick to normal rulings in $S^3$.)

\begin{figure}[ht]
    
	\begin{overpic}[scale=.357]{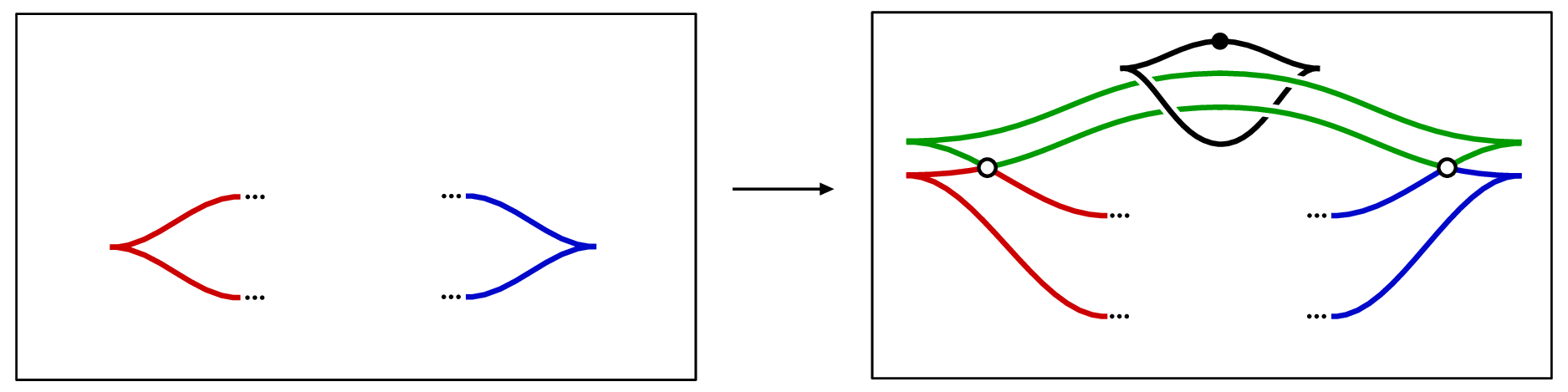}   
        
        \put(68,21.75){\tiny $(+1)$}
	\end{overpic}
	\caption{Attaching a coupled Weinstein-Lagrangian $1$-handle and extending the normal ruling of $\Lambda$ (left) to a normal ruling of $\Lambda_0$ (right).}
	\label{fig:NR1handle}
\end{figure}

Now we attach the Weinstein $2$-handles. As in the proof of \cref{thm:main3}, the assumption of regularity provides a sequence of Legendrian isotopies, handle slides, and handle births that converts the diagram into one in geometrically canceling position; we again intentionally delay any cancellations. By general position and contact isotopy extension, viewing $1$-handles as dotted circles and ``forgetting'' all surgeries in the diagram, this induces an isotopy of $\Lambda_0$ viewed as living in $S^3$.

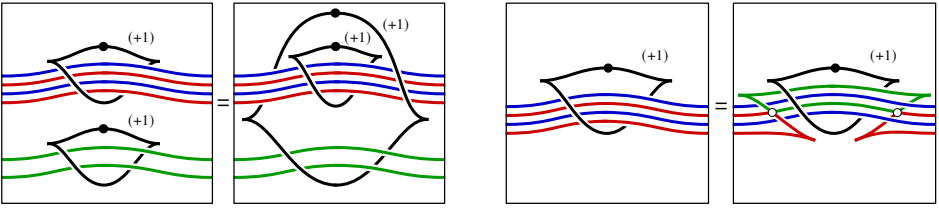
\begin{figure}[ht]
    
	\begin{overpic}[scale=.356]{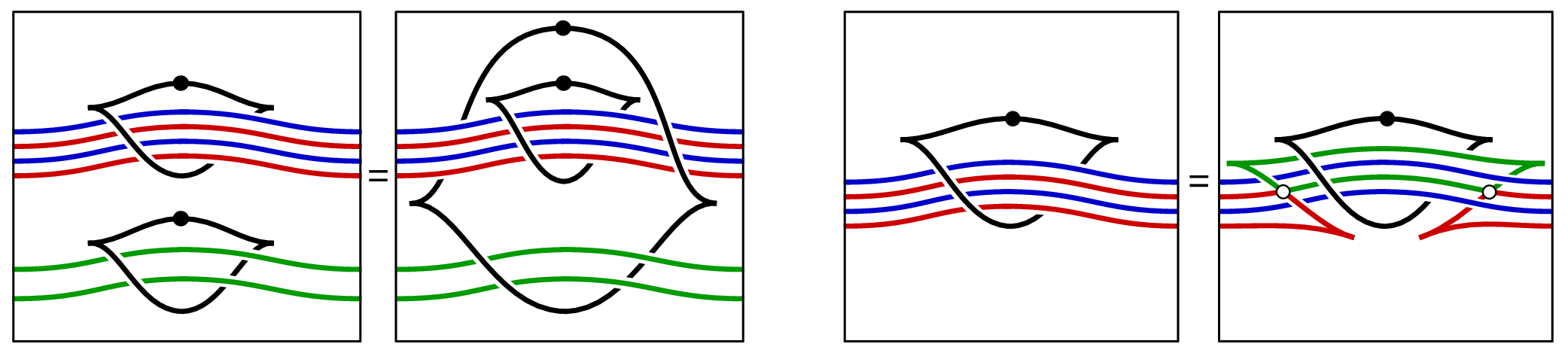}   
        
        \put(68,15.75){\tiny $(+1)$}
        \put(92,15.75){\tiny $(+1)$}

        \put(14,17.75){\tiny $(+1)$}
        \put(36.75,17.75){\tiny $(+1)$}
        \put(40.75,19){\tiny $(+1)$}
        \put(14,9){\tiny $(+1)$}
	\end{overpic}
	\caption{The left two panels depict a $1$-handle slide, using the model from \cref{fig:slidemodel} and a sequence of Reidemeister moves. The right two panels depict a slide of the bottom red strand of $\Lambda_0$ across the $1$-handle, and the extension of the normal ruling.}
	\label{fig:NR1handleslide}
\end{figure}

Note that any $1$-handle slides preserve the companion path property described above, along with slides of $\Lambda_0$ across $1$-handles (which correspond to isotopy in $\#^n S^1\times S^2$); see \cref{fig:NR1handleslide}. Therefore, we may assume that $\Lambda_0$, equipped with a $1$-handle-compatible normal ruling, is drawn in a Weinstein Kirby diagram in geometrically canceling position. By additional slides over the $1$-handles we may permute the Reeb-height of all Legendrian strands passing through $1$-handles so that, for each $1$-handle, the canceling $2$-handle strand has largest Reeb-coordinate and the necessarily even number of strands associated to $\Lambda_0$ lie below. 

It remains to slide $\Lambda_0$ across the $2$-handles sufficiently many times to make way for the handle cancellation. We must argue that the resulting link $\Lambda_+$, obtained after the handle slides, inherits a normal ruling from $\Lambda_0$. We may argue one $1$-handle at the time, so we assume without loss of generality that the Kirby diagram contains one $1$-handle and one geometrically canceling $2$-handle attached along a knot $\Lambda$. We moreover assume that all strands of $\Lambda \cup \Lambda_0$ intersect the belt sphere along a Reeb chord $\zeta$ connecting the bottom-most strand of $\Lambda_0$ to the unique strand of $\Lambda$. See the left side of \cref{fig:NR2slide}
 
Let $2k$ denote the number of strands of $\Lambda_0$ passing through the $1$-handle, so that $\Lambda_+$ is obtained by $2k$-many handle slides of $\Lambda_0$ up across the $(-1)$-surgery along $\Lambda$. Next, let $B$ be a Darboux neighborhood of the Reeb chord $\zeta$ which has small $x$-thickness but very large $y$-thickness; here we refer to the standard coordinates $dz - y\, dx$. Denote $M := S^3 - B$. Appealing to the local model for a Legendrian handle slide (see \cref{fig:slidemodel}), we make the following observations about $\Lambda_+$:
\begin{enumerate}
    \item Letting $\Lambda_*^M := \Lambda_* \cap M$, we have $\Lambda_+^M = \Lambda_0^M \cup \Lambda^M(2k)$, where $\Lambda^M(2k)$ denotes a contact-framed $2k$-copy of $\Lambda^M$, and 
    \item likewise letting $\Lambda_+^B := \Lambda_+ \cap B$, we have $\Lambda_+^B = C_R(2k)\cup C_L(2k)$ where $C_R,C_L\subset B$ are right (resp. left) cusp arcs connecting the left (resp. right) endpoints of $\Lambda_0\cap \partial B$ and $\Lambda \cap \partial B$.
\end{enumerate}
See the right side of \cref{fig:NR2slide}.

We induce a normal ruling of $\Lambda_+$ separately in $B$ and in $M$, and glue the rulings together to complete the proof. To this end we make use of work of Ng and Rutherford \cite{ng2013satsrulings} on normal rulings of satellites. Broadly speaking, a special case of their Theorem 3.6 (with $L^{\tau} = \emptyset$) implies that normal rulings of patterns in $J^1(S^1)$ induce normal rulings of satellites. Let $P\subset J^1(S^1)$ denote the pattern link of $2k$ parallel strands, and let $\rho_P$ be the normal ruling --- now thought of as a fixed-point free involution on the $2k$-many strands --- agreeing with the ruling of the $\Lambda_0$-strands passing through the $1$-handle.

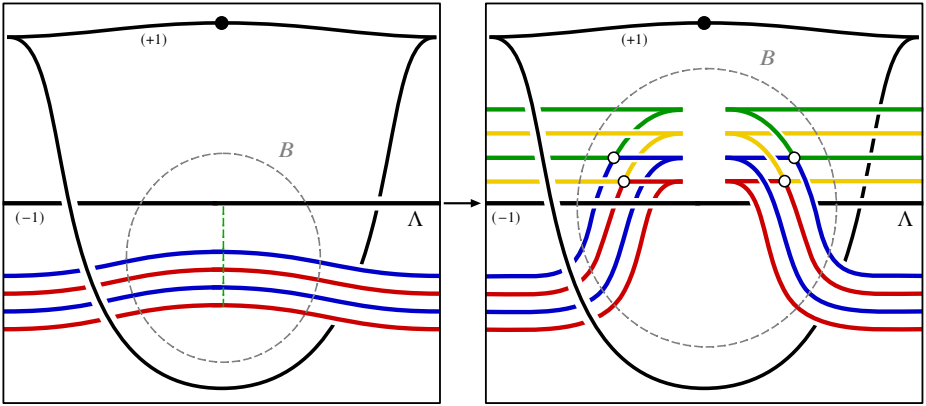
\begin{figure}[ht]
    
	\begin{overpic}[scale=.346]{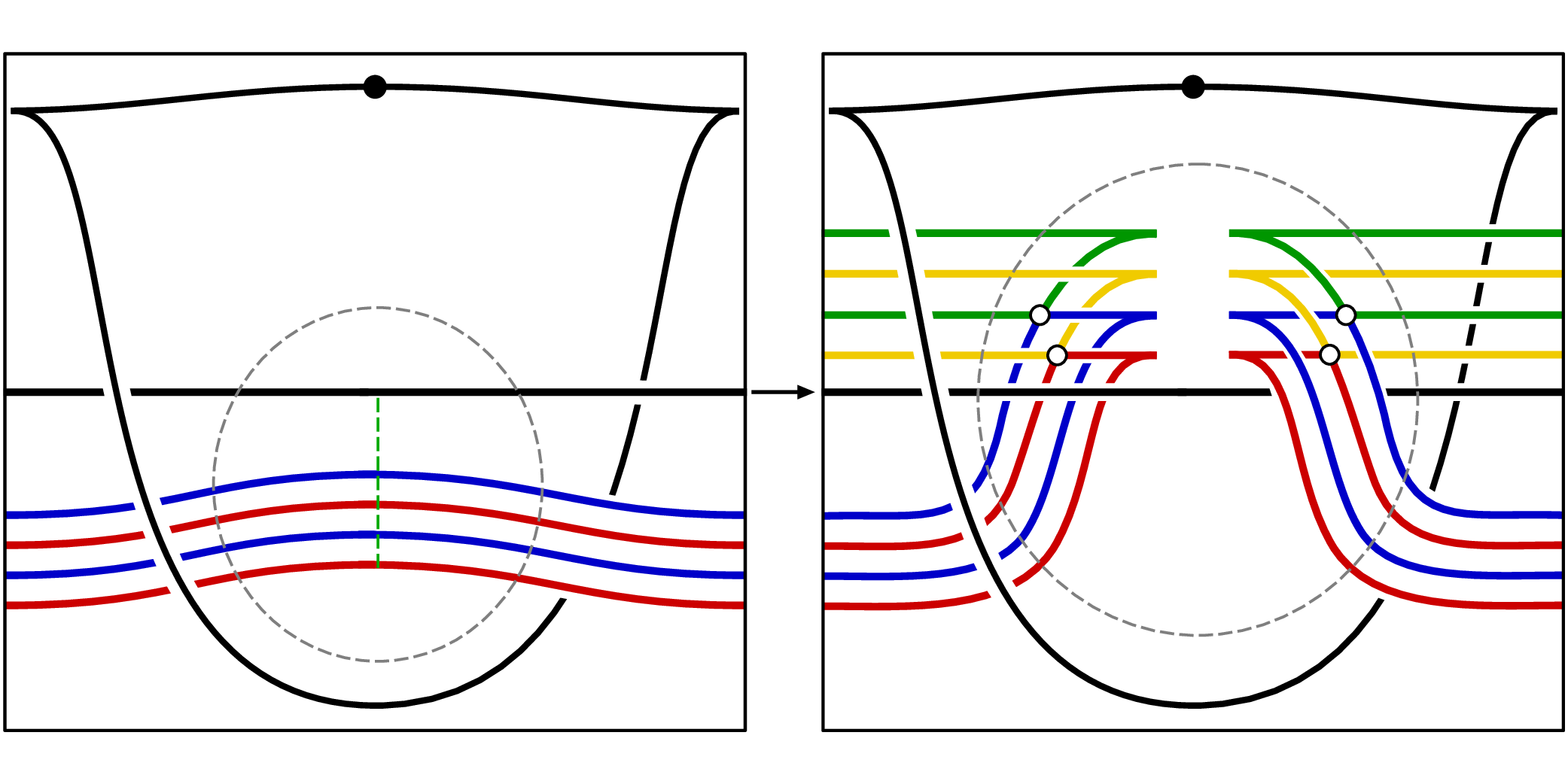}   
        
        \put(15,42.25){\tiny $(+1)$}
        \put(1.5,23){\tiny $(-1)$}
        \put(44,22.5){\footnotesize $\Lambda$}
        \put(30,30){\footnotesize \textcolor{gray}{$B$}}

         \put(67,42.25){\tiny $(+1)$}
        \put(53,23){\tiny $(-1)$}
        \put(97,22.5){\footnotesize $\Lambda$}
        \put(82,40){\footnotesize \textcolor{gray}{$B$}}

	\end{overpic}
	\caption{Extending a normal ruling of $2k=4$ $\Lambda_0$-strands passing through a $1$-handle via a slide across the geometrically canceling $2$-handle.}
	\label{fig:NR2slide}
\end{figure}

By \cite[Theorem 3.6]{ng2013satsrulings}, the $2k$-copy $\Lambda(2k) = \Sigma(\Lambda, P)$ inherits a normal ruling from $\rho_P$. Removing $B$, we view $\Lambda^M(2k)$ as being equipped with a ``relative'' normal ruling. As this ruling involves no switched crossings between strands of $\Lambda^M(2k)$ and strands of $\Lambda_0^M$, the tangle $\Lambda_+^M = \Lambda_0^M \cup \Lambda^M(2k)$ is relatively normally ruled. Inside $B$, we similarly apply \cite[Theorem 3.6]{ng2013satsrulings} to the cusped arcs $C_R\cup C_L$ to obtain a relative normal ruling of $C_R(2k) \cup C_L(2k) = \Sigma(C_R \cup C_L, P)$ inside $B$. By Ng and Rurtherford's cusp lemma \cite[Lemma 3.4]{ng2013satsrulings}, the \textit{thin} part of this ruling --- i.e., the local involution of the satellite corresponding to a point $p\in C_R\cup C_L$ in the companion tangle --- agrees above and below the cusps. The thin part of the ruling therefore agrees with the ruling of $\Lambda_0^M$ near $\partial B$ by definition of $\rho_P$, and agrees with the thin part of the ruling of $\Lambda^M(2k)$ near $\partial B$ as the latter is a $P$-satellite. Therefore, we may coherently glue the rulings together to obtain, after canceling the Weinstein handles, a normal ruling of 
\[
\Lambda_+ = \Lambda_+^M \cup \Lambda_+^B =  [\Lambda_0^M \cup \Lambda^M(2k)]\, \cup\, [C_R(2k) \cup C_L(2k)] \subset S^3. 
\]
The right side of \cref{fig:NR2slide} gives an example when $2k=4$. 
\end{proof}
\section{Regularly slice vs. decomposably slice vs. strongly decomposably slice}\label{sec:nec}

We close this article by probing the lower rungs of symplectic slice-ribbon (see \cref{fig:SRgradations}) with examples and additional questions. \cref{subsec:regvsdec} considers regularity vs. decomposability, while \cref{subsec:decvsstrong} turns to the question of distinguishing the stronger notion of decomposability from its non-commutable counterpart. 

\subsection{Regularly slice vs. decomposably slice}\label{subsec:regvsdec}

We begin with a discussion of \cref{conj:slice}, framed around a specific Legendrian knot that arose in attempts to find a counterexample. While the knot fails to disprove \cref{conj:slice} --- it is, in fact, decomposably slice --- we hope its construction illuminates the maneuvering necessary to conclude decomposability from regularity, or, if the conclusion of \cref{conj:slice} is false, inspires an actual counterexample.  

\begin{figure}[ht]
\vskip-0.3cm
	\begin{overpic}[scale=.25]{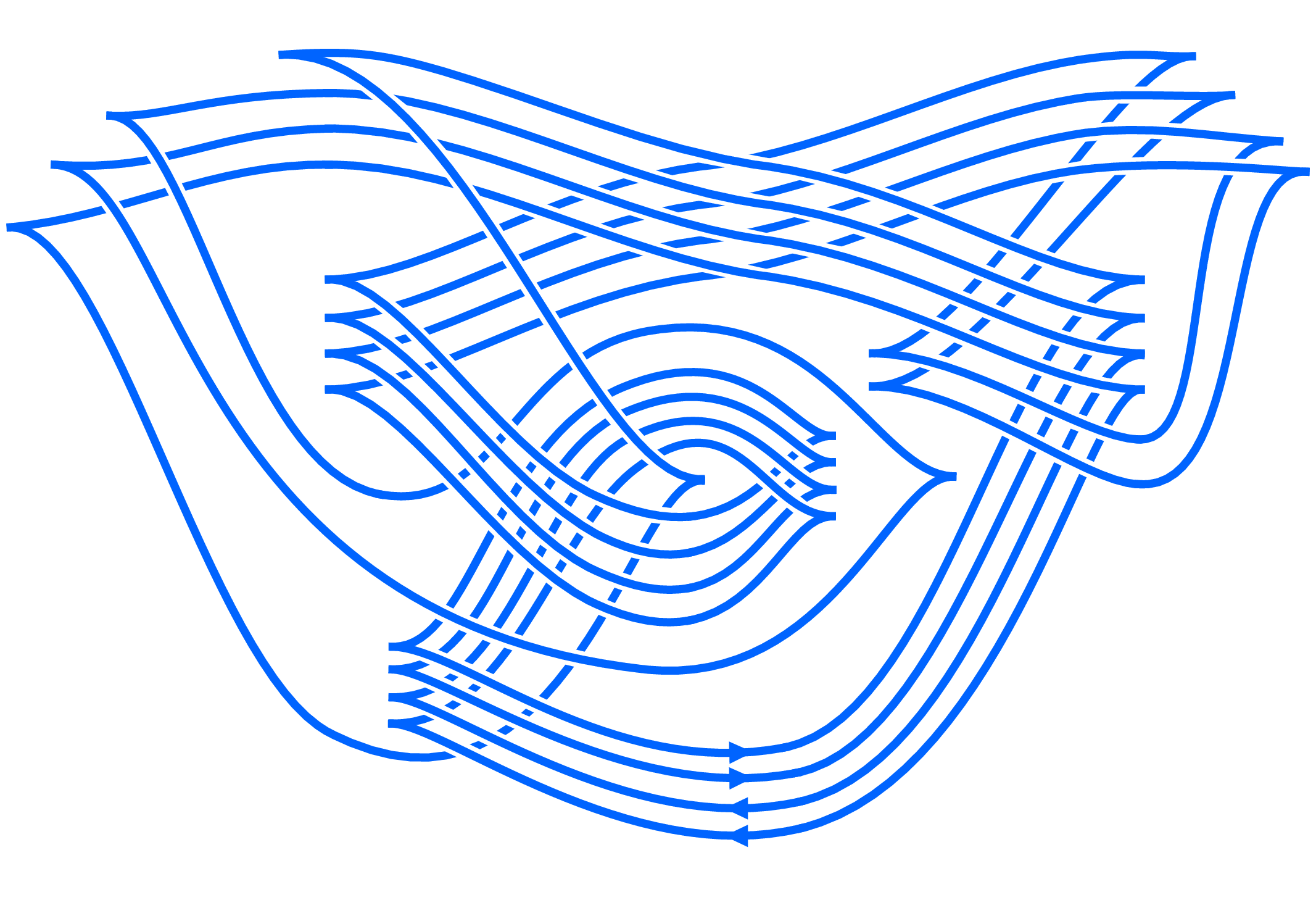}   
    \put(13,13){ \textcolor{lightblue}{$\Lambda$}}
	\end{overpic}
 \vskip-0.3cm
	\caption{A non-obviously decomposably slice knot.}
	\label{fig:candidate}
\end{figure}

The knot $\Lambda$ in question is given in \cref{fig:candidate}. We claim that $\Lambda$ is decomposably slice, but not obviously so. By ``not obviously so'' we mean that there is a natural sequence of two nested orientable pinch moves that may be performed, for instance, in the region indicated by the dashed box in the top left panel of \cref{fig:candidatetry}, but that these pinch moves result in a non-fillable link; one of the components is stabilized, and fillable Legendrians have maximum Thurston-Bennequin invariant in their smooth isotopy class. Some time spent with the diagram will convince the reader that, as presented, there are no other pinch moves which induce a filling.

\begin{figure}[hbt]
    
	\begin{overpic}[scale=.2]{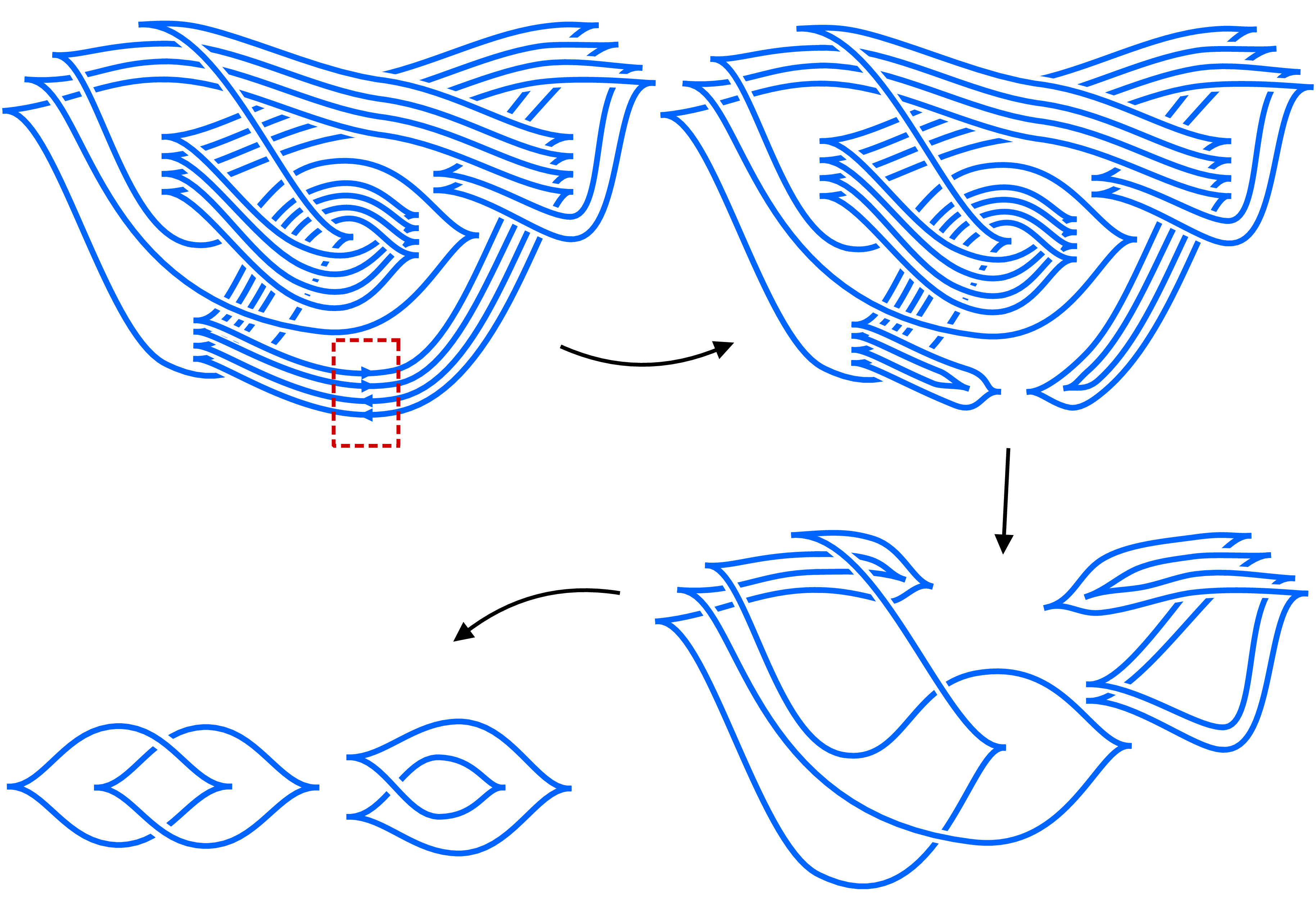}   
    
	\end{overpic}
	\caption{A natural sequence of pinch moves that fails to yield a slice disk.}
	\label{fig:candidatetry}
\end{figure}

Therefore, to conclude decomposable sliceness, it is necessary to perform a Legendrian isotopy to a different diagram that does admit pinch moves descending to a slice disk. Rather than do so directly, we first present the knot in \cref{fig:candidate} as the unknot in a Weinstein structure homotopic to the symplectization as in \cite[Theorem 1.10]{conway2017symplectic}.

Consider the left side of \cref{fig:candidate2}, which depicts a Weinstein handlebody diagram with a max-tb unknot linked with the attaching sphere of the $2$-handle. A sequence of Legendrian isotopies puts the Weinstein structure in geometrically canceling position as on the right. After sliding the four blue strands across the $(-1)$ surgery and then erasing the Weinstein handles, one obtains the front projection of $\Lambda$ in \cref{fig:candidate}. As we began with a max-tb unknot in the boundary of a Weinstein handlebody diagram homotopic to the symplectization, we conclude that $\Lambda$ is regularly slice.

\begin{figure}[ht]
    
	\begin{overpic}[scale=.347]{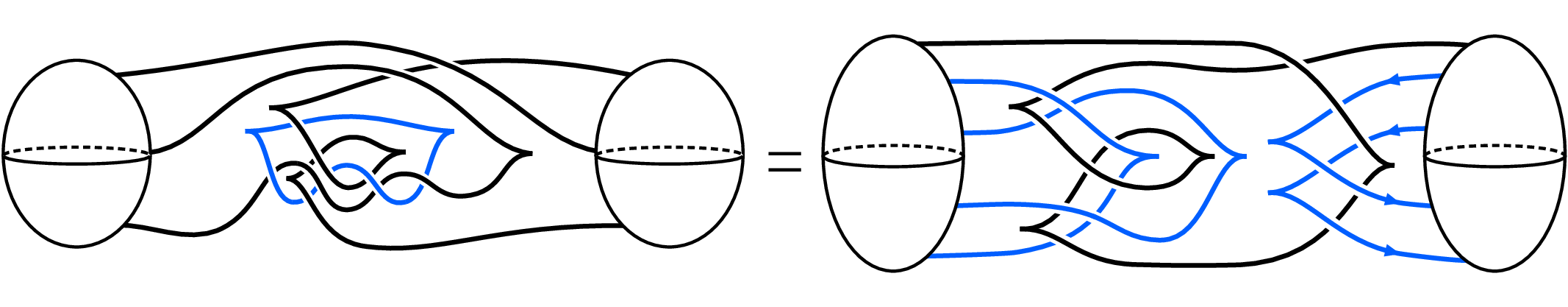}   
        \put(10,16.5){\tiny $(-1)$}
        \put(62,17.75){\tiny $(-1)$}
	\end{overpic}
	\caption{Alternate presentations of the knot in \cref{fig:candidate}.}
	\label{fig:candidate2}
\end{figure}

Next, observe that the pinch moves performed in \cref{fig:candidatetry} correspond to the pinch moves suggested by the $1$-handle on the right side of \cref{fig:candidate2}. We will see that this presentation gives a warning that the pinch moves will fail to descend to a slice disk, thanks to the unique normal ruling of the unknot, drawn in \cref{fig:candidateruling}. 

\begin{figure}[ht]
    
	\begin{overpic}[scale=.347]{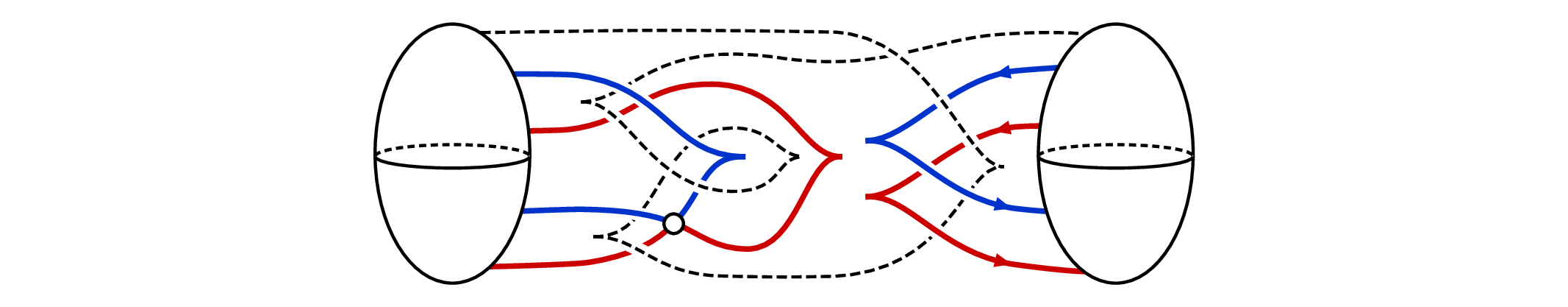}   
        \put(35,18.5){\tiny $(-1)$}
	\end{overpic}
	\caption{The normal ruling of the unknot.}
	\label{fig:candidateruling}
\end{figure}

Observe that pinching companion paths of a normal ruling produces a link that also admits a normal ruling --- this is the observation ``decomposable cobordisms induce canonical normal rulings'' in the opposite symplectization direction --- while there is no reason to expect a link produced by pinching non-paired paths to have a normal ruling. We reason that to conclude decomposability, it is preferable, if possible, to isotope the unknot in the geometrically canceling diagram so that companion paths pass through the $1$-handle in a ``companion pinchable'' order. A local isotopy that reorders the heights of companion paths along the $1$-handle near a Reeb chord is not possible; this is always the case, as discussed in \cref{fig:strongdec}.

The point of our example is that such an isotopy is not possible even semi-locally, in the following sense. To permute the order of companion paths as they pass through the $1$-handle, one must push a crossing of the unknot across the belt sphere. In our case, there are three crossings, but the region between each crossing and the belt sphere is obstructed by a cusp of the $2$-handle attaching sphere; see \cref{fig:regions}. 

\begin{wrapfigure}{r}{0.3\textwidth}
  \begin{overpic}[scale=.2]{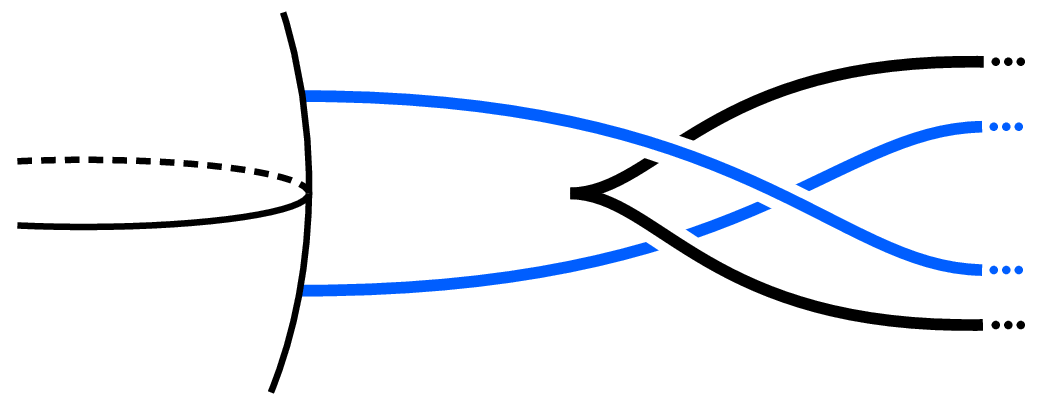}   
        \put(70,40.5){\footnotesize $(-1)$}
	\end{overpic}
  \caption{}
  \label{fig:regions}
\end{wrapfigure} 

There is no Legendrian isotopy supported in a neighborhood of such a region which pushes the crossing through the belt sphere (while preserving geometric intersection of the Weinstein handles). Consequently, the desired Legendrian isotopy, should it exist, is more global in nature. One successful isotopy is given in \cref{fig:candidatenew}.

The companion paths of the normal ruling of the unknot in the lower left panel of \cref{fig:candidatenew} now pass through the $1$-handle in such a way that the nested pinch moves are performed on matching colors. Performing these pinch moves yields a three-component max-tb unlink. The normal ruling is drawn in the top left of \cref{fig:candidatefinal}, along with the subsequent pinch moves and their nested surgery arcs. Sliding the surgery arcs and canceling the handles produces a (nested) Legendrian surgery presentation of $\Lambda$ starting with a max-tb unlink. We conclude that $\Lambda$ is decomposably slice. 

\begin{figure}[htb]
    
	\begin{overpic}[scale=.3]{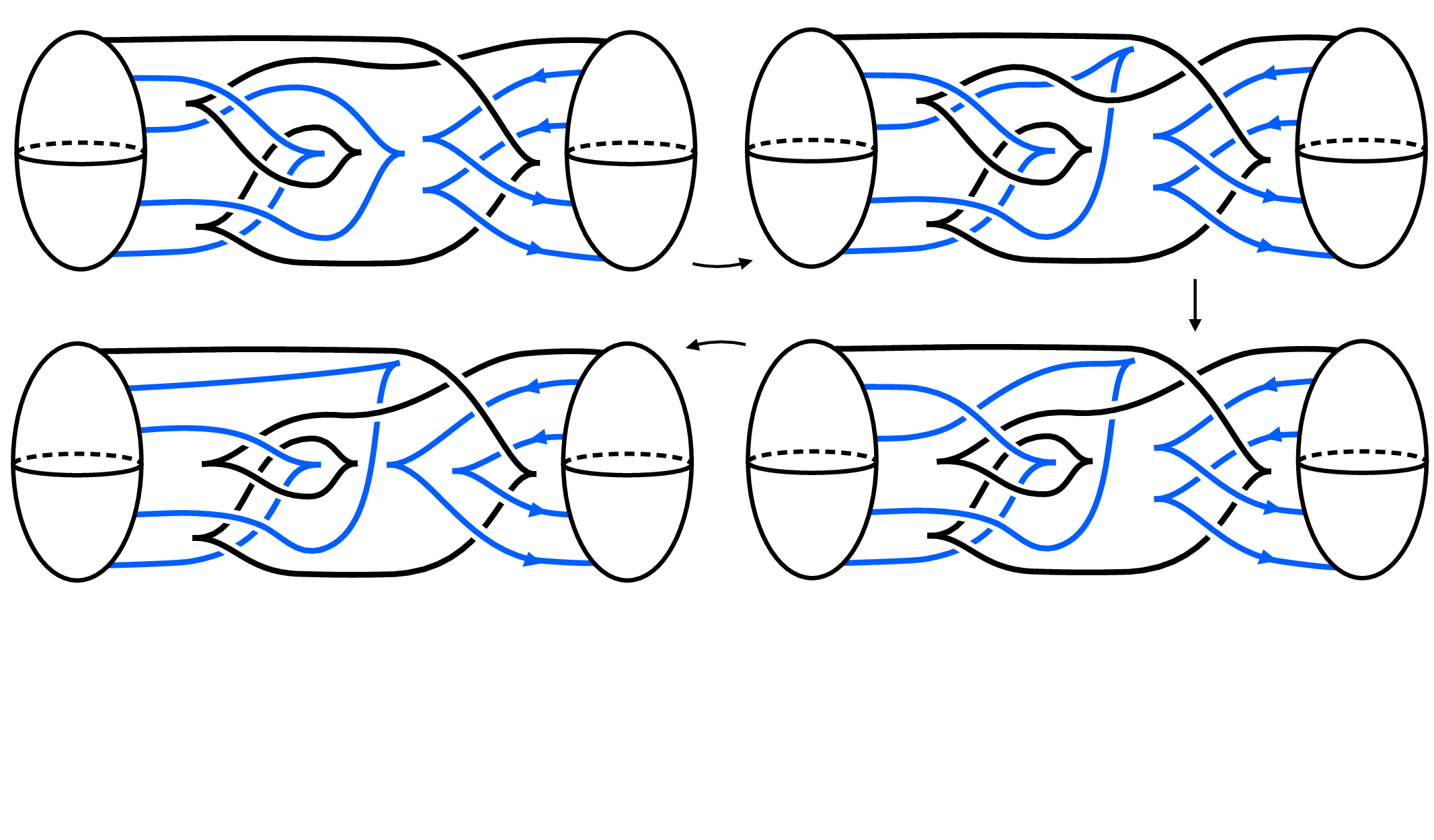}   
        \put(10,57.5){\tiny $(-1)$}
        \put(62,57.75){\tiny $(-1)$}
        \put(10,35.5){\tiny $(-1)$}
        \put(62,35.75){\tiny $(-1)$}
	\end{overpic}
    \vskip-1.75cm
	\caption{Isotoping the unknot in the geometrically canceling diagram.}
	\label{fig:candidatenew}
\end{figure}

\begin{figure}[ht]
    
	\begin{overpic}[scale=.345]{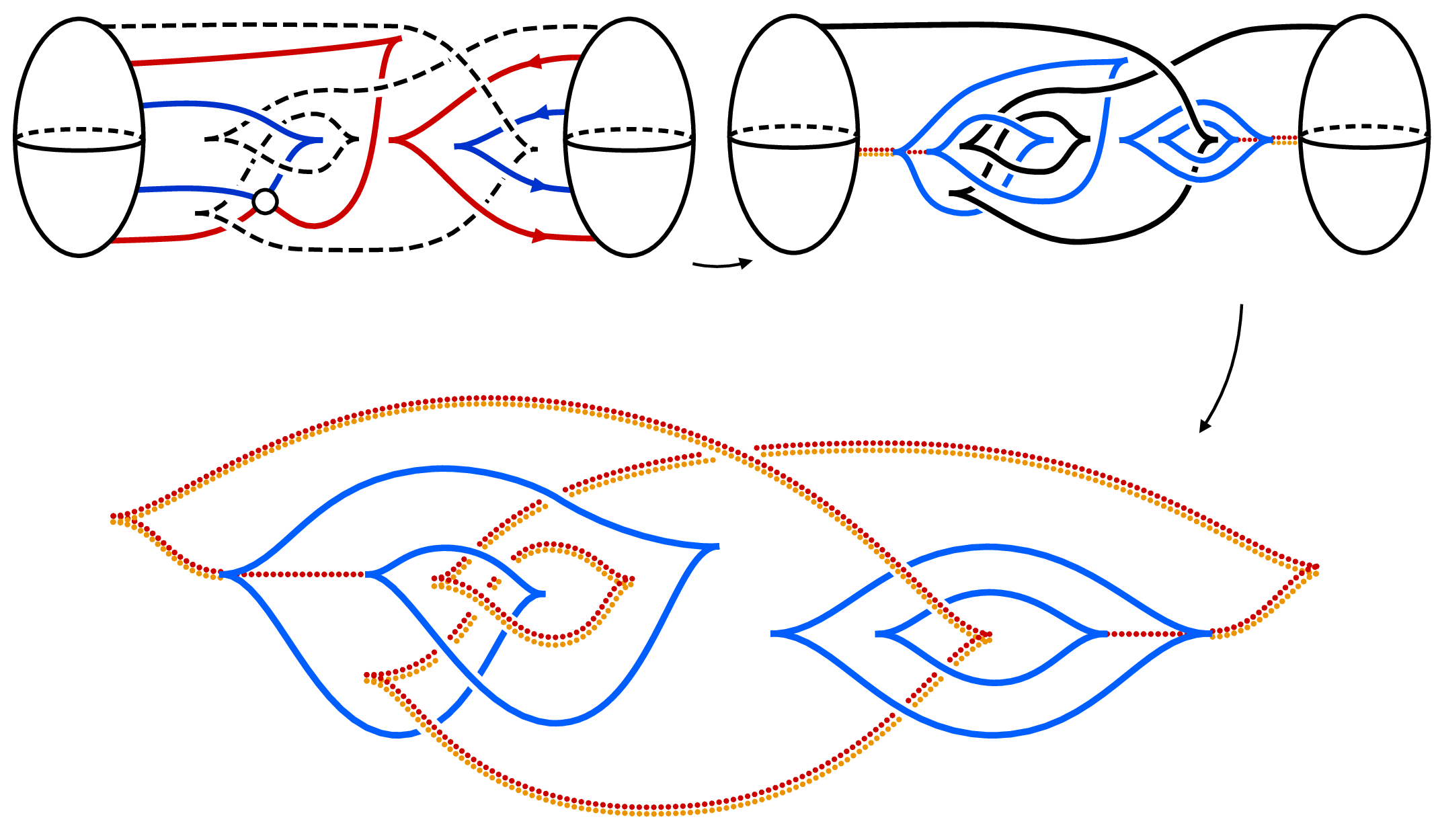}   
        \put(12,58){\tiny $(-1)$}
        \put(64,58){\tiny $(-1)$}
	\end{overpic}
	\caption{A pinch sequence on the unknot that produces a max-tb unlink. Sliding the nested surgery arcs to cancel the Weinstein handles results in the nested surgery presentation of a decomposable slice disk for $\Lambda$. The double-dotted surgery arcs represent non-commutable nested surgeries, where the orange is performed first.}
	\label{fig:candidatefinal}
\end{figure}

With this example in mind, we summarize with two questions that form the basis for a strategy toward solving \cref{conj:slice}. 

\begin{question}\label{q:f1}
Perform a pinch move on companion paths in the canonical normal ruling of a decomposably slice link. Is the resulting link still decomposably slice?   
\end{question}

\noindent Note that Hayden \cite{hayden2019crosssections} showed that one can pinch a max-tb unknot down to a non-trivial decomposably slice link, so it is not the case that companion pinch moves are always topologically simplifying (as they are in \cref{fig:candidatefinal}). However, it seems plausible that decomposable sliceness is preserved. 

\begin{question}\label{q:f2}
Consider a max-tb unknot in the boundary of a geometrically canceling Weinstein handlebody diagram. Can one always isotope the unknot in the complement of the $2$-handle attaching spheres so that companion paths of its unique normal ruling pass through the $1$-handles in a pinchable order? 
\end{question}

\noindent An affirmative answer to both \cref{q:f1} and \cref{q:f2} would provide a solution to \cref{conj:slice}.

\subsection{Decomposably slice vs. strongly decomposably slice}\label{subsec:decvsstrong}

Although \cref{fig:candidatefinal} shows that the knot $\Lambda$ in \cref{fig:candidate} is decomposably slice, the two pinch moves are nested. That is, it is necessary to perform one before the other, and locally the order cannot be commuted. In other words, we have shown that $\Lambda$ is decomposably slice, but not necessarily strongly decomposably slice. Appealing to same idea, a simpler example is drawn in \cref{fig:finalex}. The nested pinch moves given by the red box on the right descend to a slice disk.

\begin{figure}[ht]
    
	\begin{overpic}[scale=.347]{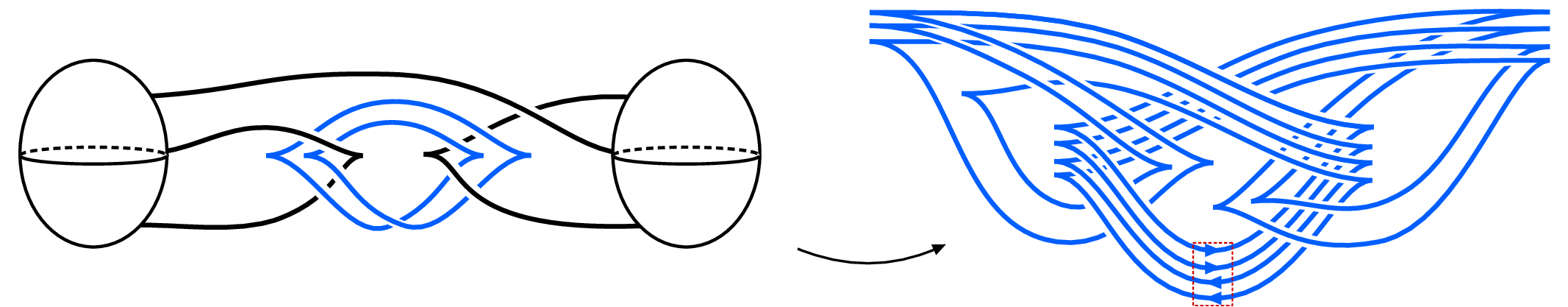}   
        \put(12,15){\tiny $(-1)$}
        \put(92,3){\small \textcolor{lightblue}{$\Lambda'$}}
        
	\end{overpic}
	\caption{A decomposably slice knot. Is it strongly decomposably slice?}
	\label{fig:finalex}
\end{figure}

In fact, the construction of $\Lambda$ and $\Lambda'$ has as underlying mechanism the following observation. Let $\Delta_2=\sigma_1$ denote the right-handed half-twist on two strands, viewed as a pattern knot in $J^1(S^1)$. Then if $U$ is the max-tb unknot, the satellite $\Sigma(U, \Delta_2^3)$ is Legendrian isotopic to $U$. Consequently, if $\Lambda$ is (regularly) slice, then $\Sigma(\Lambda, \Delta_2^3)$ is (regularly) slice, non-parenthetically following from \cite{cornwell2016concordance}, while \cref{thm:main_sat} gives preservation of regularity. See \cref{fig:satellitestrong} for the case of $\overline{9_{46}}$.

\begin{figure}[ht]
	\begin{overpic}[scale=.3465]{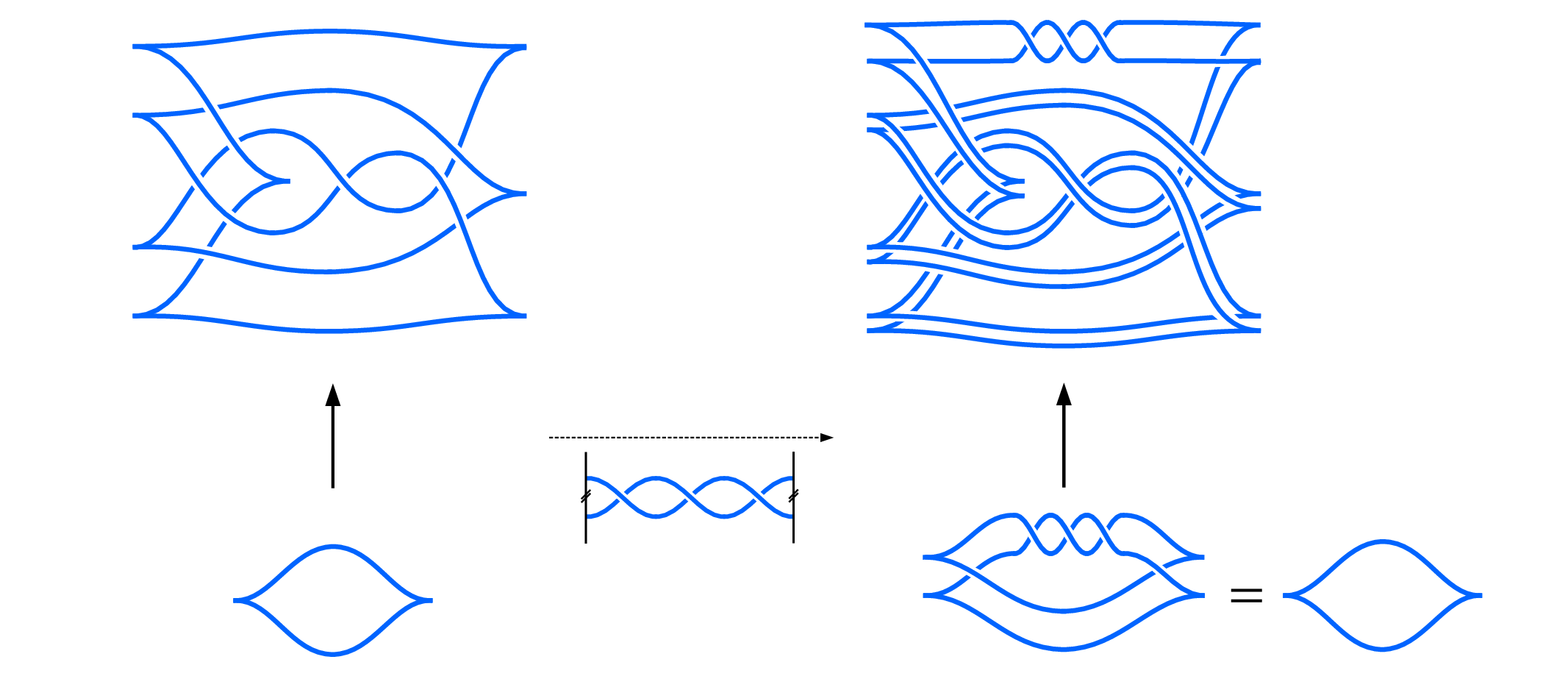}   
        \put(40.25,17){\footnotesize satellite}
        \put(80.25,26){\footnotesize \textcolor{lightblue}{$\Lambda''$}}
	\end{overpic}
	\caption{Is the $\Delta_2^3$-satellite of $\overline{9_{46}}$ strongly decomposably slice?}
	\label{fig:satellitestrong}
\end{figure}

In fact, decomposability is also preserved:  

\begin{proposition}\label{prop:slicesats}
If $\Lambda$ is (strongly) decomposably slice, then $\Sigma(\Lambda, \Delta_2^3)$ is decomposably slice.
\end{proposition}

This result follows from the satellite operation of \cite[Theorem 1.6]{guadagni2022satellites} by choosing the patterns $\Pi_- = \Pi_+ = \Delta_2^2$ in their statement, coupled with the above observation that $\Sigma(U, \Delta_2^3) = U$. However, strong decomposability is not addressed in their work, and the point of \cref{prop:slicesats} in our discussion is that strongness is not obviously preserved. For this reason, we provide a direct and independent proof under the assumption of strong decomposability to clarify why strongness need not persist under a $\Delta_2^3$-satellite. 

\begin{proof}[Proof.]
Assume $\Lambda$ is strongly decomposably slice. Since the Euler characteristic of a disk is $1$, $\Lambda$ admits a presentation as $\mathrm{Surg}(\tilde{U}, G)$, where $\tilde{U} = U_0 \cup U_1 \cup \cdots \cup U_n$ is a max-tb unlink with $n+1$ components and $G = \{\gamma_1, \dots, \gamma_n\}$ is a set of $n$ surgery arcs for $\tilde{U}$. For brevity we refer to the minimum such $n$ as the \textit{surgery number} of the knot. We induct on $n$. The base case corresponds to when $\Lambda = U$ is the max-tb unknot; the claim then follows from the observation $\Sigma(U, \Delta_2^3) = U$ from above. 

Now we perform the induction. Assume $\Lambda_n$ has surgery number $n$, and that the claim holds for knots of surgery number $\leq n-1$. To the surgery presentation of $\Lambda_n$ as above we associate a graph $\Gamma(\Lambda_n)$, where the $n+1$ vertices correspond to the components of $\tilde{U}$, 
and two vertices are connected by an edge if the corresponding link components are connected by a surgery arc in $G$. Since $\Lambda_n$ is connected, Euler characteristic considerations again imply that $(V,E)$ is a tree. Consider a surgery arc $\gamma$ corresponding to an edge adjacent to a leaf of the tree. Then $\gamma$ corresponds to a contractible Reeb chord $\zeta$ on $\Lambda_n$ along which the pinch moves produces a link $\Lambda_{n-1} \cup U_0$, where $U_0$ is a (possibly relabeled) max-tb unknot component and $\Lambda_{n-1}$ is a strongly decomposably slice link with surgery number $\leq n-1$. Note, however, that $U_0$ may be linked with $\Lambda_{n-1}$; see \cite[Figure 8]{hayden2019crosssections} for an example. 

\vspace{2mm}
\noindent \textit{Case I: $U_0$ is unlinked with $\Lambda_{n-1}$.}
\vspace{2mm}

Consider a neighborhood of the Reeb chord $\zeta$. Assume without loss of generality that $\Lambda_{n-1}$ corresponds to the component on the left after the pinch move --- i.e., the component with the right cusp --- while $U_0$ corresponds to the component on the right. When we construct the satellite $\Sigma(\Lambda_n, \Delta_2^3)$, we are free to position the pattern anywhere along the companion knot; we deliberately do so near the top of the Reeb chord as in \cref{fig:satpinch}. 

\begin{figure}[ht]
	\begin{overpic}[scale=.34]{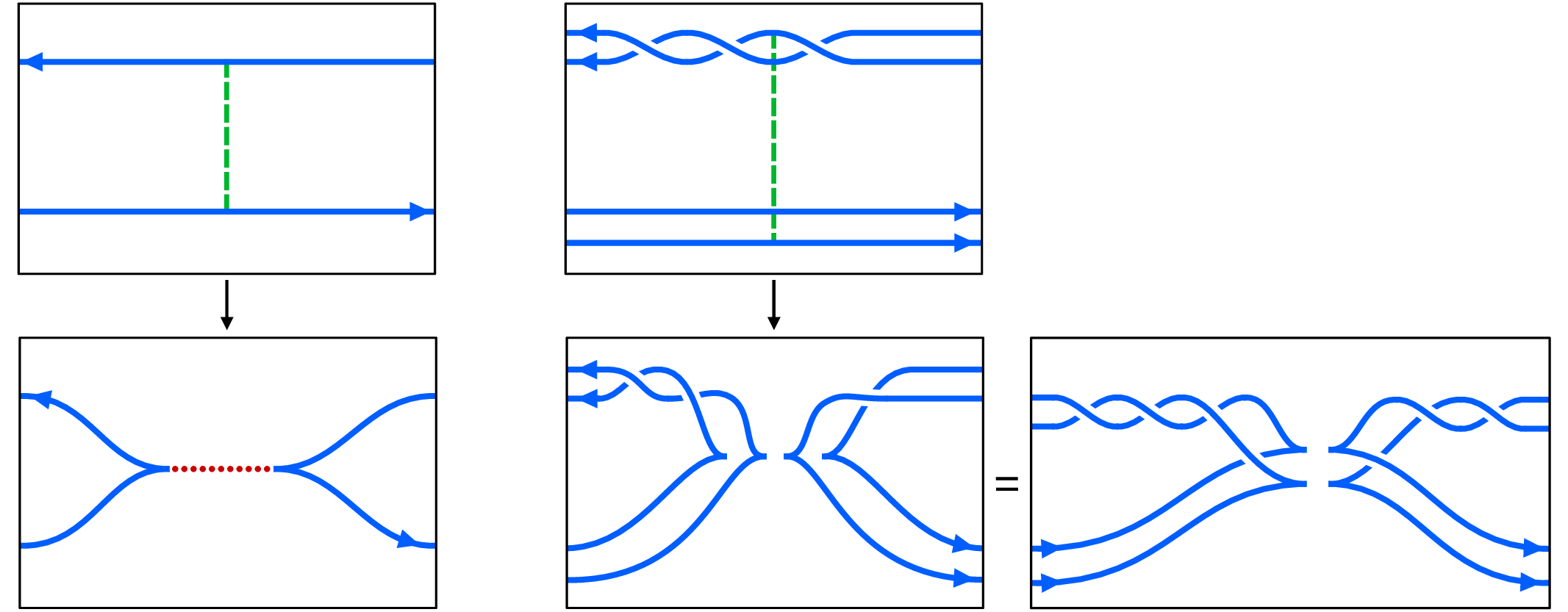}   
       \put(14,10){\footnotesize \textcolor{darkred}{$\gamma$}}
       \put(6,13){\footnotesize \textcolor{lightblue}{$\Lambda_{n-1}$}}
       \put(20,13){\footnotesize \textcolor{lightblue}{$U_0$}}
       \put(67,15.1){\footnotesize \textcolor{lightblue}{$\Sigma(\Lambda_{n-1}, \Delta_2^3)$}}
       \put(88,14.95){\footnotesize \textcolor{lightblue}{$U_0 \cup U_0'$}}

       \put(20,30){\footnotesize \textcolor{lightblue}{$\Lambda_n$}}
       \put(50.75,30){\footnotesize \textcolor{lightblue}{$\Sigma(\Lambda_{n}, \Delta_2^3)$}}
       
	\end{overpic}
	\caption{The proof of \cref{prop:slicesats}.}
	\label{fig:satpinch}
\end{figure}

With the pattern positioned accordingly, perform the two nested pinch moves on the satellite along the Reeb chord indicated in the figure. Observe that the leftmost part of the diagram, after an isotopy, is $\Sigma(\Lambda_{n-1}, \Delta_2^3)$. Next, observe that the rightmost part of the diagram is a $2$-cable of $U_0$ with framing given by a right-handed full-twist with respect to the contact framing, i.e., $\Sigma(U_0, \Delta_2)$. Since $\mathrm{tb}(U_0) = -1$, $\Sigma(U_0, \Delta_2)$ is a smooth $0$-framed $2$-cable of the unknot, which is a $2$-component unlink $U_0 \cup U_0'$. Both components are max-tb, and under the Case I assumption, are unlinked with $\Lambda_{n-1}$. Applying the inductive hypothesis to $\Lambda_{n-1}$ then gives the result. 

\vspace{2mm}
\noindent \textit{Case II: $U_0$ is linked with $\Lambda_{n-1}$.}
\vspace{2mm}

In this case, we begin by applying the same procedure in Case I. Namely, we satellite and perform two pinch moves using the pattern to end up with a link $\Sigma(\Lambda_{n-1}, \Delta_2^3)\cup U_0 \cup U_0'$. While $U_0 \cup U_0'$ is still a $2$-component max-tb unlink, and $\Sigma(\Lambda_{n-1}, \Delta_2^3)$ is individually decomposably slice by the inductive hypothesis, nontrivial linking prevents us for now from concluding that the whole link is decomposably slice. 

However, we press on, applying the same procedure to $\Sigma(\Lambda_{n-1}, \Delta_2^3)$. With the same maneuver in Case I we can identify a leaf of the graph $\Gamma(\Lambda_{n-1})$, position the pattern $\Delta_2^3$ near the top of the corresponding Reeb chord, and perform two more pinch moves to yield $\Sigma(\Lambda_{n-2}, \Delta_2^3)\cup U_1 \cup U_1' \cup U_0 \cup U_0'$, where $U_1 \cup U_1'$ is a $2$-component max-tb unlink and $\Lambda_{n-2}$ is strongly decomposably slice with surgery number $\leq n-2$. If there is still linking, we repeat. Eventually, since the base link $\tilde{U}$ in $\Lambda = \mathrm{Surg}(\Lambda, G)$ is a max-tb unlink, this process must terminate in a link of unlinked components that are all decomposably slice.
\end{proof}

Though the examples from this section ($\Lambda, \Lambda'$, $\Lambda''$, and more generally $\Delta_2^3$-satellites of strongly decomposably slice knots) fail to break symplectic slice-ribbon between regularity and decomposability, \cref{prop:slicesats} suggests they can still probe the divide between the lowest symplectic rungs of \cref{fig:SRgradations}:

\begin{question}\label{q:satstrong}
Let $\Lambda\neq U$ be a (strongly) decomposably slice knot. Is $\Sigma(\Lambda, \Delta_2^3)$ strongly decomposably slice?    
\end{question}

\bibliography{references}
\bibliographystyle{amsalpha}

\end{document}